\RequirePackage[l2tabu, orthodox]{nag}

\documentclass[12pt]{amsart}
\usepackage{fullpage,url,amssymb,enumerate,colonequals}
\usepackage[all]{xy} 
\usepackage{comment}
\usepackage{graphicx}


\usepackage[OT2,T1]{fontenc}

\usepackage{color}

\def\Adm{\mathcal{A}dm}

\def\cA{\mathcal{A}}
\def\cB{\mathcal{B}}
\def\bC{\mathbb{C}}
\def\cC{\mathcal{C}}

\def\cH{\mathcal{H}}

\def\cM{\mathcal{M}}

\def\cO{\mathcal{O}}
\def\bP{\mathbb{P}}

\def\bQ{\mathbb{Q}}

\def\cT{\mathcal{T}}

\def\cY{\mathcal{Y}}
\def\bZ{\mathbb{Z}}

\def\barM{\overline{\cM}}
\def\barH{\overline{\cH}}

\def\wt{\widetilde}

\DeclareMathOperator{\Aut}{Aut}

\DeclareMathOperator{\ch}{ch}
\DeclareMathOperator{\cores}{cores}

\DeclareMathOperator{\id}{id}

\DeclareMathOperator{\im}{im}

\DeclareMathOperator{\lcm}{lcm}

\DeclareMathOperator{\ord}{ord}

\DeclareMathOperator{\pr}{pr}

\DeclareMathOperator{\Qmod}{Qmod}
\DeclareMathOperator{\red}{red}
\DeclareMathOperator{\res}{res}

\DeclareMathOperator{\Spec}{Spec}
\DeclareMathOperator{\Stab}{Stab}

\DeclareMathOperator{\topchern}{top}

\newtheorem{thm}{Theorem}[section]
\newtheorem{lem}[thm]{Lemma}

\newtheorem{prop}[thm]{Proposition}

\newtheorem{conjecture}{Conjecture}

\theoremstyle{definition}
\newtheorem{rem}[thm]{Remark}
\newtheorem{ex}[thm]{Example}
\newtheorem{defn}[thm]{Definition}

\makeatletter
\g@addto@macro\bfseries{\boldmath} 
\makeatother

\usepackage{microtype}

\usepackage[
	backref,
	pdfauthor={Carl Lian}, 
	pdftitle={The H-tautological ring},
]{hyperref}

\begin{document}

\title{The $\cH$-tautological ring}

\author{Carl Lian}
\address{Institut f\"{u}r Mathematik, Humboldt-Universit\"{a}t zu Berlin, Berlin, 12489, Germany}
\email{liancarl@hu-berlin.de}
\urladdr{\url{http://sites.google.com/view/carllian/}}

\date{\today}

\begin{abstract}
We extend the theory of tautological classes on moduli spaces of stable curves to the more general setting of moduli spaces of admissible Galois covers of curves, introducing the so-called $\mathcal{H}$-tautological ring. The main new feature is the existence of restriction-corestriction morphisms remembering intermediate quotients of Galois covers, which are a rich source of new classes. In particular, our new framework includes classes of Harris-Mumford admissible covers on moduli spaces of curves, which are known in some (and speculatively many more) examples to lie outside the usual tautological ring. We give additive generators for the $\mathcal{H}$-tautological ring and show that their intersections may be algorithmically computed, building on work of Schmitt-van Zelm. As an application, we give a method for computing integrals of Harris-Mumford loci against tautological classes of complementary dimension, recovering and giving a mild generalization of a recent quasi-modularity result of the author for covers of elliptic curves.
\end{abstract}

\maketitle


\tableofcontents

\section{Introduction}\label{intro}

\subsection{Hurwitz loci on moduli spaces of curves}

The cycle theory of the moduli space of stable curves $\barM_{g,n}$ is central to understanding its geometry. Moduli spaces of branched covers of curves provide a particularly rich source of algebraic cycles on $\barM_{g,n}$. Let $\cH_{g/h,d}$ be a moduli space of branched covers (a \textit{Hurwitz space}), parametrizing finite degree $d$ morphisms $f:X\to Y$, where $X,Y$ are smooth curves of genus $g,h$, respectively. Suppose further that $f$ is simply branched over
\begin{equation*}
b=(2g-2)-d(2h-2)
\end{equation*}
points. We then get morphisms
\begin{equation*}
\xymatrix{
\cH_{g/h,d} \ar[r]^{\phi_{g/h,d}} \ar[d]^{\delta_{g/h,d}} & \cM_{g,b} \\
\cM_{h,b}
}
\end{equation*}
remembering the source and target, marked with the ramification and branch points, respectively.

Via the Harris-Mumford space of admissible covers, one can compactify $\cH_{g/h,d}$ and extend the morphisms above:
\begin{equation*}
\xymatrix{
\barH_{g/h,d} \ar[r]^{\phi_{g/h,d}} \ar[d]^{\delta_{g/h,d}} & \barM_{g,b} \\
\barM_{h,b}
}
\end{equation*}

By to the Riemann Existence Theorem, the map $\delta_{g/h,d}$ is flat and quasi-finite of degree given combinatorially by a \textit{Hurwitz number} counting monodromy actions of the fundamental group of the target curve, punctured at the branch points, on the fibers of $f$.

On the other hand, the map $\phi_{g/h,d}$ gives an interesting cycle of homological degree $3h-3+b$ on $\barM_{g,b}$ upon pushing forward the fundamental class $1\in A_{*}(\barH_{g/h,d})$; we refer to such cycles as \textit{branched cover loci} or \textit{Hurwitz loci}. 

We note three reasons for a detailed study of such cycles:
\begin{enumerate}
\item Cycles $(\phi_{g/h,d})_{*}(1)$ are interesting from the point of view of birational geometry. For example, Eisenbud, Harris, and Mumford prove \cite{hm,harris,eh_kodaira} that $\barM_g$ is of general type for $g\ge24$ by producing effective divisors $E$ on $\barM_g$ for which $K_{\barM_g}=A+E$, for some ample divisor $A$. When $g$ is odd, they take $E$ to be the branched cover locus $(\pi_{g/0,(g+1)/2})_{*}(1)$ of curves of sub-generic gonality.

More recently, numerous examples of extremal classes on $\barM_{g,n}$ have been constructed via branched cover loci, see, for instance, \cite{cc1,cc2,cp,ct,blankers}.

\item Cycles $(\phi_{g/h,d})_{*}(1)$ satisfy interesting numerological phenomena. For example, in the case $h=1$, we have the following:
\begin{conjecture}\label{qmod_conj}\cite[Conjecture 1]{lian_qmod}
Fix an integer $g\ge2$. We have:
\begin{equation*}
\sum_{d\ge1}[(\phi_{g/1,d})_{*}(1)]q^d\in A^{g-1}(\barM_g)\otimes\Qmod,
\end{equation*}
where $\Qmod$ is the ring of quasimodular forms.
\end{conjecture}
Conjecture \cite{lian_qmod} is known to be true for $g=2,3$ \cite[Theorem 1.2]{lian_qmod}, with a number of variants, for example, allowing for marked points to be taken into account. We give a slight improvement on this result in the case $g=2$ in Theorem \ref{d-elliptic_recover_general}.

\item Cycles $(\phi_{g/h,d})_{*}(1)$ are known in some cases to be non-tautological, see, for example, \cite{gp,vanzelm}. We defer a more detailed discussion of this phenomenon to the next section.
\end{enumerate}

\subsection{The classical tautological ring}\label{intro_classical}

Both the Chow ring $A^{*}(\barM_{g,n})$ and the cohomology ring $H^{*}(\barM_{g.n})$ are extremely complicated and a full understanding of these objects seems out of reach. However, Mumford \cite{mumford} initiated a study of certain \textit{tautological classes} on $\barM_{g,n}$ that appear in many natural geometric situations and are largely computable in practice.

By definition, the tautological rings $R^{*}(\barM_{g,n})\subset A^{*}(\barM_{g,n})$ form the smallest system of $\bQ$-subalgebras containing the $\psi$ and $\kappa$ classes and closed under all pushforwards by forgetful morphisms $\pi:\barM_{g,n+1}\to\barM_{g,n}$ and boundary morphisms $\xi_{\Gamma}:\barM_{\Gamma}\to\barM_{g,n}$, indexed by stable graphs $\Gamma$ of genus $g$ with $n$ marked points.\footnote{In fact, one can make this definition without mentioning the $\psi$ and $\kappa$ classes, which will automatically be included, but we include them explicitly to highlight their fundamental importance.} Additive generators for the tautological ring may be given by \textit{decorated boundary strata}, that is, pushforwards of polynomials in $\psi$ and $\kappa$ classes by boundary morphisms, see \cite[Appendix A]{gp}, and indeed such classes may be algorithmically intersected in terms of the combinatorics of dual graphs.

Many cohomology classes on moduli spaces of curves arising in geometry turn out to be tautological. However, as first shown by Graber-Pandharipande \cite{gp}, Hurwitz loci, in general, are not. For example, as we review in \S\ref{non-taut_section}, the $(3h-1)$-dimensional locus of stable curves of genus $2h$ admitting an admissible double cover of a curve of genus $h$ is non-tautological for $h$ sufficiently large. Indeed, its pullback to the boundary stratum $\barM_{h,1}\times\barM_{h,1}\to\barM_{2h}$ is a positive multiple of the diagonal class. However, if it were tautological, it would have a tautological K\"{u}nneth decomoposition upon pullback, which is impossible because $\barM_{h,1}$ is known to have odd cohomology for $h$ sufficiently large.

This example yields two insights: firstly, one would get a strictly larger theory by adding Hurwitz loci to the calculus of tautological classes, and secondly, doing so would require diagonal classes to be added as well. We will, in fact, go further: we develop this theory in the Chow rings of Hurwitz spaces themselves.

\subsection{Galois covers}

The main drawback to the Harris-Mumford compactification of Hurwitz spaces by admissible covers is that in general, spaces of admissible covers fail to be smooth. In order to carry out intersection theory on them in a robust way, therefore, we must pass to a resolution of singularities.

We will instead work with \textit{Galois covers}. Suppose that $f:X\to Y$ is a degree $d$ \'{e}tale cover of (possibly non-proper) curves. Then, we have a principal $S_d$ bundle
\begin{equation*}
\wt{f}:\wt{X}:=(X\times_Y\cdots\times_Y X)-\Delta\to Y
\end{equation*}
given by taking the $d$-fold product over $X$ and removing all diagonals. Moreover, the data of $f$ can be recovered from the $S_d$-cover $\wt{X}\to Y$ by defining $X=\wt{X}/S_{d-1}$.

If $f:X\to Y$ is now assumed to be a possibly branched cover of smooth curves, one can carry out the same construction on the \'{e}tale locus, then pass to the unique extension to proper curves, to produce a map $\wt{f}:\wt{X}\to Y$ that is an $S_d$-cover, but no longer a principal $S_d$-bundle.

If, however, $f:X\to Y$ is an admissible cover of nodal curves, the construction breaks down. On the other hand, the inverse construction is still valid: an \textit{admissible $S_d$-cover} $\wt{f}:\wt{X}\to Y$ gives rise to an admissible cover $f:\wt{X}/S_{d-1}\to Y$.

Thus, we have a map $\nu:\wt{H}_{g/h,d}\to\barH_{g/h,d}$, where $\wt{H}_{g/h,d}$ parametrizes $S_d$-covers $\wt{f}:\wt{X}\to Y$ that, upon taking the intermediate quotient by $S_{d-1}$, gives an admissible cover of the required type. The work of Abramovich-Corti-Vistoli in the language of \textit{twisted stable maps} (where the scheme-theoretic $S_d$-cover $\wt{f}:\wt{X}\to Y$ is replaced by a principal $S_d$-bundle $\wt{f}:\wt{X}\to\cY$ over a stacky curve $\cY$) implies that $\wt{H}_{g/h,d}$ is smooth and $\nu$ is a normalization, see \cite[Proposition 4.2.2]{acv}.

Therefore, we work in the following setting, following \cite{schmittvanzelm}: if $G$ is a finite group, let $\barH_{g,G,\xi}$ be the moduli space of \textit{pointed admissible $G$-covers} with source genus $g$ and with \textit{monodromy data} given by $\xi$, see Definition \ref{admissible_galois_main}. Then, $\barH_{g,G,\xi}$ is a smooth Deligne-Mumford stack, see Theorem \ref{source_target_properties}; it admits source and target maps 
\begin{align*}
\phi:\barH_{g,G,\xi}\to\barM_{g,r}\\
\delta:\barH_{g,G,\xi}\to\barM_{g',b}
\end{align*} 
as in the Harris-Mumford setting.

\subsection{The $\cH$-tautological ring}

We develop the theory of tautological classes on spaces of admissible $G$-covers $\barH_{g,G,\xi}$ in parallel with that of the classical theory.

As in the classical setting, we have natural notions of forgetful morphisms $\pi:\barH_{g,G,\xi+\{1\}}\to\barH_{g,G,\xi}$ and boundary morphisms $\xi_{(\Gamma,G)}:\barH_{(\Gamma,G)}\to\barH_{g,G,\xi}$, the latter indexed by \textit{admissible $G$-graphs}. In addition, one can define $\psi$ and $\kappa$ classes on $\barH_{g,G,\xi}$ by pullback under $\phi$.

However, the main new feature of the theory of spaces of admissible $G$-covers is the existence of \textit{restriction morphisms}
\begin{equation*}
\res^{G}_{G_1}:\barH_{g,G,\xi}\to\barH_{g,G_1,\xi'}
\end{equation*}
restricting the $G$-action on a stable curve to that of a subgroup $G_1\subset G$, and \textit{corestriction morphisms}
\begin{equation*}
\cores^{G}_{G/G_1}:\barH_{g,G,\xi}\to\barH_{g_1,G/G_1,\xi'}
\end{equation*}
corestricting the $G$-action on a stable curve $X$ to that of $G/G_1$ on the partial quotient $X/G_1$, where $G_1\subset G$ is a normal subgroup.

Combining these operations gives a more general notion of \textit{restriction-corestriction morphisms}, see Definition \ref{res_cores_defn}.

For example, restriction maps $\res^{G}_{\{1\}}$ recover source maps $\phi$, and corestriction maps $\cores^{G}_{\{1\}}$ recover target maps $\delta$. The pre-composition
\begin{equation*}
\phi_{g/h,d}\circ\nu:\wt{H}_{g/h,d}\to\barM_{g,b}
\end{equation*}
of the source map on a Harris-Mumford space with its normalization by $S_d$-covers may also be regarded as the composition
\begin{equation*}
\cores^{S_{d-1}}_{\{1\}}\circ\res^{S_d}_{S_{d-1}}.
\end{equation*}

Finally, as indicated in \S\ref{intro_classical}, our theory must also include diagonal morphisms
\begin{equation*}
\Delta:\barH_{g,G,\xi}\to\barH_{g,G,\xi}\times\barH_{g,G,\xi}
\end{equation*}
in the calculus, as pullbacks of restriction maps by boundary morphisms may include contributions from such classes.

The main definition of the paper (see also Definition \ref{main_defn}) is then:
\begin{defn}
The \textit{$\cH$-tautological ring} is the smallest system of $\bQ$-subalgebras
\begin{equation*}
R^{*}_{\cH}\left(\prod_{j=1}^{m}\barH_{g_j,G_j,\xi_j}\right)\subset A^{*}\left(\prod_{j=1}^{m}\barH_{g_j,G_j,\xi_j}\right),
\end{equation*}
ranging over all finite products of moduli spaces of admissible $G$-covers, that contains all $\psi$ and $\kappa$ classes and is closed under pushfoward by all (products of) boundary, forgetful, restriction-corestriction, and diagonal maps.
\end{defn}

If $\gamma$ is an arbitrary composition of tautological morphisms and $\theta$ is a polynomial in the $\psi$ and $\kappa$ classes, then certainly $\gamma_{*}(\theta)$ must be $\cH$-tautological. Our main theorem (see also Theorem \ref{main_thm}) is that such classes additively generate the $\cH$-tautological ring.

\begin{thm}\label{main_thm_intro}
The classes $\gamma_{*}(\theta)$ comprise a set of additive generators for the $\cH$-tautological ring. Moreover, the $\cH$-tautological ring is closed under bivariant pullback by the previously-described tautological morphisms.
\end{thm}

In fact, one can take $\gamma$ to be a composition of diagonal, restriction-corestriction, forgetful, and boundary morphisms, in that order, see Theorem \ref{additive_gens_efficient}.

The proof of Theorem \ref{main_thm_intro} is constructive. We give algorithms for expressing the bivariant pullback of any $\cH$-tautological class by any tautological morphism. Many of the methods are analogous to those in the classical situation. For example, one can intersect boundary classes on spaces of admissible $G$-covers (see \S\ref{boundary_strata_intersection_section}) in an essentially identical fashion to that of boundary classes on $\barM_{g,n}$ (see \S\ref{boundary_intersection_classical_section} or \cite[Appendix A]{gp}), keeping track of the additional data of a group action.

However, we highlight here two aspects of the theory requiring new ideas.
\begin{enumerate}
\item Intersections of restriction-corestriction morphisms with boundary strata, see \S\ref{res-cores_boundary}. This was initiated in the case of the source map $\phi=\res^{G}_{\{1\}}$ by Schmitt-van Zelm \cite{schmittvanzelm}, where many examples of classes of Galois Hurwitz loci on $\barM_{g,n}$ determined by their integrals against classical tautological classes are computed.

We extend the results of Schmitt-van Zelm to the setting of restriction-corestriction morphisms, where the functorial intersections are, in general, non-reduced, owing to the fact that boundary strata in general lie in the branch locus of corestriction morphisms.

In particular, our methods allow for the integration of classical tautological classes against arbitrary (non-Galois) Hurwitz loci coming from Harris-Mumford spaces, which is already new. In \S\ref{harris_mumford_section}, we employ our methods to re-compute the classes of $d$-elliptic loci in genus 2, as originally carried out in \cite{lian_qmod}, and state a mild generalization of the quasimodularity observed therein, see Theorem \ref{d-elliptic_recover_general}.

\item Intersections of two restriction-corestriction morphisms, see \S\ref{res-cores_int}. These methods are new, but because the intersection-theoretic formulas are quite complicated, with terms indexed by group-theoretic data that seem difficult to compute in practice, the results may perhaps be more profitably viewed as proofs of existence. 
\end{enumerate}

\subsection{Further directions}
We briefly propose here four avenues for further investigation on the $\cH$-tautological ring.
\begin{enumerate}
\item Improvement and implementation of algorithms. SAGE code of Delecroix-Schmitt-van Zelm is available to perform computations in the classical tautological ring, see \cite{sage}. Their code has additional functionality for certain Galois Hurwitz loci on $\barM_{g,n}$.

A next step would be extend this work to include functionality for the general $\cH$-tautological calculus, particularly that of Harris-Mumford loci, which, as we explain in \S\ref{harris_mumford_section}, can in fact can be dealt with computationally without making explicit reference to Galois covers.

As we indicated above, the results of \S\ref{res-cores_int} are, at least as stated, likely too unwieldy to be implemented computationally; some theoretical speedups would thus be desirable here.

\item Study of the structure of the $\cH$-tautological ring, following, for instance, the Faber conjectures for the classical tautological ring, see \cite{faber_conj}.

Two places to start would be the following questions:
\begin{itemize}
\item What is the rank of $(R_{\cH})_{0}(\barH_{g,G,\xi})$? In the classical situation, the answer is always 1, by a result of Graber-Vakil \cite{gv}. When $G=\{1\}$, that is, $\barH_{g,G,\xi}=\barM_{g,n}$, the answer in the $\cH$-tautological setting is again 1 assuming that zero-cycles of $\barM_{g,n}$ defined over $\overline{\bQ}$ are tautological (see \cite[Speculation 1.3]{pandharipandeschmitt}), though we do not have an independent proof. In the general case, the situation is complicated by the fact that $\barH_{g,G,\xi}$ may be disconnected, but in principle the group of zero-dimensional $\cH$-tautological classes could still have rank 1.
\item What is the largest integer $k$ for which $R_{\cH}^{k}(\cM_{g})=0$? In the classical situation, the answer is $k=g-1$ due to Looijenga and Faber, see \cite{looijenga,faber_non-vanishing}. However, when one allows $\cH$-tautological classes, this becomes false when $g=12$ due to van Zelm's result that the bielliptic locus on $\cM_{12}$ is non-tautological, see \cite[Theorem 2]{vanzelm}.
\end{itemize}

\item Construction of non-trivial $\cH$-tautological relations. For classical tautological classes on $\barM_{g,n}$, the Pixton relations \cite{pixton,ppz} are conjecturally a complete set of relations; these can be pulled back to $\barH_{g,G,\xi}$. However, establishing $\cH$-tautological relations not arising in this way would require new tools. For example, the full Conjecture \ref{qmod_conj} is unlikely to be proven by (classical) ``tautological methods,'' as the classes appearing are not all tautological. 

As a step in this direction, Clader \cite{clader} has constructed relations on the moduli space of cyclic covers, pushing forward to (classical) tautological relations on $\barM_{g,n}$.

\item Study of the size of the $\cH$-tautological ring, relative to that (for example) of the classical tautological ring, or on the other hand the entire cohomology or Chow ring of $\barM_{g,n}$. In particular, it would be interesting now to attempt to construct cycles that are not $\cH$-tautological.
\end{enumerate}

\subsection{Conventions}\label{conventions} We work over $\bC$. All curves are assumed projective and connected with only nodes as singularities, unless otherwise indicated. The \textit{genus} of a curve refers to its arithmetic genus. All moduli spaces are understood to be stacks, rather than coarse moduli spaces.

We will work throughout with rational Chow groups $A^{*}(\mathcal{X})$, where $\mathcal{X}$ ranges over a system of products of Hurwitz spaces.

All group actions are left group actions, unless otherwise stated.

For sake of clarity of exposition, we will often describe constructions involving moduli spaces on the level of closed points, rather than functors; we invite the reader to generalize the discussion appropriately, for example, by replacing marked points on curves with sections of families of curves.

\subsection{Acknowledgments}

I thank Alessio Cela, Johan de Jong, Johannes Schmitt, Johannes Schwab and Jason van Zelm for helpful comments and discussions about ideas related to this paper, and David Yang for providing the proof of Lemma \ref{david_lemma}. I also thank the anonymous referee for numerous suggestions that improved the exposition and content of the paper.

This work was completed primarily with the support of a NSF Postdoctoral Research Fellowship, grant DMS-2001976.


\section{Tautological classes on moduli spaces of curves}\label{tautological_classical_section}

In this section, we recall the classical intersection theory of tautological classes on moduli spaces of curves, as initiated by \cite{mumford}; we follow the approaches of \cite{arbarellocornalba} and \cite{gp}.

\subsection{The universal family}\label{universal_family_mgn}

\begin{defn}\label{forgetful_mgn_defn}
We define the forgetful morphism $\pi:\barM_{g,n+1}\to\barM_{g,n}$ forgetting the last marked point and contracting non-stable components, along with its \textbf{universal sections} $\sigma_i:\barM_{g,n}\to\barM_{g,n+1}$ attaching a 2-pointed rational tail to the $i$-th marked point.
\end{defn}
Taken together, $\pi$ and the sections $\sigma_i$ form the universal family over $\barM_{g,n}$.

We have the following easy result:

\begin{lem}\label{double_forgetful_mgn}
The commutative diagram
\begin{equation*}
\xymatrix{
\barM_{g,n+2} \ar[r]^{\pi_{n+1}} \ar[d]^{\pi_{n+2}} & \barM_{g,n+1}\ar[d]^{\pi}\\
\barM_{g,n+1} \ar[r]^{\pi} & \barM_{g,n}
}
\end{equation*}
is Cartesian up to birational equivalence.\footnote{by this we mean: $\barM_{g,n+2}$ maps properly and birationally onto the fiber product.}  Here, the two maps out of $\barM_{g,n+2}$ forget the $(n+1)$-st and $(n+2)$-nd points, and the maps into $\barM_{g,n}$ forget the other of these two points.
\end{lem}

\subsection{$\psi$ and $\kappa$ classes}

\begin{defn}\label{psi_mgn_defn}
Let $\omega_\pi$ be the relative dualizing sheaf of the forgetful morphism $\pi:\barM_{g,n+1}\to\barM_{g,n}$. The \textbf{$\psi$ classes} are defined by
\begin{equation*}
\psi_i:=\sigma_i^{*}(c_1(\omega_\pi))\in A^1(\barM_{g,n}).
\end{equation*}
\end{defn}

\begin{defn}\label{kappa_mgn_defn}
Let $K$ be the divisor class of the log-canonical sheaf $\omega_{\pi}(\sigma_1+\cdots+\sigma_n)$. We define the \textbf{$\kappa$ classes}
\begin{equation*}
\kappa_i:=\pi_{*}(K^{i+1})\in A^i(\barM_{g,n}).
\end{equation*}
\end{defn}

\begin{rem}
One can also take as a definition 
\begin{equation*}
\widetilde{\kappa}_i:=\pi_{*}(c_1(\omega_\pi)^{i+1})\in A^i(\barM_{g,n}).
\end{equation*}
The two possible notions of $\kappa$ classes are related by the formula
\begin{equation*}
\kappa_i=\widetilde{\kappa}_i+\sum_{j=1}^{n}\psi_j^i,
\end{equation*}
see \cite[(1.5)]{arbarellocornalba}.
\end{rem}

\subsection{Stable graphs and boundary strata}

The boundary of $\barM_{g.n}$ admits a stratification by substacks indexed by stable graphs and parametrized by products of moduli spaces of curves of smaller dimension.

\begin{defn}
A \textbf{graph} $\Gamma$ consists of the data of a vertex set $V(\Gamma)$, a half-edge set $H(\Gamma)$, a fixed-point free involution $\iota:H(\Gamma)\to H(\Gamma)$ whose orbits form the edge set $E(\Gamma)$ of $E$, and an assignment $a:H(\Gamma)\to V(\Gamma)$ of a vertex to each half-edge. 

A \textbf{stable graph} includes the additional data of a genus function $g:V(\Gamma)\to\bZ_{\ge0}$ and an ordered set $L(\Gamma)$ of legs with an assignment $\zeta:L(\Gamma)\to V(\Gamma)$ of a vertex to each leg. We denote the \textbf{valency} of a vertex, the sum of the number of half-edges and legs of $\Gamma$ incident at $v$, by $n(v)$, and require that $\Gamma$ satisfies the stability condition that $2g(v)-2+n(v)>0$ for any $v\in V(\Gamma)$.

The \textbf{genus} of $\Gamma$ is defined to be 
\begin{equation*}
g(\Gamma):=h^1(\Gamma)+\sum_{v\in V(\Gamma)}g(v), 
\end{equation*}
and the number of elements $n$ of $L(\Gamma)$ is referred to as the number of \textbf{marked points} of $\Gamma$.
\end{defn}

Stable graphs of genus $g$ with $n$ marked points are exactly the dual graphs of stable curves parametrized by $\barM_{g,n}$. 

\begin{defn}\label{boundary_strata_defn}
Given a stable graph $\Gamma$, we define the \textbf{boundary stratum}
\begin{equation*}
\barM_{\Gamma}:=\prod_{v\in V(\Gamma)}\barM_{g(v),n(v)}
\end{equation*}
We then have a natural \textbf{boundary morphism}
\begin{equation*}
\xi_{\Gamma}:\barM_{\Gamma}\to\barM_{g,n}
\end{equation*}
gluing together marked points corresponding to paired half-edges, whose image is the closure of the locus of stable curves with dual graph $\Gamma$.
\end{defn}

The map $\xi_{\Gamma}$ is finite and unramified; the degree onto its image is given by $\#\Aut(\Gamma)$, where the notion of a morphism (and hence an automorphism) of a stable graph is given below.

\begin{defn}
Suppose $\Gamma,A$ are stable graphs of genus $g$ with $n$ marked points. A \textbf{morphism} of stable graphs $f:\Gamma\to A$ (or an \textbf{$A$-structure on $\Gamma$}) consists of a surjection on vertex sets $f_V:V(\Gamma)\to V(A)$ and an injection (in the opposite direction) on half-edge sets $f_H:H(A)\to H(\Gamma)$, compatible with the half-edge involutions $\iota$, edge attachment maps $a$, and leg assignment maps $\zeta$, and respecting the genus function in the sense that the pre-image of a vertex $v\in V(A)$ along with its incident edges forms a stable graph of genus $g(v)$. (See \cite[Definition 2.5]{schmittvanzelm} for a more explicit definition.)
\end{defn}
 
A morphism $f:\Gamma\to A$ captures the degeneration of a curve with dual graph $A$ into one with dual graph $\Gamma$. Indeed, a map $f:\Gamma\to A$ of stable graphs induces in a natural way a map $\xi_{\Gamma\to A}:\barM_{\Gamma}\to\barM_A$, compatible with the boundary maps $\xi_{\Gamma},\xi_{A}$.

The moduli space $\barM_{\Gamma}$ may also be interpreted as the moduli of stable marked curves $X$ along with a $\Gamma$-structure on its dual graph, see \cite[Proposition 8]{gp}.

The normal bundle of $\xi_{\Gamma}$ is given by a direct sum of line bundle contributions from the edges of $\Gamma$, or equivalently the nodes of the general curve parametrized by $\barM_{\Gamma}$. More precisely, we have the following well-known fact:

\begin{lem}\label{normal_bundle}
\begin{equation*}
N_{\xi_{\Gamma}}=\bigoplus_{(\ell,\ell')\in E(\Gamma)}T_{\ell}\otimes T_{\ell'},
\end{equation*}
where $(\ell,\ell')$ denotes the pairs of half-edges comprising edges of $\Gamma$ and $T_{\ell},T_{\ell'}$ are the corresponding tangent line bundles on the appropriate components $\barM_{g(v),n(v)}$. In particular, we have
\begin{equation*}
c(N_{\xi_{\Gamma}})=\prod_{e\in E(\Gamma)}(1-\psi_{h_e}-\psi_{h'_e}).
\end{equation*}
\end{lem}

\subsection{Intersections of boundary strata}\label{boundary_intersection_classical_section}

The main reference for this section is \cite[Appendix A]{gp}. Let $A,B$ be stable graphs of genus $g$ with $n$ marked points. We consider the intersection $\xi_A^{*}(\xi_{B*}(1))\in A^{*}(\barM_A)$.

\begin{lem}\label{boundary_intersection_classical_cartesian}
We have a Cartesian diagram
\begin{equation*}
\xymatrix{
\displaystyle\coprod\barM_{\Gamma} \ar[d]_{\xi_{\Gamma\to A}} \ar[r]^(0.53){\xi_{\Gamma\to B}} & \barM_{B} \ar[d]^{\xi_{B}}\\
\barM_{A} \ar[r]^{\xi_{A}} & \barM_{g,n}
}
\end{equation*}
where the disjoint union ranges over \textbf{generic $(A,B)$-graphs} $\Gamma$, that is, stable graphs $\Gamma$ equipped with morphisms $\Gamma\to A,B$ satisfying the genericity condition that the map
\begin{equation*}
E(A)\sqcup E(B)\to E(\Gamma)
\end{equation*}
is surjective.
\end{lem}

The excess bundle on $\barM_{\Gamma}$ may be easily understood in terms of the normal bundles of boundary morphisms as in Lemma \ref{normal_bundle}. One finds:
\begin{lem}
\begin{equation*}
\xi_{\Gamma\to A}^{*}N_{\xi_A}/N_{\xi_{\Gamma\to B}}\cong\bigoplus_{(\ell,\ell')\in E(A)\cap E(B)}T_{\ell}\otimes T_{\ell'},
\end{equation*}
where $(\ell,\ell')$ denotes the pair of half-edges comprising an edge.
\end{lem}

Therefore, by the excess intersection formula, we conclude:
\begin{prop}\cite[(11)]{gp}
Let $A,B$ be stable graphs of genus $g$ with $n$ marked points. We have:
\begin{equation*}
\xi^{*}_A(\xi_{B*}(1))=\sum_{\Gamma}(\xi_{\Gamma\to A})_{*}\left(\prod_{(\ell,\ell')\in E(A)\cap E(B)}(-\psi_{\ell}-\psi_{\ell'})\right),
\end{equation*}
where the sum is over $\Gamma$ as above.
\end{prop}

\subsection{Pullbacks of $\psi$ and $\kappa$ classes}\label{psi_kappa_pullback_classical_section}

Here, we record formulas for the pullbacks of $\psi$ and $\kappa$ classes by forgetful and boundary morphisms.

\begin{prop}\label{psi_kappa_forgetful_pullback_classical}
Let $\pi:\barM_{g,n+1}\to\barM_{g,n}$ be the map forgetting the last marked point. We have:
\begin{equation*}
\pi^{*}(\psi_i)=\psi_i-{\sigma_i}_{*}(1),
\end{equation*}
where $\sigma_i$ is the $i$-th universal section, and 
\begin{equation*}
\pi^{*}(\kappa_i)=\kappa_i-\psi_{n+1}^{i}
\end{equation*}
\end{prop}
\begin{proof}
The first statement appears in \cite[\S 3]{arbarellocornalba_picard} and the second is \cite[(1.10)]{arbarellocornalba}.
\end{proof}

\begin{prop}\label{psi_kappa_boundary_pullback_classical}
Let $\xi_\Gamma:\barM_{\Gamma}\to\barM_{g,n}$ be the boundary map associated to the stable graph $\Gamma$. We have:
\begin{equation*}
\xi_\Gamma^{*}(\psi_i)=\pr_v^{*}\psi_i,
\end{equation*}
where $\pr$ is the projection to the component $\barM_{g(v),n(v)}$ of $\barM_{\Gamma}$ corresponding to the component containing the $i$-th marked point, and
\begin{equation*}
\xi_\Gamma^{*}(\kappa_i)=\sum_{v\in V(\Gamma)}\pr_v^{*}(\kappa_i).
\end{equation*}
\begin{proof}
The first statement is trivial, and the second is \cite[(1.8)]{arbarellocornalba}.
\end{proof}
\end{prop}

\subsection{The tautological ring}

\begin{defn}
The \textbf{tautological ring} is the smallest system of $\bQ$-subalgebras $R^{*}(\barM_{g,n})\subset A^{*}(\barM_{g,n})$ containing all $\psi$ and $\kappa$ classes and closed under pushforwards by all boundary morphisms $\xi_{\Gamma}:\barM_{\Gamma}\to\barM_{g,n}$ and all forgetful morphisms $\pi:\barM_{g,n+1}\to\barM_{g,n}$.
\end{defn}

Similarly, one can define the tautological $RH^{*}(\barM_{g,n})\subset H^{*}(\barM_{g,n})$ by the image of the tautological ring in Chow under the cycle class map.

Additive generators for the tautological ring are given by \textbf{decorated boundary classes}, that is, pushforwards of monomials in $\psi$ and $\kappa$ classes by boundary morphisms, see \cite[Proposition 11]{gp}. Indeed, by definition, these classes must be tautological, and the results of \S\ref{boundary_intersection_classical_section} and \ref{psi_kappa_pullback_classical_section} show that the sub-$\bQ$-vector spaces $A^{*}(\barM_{g,n})$ generated by decorated boundary classes are closed under multiplication. In fact, one easily sees that the tautological ring is closed under all bivariant pullbacks by forgetful and boundary morphisms.

Computing intersection numbers of tautological classes amounts to computing to integrals of polynomials in $\kappa$ and $\psi$ classes in top degree. Such integrals are given by the KdV hierarchy, see \cite{kontsevich}.

\subsection{Chern classes of the tangent bundle}

The Chern classes of the tangent bundle of $\barM_{g,n}$ do not play a direct role in the calculus of tautological classes as we have introduced it, but we will need the following result in our extended theory. Recall from standard deformation theory that the cotangent bundle to $\barM_{g,n}$ is given by $\pi_{*}(\Omega_{\pi}(\sigma_1+\cdots+\sigma_n)\otimes\omega_{\pi})$, where $\Omega_{\pi}$ denotes the relative sheaf of differentials.

Adapting Mumford's \cite{mumford} well-known application of the Grothendieck-Riemann-Roch theorem to the universal family $\pi$, Bini \cite[Theorem 2]{bini} gives an explicit formula for $\ch(\mathcal{T}^{*}_{\barM_{g,n}})$. We do not reproduce the formula here, the following will suffice for us.
\begin{thm}\cite{bini}
We have
\begin{equation*}
c(\mathcal{T}^{*}_{\barM_{g,n}})\in R^{*}(\barM_{g,n}).
\end{equation*}
\end{thm}

\section{From moduli of Galois covers to the $\cH$-tautological ring}\label{hurwitz_space_section}

\subsection{Admissible $G$-covers}\label{galois_covers_def_section}

Stacks of admissible covers of nodal curves were introduced by Harris-Mumford in \cite{hm} to compactify Hurwitz spaces of branched covers of smooth curves. Later work by Mochizuki \cite{mochizuki} and Abramovich-Corti-Vistoli \cite{acv} clarified the notion of a family of admissible covers over an arbitrary base scheme. In addition, the moduli space $\barM_{g',b}(BG,d)$ of stable maps to $BG$ constructed by Abramovich-Corti-Vistoli gives a normalization of the Harris-Mumford space by way of twisted $G$-covers. Incarnations of this space of also appear in \cite{jkk} and \cite{br}; for enumerative applications, it is most convenient to adopt the conventions of \cite{schmittvanzelm}, which we now recall.

\begin{defn}
Let $G$ be a finite group and let $S$ be an arbitrary base scheme. An \textbf{admissible $G$-cover} over $S$ is a finite morphism $f:X\to Y$ of connected nodal curves over $S$ with a fiberwise action of $G$ on $X$ and distinct marked sections $q_1,\ldots,q_b:S\to Y$ satisfying the following properties:
\begin{itemize}
\item $[D,q_1,\ldots,q_b]$ is a stable curve over $S$,
\item if $s\in S$ is a geometric point, then $f(x)\in Y_s$ is a node if and only if $x\in X_s$ is,
\item $f$, with the given $G$-action on $X$, is a principal $G$-bundle away from $q_1,\ldots,q_b\in Y$ and the nodes of $Y$
\item in analytic local coordinates over a marked point of $D$, the map $X\to Y\to S$ is isomorphic to
\begin{equation*}
\Spec(R[x'])\to\Spec(R[x])\to\Spec(R),
\end{equation*}
where $e\ge1$ is an integer and $f^{*}(x)=(x')^e$,
\item in analytic local coordinates over a node of $D$, the map $X\to Y\to S$ is isomorphic to
\begin{equation*}
\Spec(R[x',y']/(x'y'-r))\to\Spec(R[x,y]/(xy-r^e))\to\Spec(R),
\end{equation*}
where $e\ge1$ is an integer and $f^{*}(x)=(x')^e$ and $f^{*}(y)=(y')^e$, and
\item in analytic local coordinates at a node $p\in C$ over a geometric point $s\in S$, the $G$-action on $C_s$ has cyclic stabilizer $G_p\cong\bZ/e\bZ$ and is \textbf{balanced}. That is, if
\begin{equation*}
\Spec(R[x',y']/(x'y'-r))\to\Spec(R)
\end{equation*}
is a local coordinate chart near $p$, then a generator of $g_p\in G_p$ acts by
\begin{equation*}
g_p\cdot(x',y')=(\mu x',\mu^{-1}y'),
\end{equation*}
where $\mu$ is a primitive $e$-th root of unity.
\end{itemize}
\end{defn}

Note that, if $f:X\to Y$ is an admissible $G$-cover, then $Y$ is isomorphic to the scheme-theoretic quotient $X/G$.

By definition, the target of an admissible $G$-cover is a stable marked curve, allowing us to define later the target morphism in Definition \ref{target_morphism_defn}. In order for the source to be another such, one needs to mark the points of the marked fibers of an admissible $G$-cover. 

\begin{defn}
A \textbf{pointed admissible $G$-cover} consists of the data of an admissible $G$-cover, along with, for each $i=1,2,\ldots,b$, a distinguished marked point $p_{i1}\in f^{-1}(q_i)$. 
\end{defn}

Given a choice of $p_{i1}$, the $G$-action on the fibers of $f$ yields a marking of the entire fiber $f^{-1}(q_i)$ by the left coset $G/G_{p_{i1}}$. Indeed, if $t\in G$, then we may label the point $t\cdot p_{i1}$ by the coset $tG_{p_{i1}}$. The resulting curve with all pre-images of the $q_i$ taken as marked points is then stable, see \cite[Lemma 3.4]{schmittvanzelm}, allowing us to define later the source morphism in Definition \ref{source_morphism_defn}.

We also record the monodromy of the $G$-cover $f$ at each distinguished marked point. The stabilizer of the $G$-action at $p_{i1}$ is a cyclic group $G_{p_{i1}}$ of order $n_i$, see \cite[Lemma 1]{schmittvanzelm}. Thus, $G_{p_{i1}}$ acts on the tangent space $T_{p_{i1}}X\cong\bC$ by a character of order $n_i$.

\begin{defn}
Let $f:X\to Y$ be a pointed admissible $G$-cover. For $i=1,2,\ldots,b$, we define the element $h_i\in G_{p_{i1}}\subset G$ to be the unique element acting on $T_{p_{i1}}X\cong\bC$ by multiplication by the root of unity $e^{2\pi\sqrt{-1}\cdot n_i}$. We call the $b$-tuple $\xi=(h_1,\ldots,h_b)\in G^b$ the \textbf{monodromy data} of the pointed admissible cover $f$.
\end{defn}.

The monodromy data at the distinguished marked points of an admissible $G$-cover $f$ determines the monodromy data at \textit{all} marked points. Indeed, the stabilizer of the point $t\cdot p_{i1}$ is the conjugate $tG_{p_{i1}}t^{-1}$ with distinguished generator $th_it^{-1}$.

We are now ready to define the moduli spaces of admissible $G$-covers.

\begin{defn}\label{admissible_galois_main}
Let $b,g\ge0$ integers, $G$ be a finite group, and let $\xi=(h_1,\ldots,h_b)\in G^b$ be a tuple of group elements. We define $\barH_{g,G,\xi}$ to be the moduli space parametrizing admissible $G$-covers $f:X\to Y$, along with the data of marked points $q_1,\ldots,q_b\in Y$ and distinguished marked points $p_{i1}\in f^{-1}(q_i)$, such that $X$ has genus $g$ and $f$ has monodromy data $\xi$ (as measured at the $p_{i1}$).
\end{defn}

When $G=\{1\}$, we recover the usual moduli spaces of stable pointed curves. Note that the genus $g'$ of $Y$ is determined by the Riemann-Hurwitz formula:
\begin{equation*}
2g-2=\#G\cdot \left[(2g'-2)+\sum_{i=1}^{b}\frac{\ord_{G}(h_i)-1}{\ord_G(h_i)}\right].
\end{equation*}

The space $\barH_{g,G,\xi}$ is equipped with two natural maps remembering the source and target curve, respectively.

\begin{defn}\label{target_morphism_defn}
The \textbf{target morphism}
\begin{equation*}
\delta:\barH_{g,G,\xi}\to\barM_{g',b}
\end{equation*}
sends a pointed $G$-cover $f:X\to Y$ to the curve $[Y,\{y_i\}]$.
\end{defn}

\begin{defn}\label{source_morphism_defn}
The \textbf{source morphism}
\begin{equation*}
\phi:\barH_{g,G,\xi}\to\barM_{g,r}
\end{equation*}
sends a pointed $G$-cover $f:X\to Y$ to the curve $[X,\{p_{ia}\}]$. Here, the index set for the set of marked points is the disjoint union of the left coset spaces $G/\langle h_i\rangle$.
\end{defn}

We point out that there is already a source of ambiguity here, namely that the target $\barM_{g,r}$ parametrizes genus $g$ curves with \textit{ordered} marked points, whereas, while the marked points on a source curve parametrized by $\barH_{g,G,\xi}$ are distinguished, only the marked points of the target (and hence the $G$-orbits on the source) are ordered. Thus, the map $\phi$ must be defined relative to an ordering on the marked points, which we allow to be chosen independently on the ordering of the $G$. For sake of brevity, we will suppress this choice in the notation. We will see this ambiguity in the more general definition of restriction morphisms later.

\begin{thm}\cite[Theorem 3.7]{schmittvanzelm}\label{source_target_properties}
The map $\phi$ is representable, finite, unramified, and a local complete intersection. The map $\delta$ is flat, proper, and quasi-finite, but not necessarily representable (hence not necessarily finite). 

Moreover, the space $\barH_{g,G,\xi}$ is a smooth, proper, Deligne-Mumford stack of dimension $3g'-3+b$.
\end{thm}

The degree of $\delta$ is given by a \textbf{Hurwitz number}, which may be computed purely group-theoretically, see \cite[Theorem 3.19]{schmittvanzelm}.

The ramification locus of $\delta$ consists of covers $f:X\to Y$ ramified to order $e\ge2$ over at least one node $q\in Y$. Indeed, if $t$ is a smoothing parameter on the deformation space of $Y$ corresponding to $q$, and $s$ is a smoothing parameter on the deformation space of $X$ (and thus for the cover $f$, because $\phi$ is unramified) corresponding to the $G$-orbit above $q$, then $\delta^{*}(t)=s^e$, up to a unit.

However, the deformations of $f$ are controlled by the \textit{twisted curve} $\mathcal{Y}$ associated to the cover $f$, see \cite[\S 3]{acv}. Here, $\mathcal{Y}$ is a Deligne-Mumford stack with coarse space $Y$ such that the natural map $c:\mathcal{Y}\to Y$ is an isomorphism away from the nodes and marked points, equipped with a principal $G$-bundle $\wt{f}:X\to \mathcal{Y}$ such that $c\circ \wt{f}=f$.

\subsection{The universal family}\label{universal_family_hurwitz}

Given a space of admissible $G$-covers $\barH_{g,G,\xi}$, with $\xi=(h_1,\ldots,h_b)$, we will denote the monodromy datum $(h_1,\ldots,h_b,1)$ by the somewhat abusive notation $\xi+\{1\}$. Then, the space $\barH_{g,G,\xi+\{1\}}$ parametrizes the same covers as $\barH_{g,G,\xi}$, along with an additional $G$-orbit of unramified points. 

\begin{defn}
We define the natural forgetful morphism $\pi:\barH_{g,G,\xi+\{1\}}\to \barH_{g,G,\xi}$ forgetting the additional unramified orbit and contracting non-stable components on the source and target.
\end{defn}

The following is straightforward:

\begin{lem}\label{universal_family_hurwitz_comparison}
The diagram
\begin{equation*}
\xymatrix{
\barH_{g,G,\xi+\{1\}} \ar[r]^{\phi'} \ar[d]^{\pi_G} & \barM_{g,r+1} \ar[d]^{\pi}\\
\barH_{g,G,\xi} \ar[r]^{\phi} & \barM_{g,r}
}
\end{equation*}
where $\phi'$ is the composition of the source map on $\barH_{g,G,\xi+\{1\}}$ followed by the map forgetting all marked points in the extra $G$-orbit except the distinguished one, is Cartesian.

Moreover, the forgetful map $\pi_G:\barH_{g,G,\xi+\{1\}}\to \barH_{g,G,\xi}$, taken with the pullbacks of the universal sections $\sigma_{p_{ia}}$ by $\phi$, form the universal family of (sources of) admissible $G$-covers over $\barH_{g,G,\xi}$.
\end{lem}

We also describe a more explicit construction of the universal sections to $\pi_G:\barH_{g,G,\xi+\{1\}}\to \barH_{g,G,\xi}$, also abusively called $\sigma_{p_{ia}}$. One may also re-interpret this construction in the language of boundary classes, explained in \S\ref{boundary_strata_hurwitz_section}.

Suppose $f:X\to Y$ is a pointed admissible $G$-cover, and let $p_{ia}$ be a marked point of $X$ mapping to $q_i\in Y$ with ramification index $e_i$. We define the admissible $G$-cover $f':X'\to Y'$ as follows. Let $[Y']=\sigma_i([Y])$, and let $X'$ be the curve obtained from $X$ by attaching rational tails to the points in the $G$-orbit of $p_{ia}$. Then, $f'$ is is the admissible $G$-cover agreeing with $f$ after contraction of the new rational components, and totally ramified over the new node of $Y'$ and over $q_i$. Finally, any point in the pre-image of $q_{b+1}$ on the same component of $p_{ia}$ may be taken to be the distinguished marked point of the new marked fiber, giving $f'$ the structure of a pointed admissible $G$-cover. We then have $\sigma_{p_{ia}}([f])=[f']$.

\subsection{$\psi$ and $\kappa$ classes}

\begin{defn}
We define the $\psi$ and $\kappa$ classes on $\barH_{g,G,\xi}$ by:
\begin{equation*}
\psi_{p_{ia}}=\phi^{*}\psi_{p_{ia}}\in A^1(\barH_{g,G,\xi})
\end{equation*} and 
\begin{equation*}
\kappa_i=\phi^{*}\kappa_i\in A^i(\barH_{g,G,\xi}),
\end{equation*}
respectively, where as usual $\phi:\barH_{g,G,\xi}\to\barM_{g,r}$ is the source morphism.
\end{defn}
Lemma \ref{universal_family_hurwitz_comparison} shows that one can equivalently define the $\psi$ and $\kappa$ classes in terms of the relative dualizing sheaf of $\pi:\barH_{g,G,\xi+\{1\}}\to\barH_{g,G,\xi}$, as in Definitions \ref{psi_mgn_defn} and \ref{kappa_mgn_defn}.

Up to constant factors, the $\psi$ and $\kappa$ classes on $\barH_{g,G,\xi}$ may also be defined by pullback by $\delta$. Indeed, we have the following:

\begin{prop}\cite[Lemma 3.9]{schmittvanzelm}\label{psi_kappa_comparison_forgetful}
We have:
\begin{equation*}
\psi_{p_{ia}}=\frac{1}{\ord_G(h_i)}\cdot\delta^{*}(\psi_{q_i}).
\end{equation*}
Moreover,
\begin{equation*}
\kappa_{i}=(\# G)\cdot\delta^{*}(\kappa_i).
\end{equation*}
\end{prop}

\subsection{Admissible $G$-graphs and boundary strata}\label{boundary_strata_hurwitz_section}

We review the construction of \cite[\S 3.4-5]{schmittvanzelm}. In analogy with the stratification of $\barM_{g,n}$ indexed by stable graphs, we have a stratification of $\barH_{g,G,\xi}$ indexed by so-called admissible $G$-graphs. The dual graph $\Gamma$ of a (source) curve parametrized by $\barH_{g,G,\xi}$ naturally acquires a $G$-action, along with monodromy data at its legs and half-edges. 

\begin{defn}\cite[Definition 3.14]{schmittvanzelm}
Let $\barH_{g,G,\xi}$ be a moduli space of admissible $G$-covers, with $\xi=(h_1,\ldots,h_b)$. An \textbf{admissible $G$-graph} $(\Gamma,G)$ of genus $g$ and monodromy data $\xi$ is a stable graph of genus $g$ with $G$-action, with the additional data of a distinguished leg $\ell_{i1}$ in each $G$-orbit of legs and an element $h_{\ell}\in G$ for all $\ell\in H(\Gamma)\sqcup L(\Gamma)$, satisfying the following conditions:
\begin{itemize}
\item (monodromy data generate stabilizers) the stabilizer $G_{\ell}\subset G$ of $\ell$ is cyclic with generator $h_{\ell}$,
\item (monodromy data agrees with $\xi$) $h_{\ell_{i1}}=h_i$ for all $i=1,2,\ldots,b$,
\item (monodromy data are compatible in $G$-orbits) $h_{t\cdot\ell}=th_{\ell}t^{-1}$ for all $\ell\in H(\Gamma)\sqcup L(\Gamma)$ and $t\in G$,
\item (quotient by $G$ doesn't collapse edges) if $(\ell,\ell')$ is a pair of half-edges comprising an edge, then $t\cdot\ell\neq\ell'$ for all $t\in G$, and
\item (balancing condition) if $(\ell,\ell')$ is a pair of half-edges comprising an edge, then $h_{\ell}=h_{\ell'}^{-1}$.
\end{itemize}
\end{defn}

As with admissible $G$-covers, the $G$-orbits of the legs of an admissible $G$-graph are assumed to be \textit{ordered}. The dual graph of the source curve of any cover parametrized by $\barH_{g,G,\xi}$ naturally acquires the structure of an admissible $G$-graph.

Schmitt-van Zelm distinguish \textit{pre-admissible} $G$-graphs from admissible $G$-graphs, where an admissible $G$-graph is required to correspond to a non-empty boundary stratum; we do not make this distinction here.

We now define the boundary stratum associated to the admissible $G$-graph $(\Gamma,G)$. We first choose a collection $V'(\Gamma)\subset V(\Gamma)$ of $G$-orbit representatives on the set of vertices of $\Gamma$. Then, for each $v\in V'(\Gamma)$, we choose distinguished orbit representatives for the $G_v$-actions on the sets of legs and half-edges incident at $v$, where $G_v\subset G$ is the stabilizer of $v$. Let $L'_{v},H'_{v}$ be the resulting sets of distinguished legs and half-edges, respectively. Let 
\begin{equation*}
\xi'_{v}=\left(h_{\ell}:\ell\in L'_{v}\coprod H'_{v}\right).
\end{equation*}

\begin{defn}\label{boundary_stratum_hurwitz_defn}
With notation as above define the \textbf{boundary stratum}
\begin{equation*}
\barH_{(\Gamma,G)}=\prod_{v\in V'(\Gamma)}\barH_{g(v),G_v,\xi'_v}
\end{equation*}
\end{defn}

Suppose $[C_v]\in\barH_{g(v),G_v,\xi'_v}$, where $C_v$ is a marked stable curve carrying a $G_v$ action. Let
\begin{equation*}
\wt{C}_v=\coprod_{g\in G/G_v}C_v.
\end{equation*}
Here, $G/G_v$ acts naturally on the components, which are in natural bijection with the $G$-orbit of $v\in V(\Gamma)$. The component corresponding to $1\in G/G_v$, in addition, carries the given action of $G_v$ on $C_v$. We then choose a $G$-action on $\wt{C}_v$ compatible with both of these actions; this amounts to the choice of coset representatives of $G_v$ in $G$, and will not affect the definition of boundary morphism.

Given such a $C_v$ for each $v\in V'(\Gamma)$, we then glue the curves $\wt{C}_v$ according to the graph $\Gamma$, to produce a $G$-curve $C$ with the required genus. Moreover, we also distinguish one point $G$-orbit of marked points of $C$ according to $\Gamma$.

\begin{defn}\label{boundary_morphism_hurwitz_defn}
We define the \textbf{boundary morphism}
\begin{equation*}
\xi_{(\Gamma,G)}:\barH_{(\Gamma,G)}\to\barH_{g,G,\xi}.
\end{equation*}
by the construction given above, namely:
\begin{equation*}
\xi_{(\Gamma,G)}\left(\{[C_v]\}_{v\in V'(\Gamma)(G)}\right)=[f:C\to C/G].
\end{equation*}
\end{defn}

\begin{rem}\label{boundary_relate_to_classical}
Boundary maps on spaces of admissible $G$-covers are related to the classical boundary maps by the following commutative diagram, see \cite[(16)]{schmittvanzelm}):
\begin{equation*}
\xymatrix{
\barH_{(\Gamma,G)} \ar[r]^{\xi_{(\Gamma,G)}} \ar[d]_{\phi_{(\Gamma,G)}} & \barH_{g,G,\xi} \ar[d]^{\phi} \\
\barM_{\Gamma} \ar[r]^{\xi_{\Gamma}} & \barM_{g,r}
}
\end{equation*}
Here, the map $\phi_{(\Gamma,G)}$ is defined by the composition 
\begin{equation*}
\xymatrix{
\barH_{(\Gamma,G)}=\displaystyle\prod_{v\in V'(\Gamma)}\barH_{g_v,G_v,\xi_v} \ar[r]^(0.59){\Delta} & \displaystyle\prod_{v\in V(\Gamma)}\barH_{g_v,G_v,\xi_v} \ar[r]^(0.45){\prod_v\phi} & \displaystyle\prod_{v\in V(\Gamma)}\barM_{g_v,n_v}=\barM_{\Gamma}
}.
\end{equation*}
\end{rem}

The image of $\xi_{(\Gamma,G)}$ is the closure of the locus parametrizing admissible $G$-covers whose source curve has underlying admissible $G$-graph $\Gamma$. As in the case of boundary morphisms on $\barM_{g,n}$, the morphism $\xi_{(\Gamma,G)}$ is finite and unramified of degree $\#\Aut((\Gamma,G))$ onto its image, according to the following natural notion of morphism of admissible $G$-graphs:

\begin{defn}\label{admissible_graph_morphism_def}
A \textbf{morphism} $f:(\Gamma,G)\to(A,G)$ between admissible $G$-graphs of genus $g$ and monodromy $\xi$ is a $G$-equivariant morphism of stable graphs $f:\Gamma\to A$ satisfying the obvious notions of compatibility of monodromy data at half-edges, see Definition \cite[Proposition 3.12]{schmittvanzelm}. 
\end{defn}
The morphism $f$ induces a map $\xi_{\Gamma\to A}:\barH_{(\Gamma,G)}\to\barH_{(A,G)}$ compatible with the boundary morphisms to $\barH_{g,G,\xi}$. The map $\xi_{\Gamma\to A}$ is itself a product of usual boundary morphisms into the components of the target.

The moduli space $\barH_{(\Gamma,G)}$ may also be interpreted as the moduli of admissible $G$-covers $f:X\to Y$ along with a $(\Gamma,G)$-structure on the underlying admissible $G$-graph of $X$, see \cite[Proposition 4.6]{schmittvanzelm}.

Again in analogy with boundary morphisms on $\barM_{g,n}$, the normal bundle of $\xi_{(\Gamma,G)}$ is given by a direct sum of line bundle contributions from the $G$-orbits of edges of $\Gamma$. More precisely:
\begin{lem}\label{normal_bundle_hurwitz}
We have
\begin{equation*}
N_{\xi_{(\Gamma,G)}}=\bigoplus_{(\ell,\ell')\in E'(\Gamma)}T_{\ell}\otimes T_{\ell'}
\end{equation*}
and thus
\begin{equation*}
c(N_{\xi_{(\Gamma,G)}})=\prod_{(\ell,\ell')\in E'(\Gamma)}(1-\psi_{\ell}-\psi_{\ell'}).
\end{equation*}
Here $(\ell,\ell')$ denotes a pair of half-edges comprising an edge, taken over any subset $E'(\Gamma)\subset E(\Gamma)$ of orbit representatives for the $G$-action on $E(\Gamma)$, and $T_h,T_{h'}$ are the corresponding tangent line bundles on $\barH_{(\Gamma,G)}$.
\end{lem}
Note that the $\psi$ classes in the formula above are independent of the choice of $E'(\Gamma)$, as $\psi$ classes associated to half-edges in the same $G$-orbit are equal.

\subsection{Restriction-corestriction maps}

This is the main new feature of spaces of admissible $G$-covers not available in the classical case.

Suppose $X\to Y$ is an admissible $G$-cover, so that $Y$ is isomorphic to the scheme-theoretic quotient $X/G$. If $G_1\subset G$ is a subgroup, one can \textit{restrict} to the action of $G_1$, that is, remember the admissible $G_1$-cover $X\to X/G_1$. If $G_1\subset G$ is normal, then we may also \textit{corestrict} to the induced action of $G/G_1$ on $X/G_1$, remembering the admissible $G/G_1$-cover $X/G_1\to Y=X/G$.

One must also keep track of the monodromy data in applying these operations. This is essentially done in \cite[\S 2.2.2]{br}, but in our setting with distinguished points in marked $G$-orbits, the definitions require a bit more care.

First, consider the case of restriction to the action of an arbitrary subgroup $G_1\subset G$. Upon such a restriction, the $G$-orbits of marked points on $X$ break into $G_1$-orbits, and one must make a choice of distinguished elements $\{p_{ij}\}$ of the $G_1$-orbits. This amounts, for $i=1,2,\ldots,b$, to choosing representatives $\langle t_{ij}\rangle$ for the orbits of the $G_1$-action on $G/\langle h_i\rangle$.

The stabilizer $G_{ij}$ of the $G$-action at the point $p_{ij}=t_{ij}\cdot p_{i1}$ is the cyclic group generated by
\begin{equation*}
t_{ij}h_it_{ij}^{-1}\in G,
\end{equation*}
acting on $T_{p_{ij}}X\cong \bC$ by multiplication by $e^{2\pi\sqrt{-1}\cdot\# G_{ij}}$. (Note that this formula is independent of the choice of $t_{ij}$ within a left coset of $\langle h_i\rangle$.)

On the other hand, the stabilizer at $p_{ij}$ of the $G_1$-action is the cyclic group
\begin{equation*}
G'_{ij}=G_{ij}\cap G_1=\langle t_{ij}h_i^{r_{ij}}t_{ij}^{-1}\rangle,
\end{equation*} 
where 
\begin{equation*}
r_{ij}=\frac{\# G_{ij}}{\# G'_{ij}}.
\end{equation*}
Similarly to before, $t_{ij}h_i^{r_{ij}}t_{ij}^{-1}$ acts on $T_{p_{ij}}X\cong \bC$ by multiplication by $e^{2\pi\sqrt{-1}\cdot\# G'_{p_{ij}}}$.

Once we have chosen distinguished marked points $t_{ij}\cdot p_i$ for the cover $X\to X/G_1$, in order to obtain a point of the stack of admissible $G_1$-covers, one needs in addition to order the $G_1$-orbits of marked points, or equivalently, the marked points on the quotient $X/G_1$.

Therefore, we make the following definition:
\begin{defn}\label{restriction_hurwitz_defn}
Let $G_1\subset G$ be a subgroup, and, given a moduli space of admissible $G$-covers $\barH_{g,G,\xi}$, fix, for $i=1,2,\ldots,b$, a choice of representatives $\{t_{ij}\}$ of the orbits of the $G_1$-action on $G/\langle h_i\rangle$, and an ordering on the $t_{ij}$, ranging over all $i,j$. The \textbf{restriction morphism}
\begin{equation*}
\res^{G}_{G_1}:\barH_{g,G,\xi}\to\barH_{g,G_1,\xi'}
\end{equation*}
is defined by 
\begin{equation*}
\res^{G}_{G_1}([f:X\to X/G])=[f':X\to X/G_1],
\end{equation*}
with the new monodromy data $\xi'=\{h'_{ij}\}$ is given by the formula
\begin{equation*}
h'_{ij}=t_{ij}h_i^{r_{ij}}t_{ij}^{-1}
\end{equation*}
at the distinguished marked point $p_{ij}=t_{ij}\cdot p_{i1}$, where $r_{ij}$ is defined as above, and the marked points of $X/G_1$ (and correspondingly, the monodromy elements comprising $\xi'$) are ordered according to the ordering on the $t_{ij}$.
\end{defn}

We will refer to the ordered choice of $\{t_{ij}\}$ as \textbf{relabeling data}. As in the definition of the source maps $\phi$, we suppress this data from the notation, but we will need to keep track of it explicitly in many instances.

Taking $G_1=\{1\}$ recovers the source maps $\phi$. The composition of two restriction maps is again a restriction map, relative to appropriate relabeling data. In particular, the restriction maps $\res^{G}_{G_1}$ are compatible with source maps from both spaces to $\barM_{g,r}$, and the general restriction maps are also representable, finite, unramified, and local complete intersections.

Taking $G_1=G$, different choices of relabeling data give \textbf{relabeling isomorphisms}
\begin{equation*}
\res^{G}_{G}:\barH_{g,G,\xi}\to\barH_{g,G,\xi'}
\end{equation*}
sending an admissible $G$-cover $f:X\to Y$ to itself, but possibly changing the distinguished marked points in each $G$-orbit, conjugating the monodromy accordingly, and reordering the $G$-orbits. For general restriction maps, different relabeling data modify $\res^{G}_{G_1}$ by a relabeling isomorphism on the target.

We also define the operation of restriction on admissible $G$-graphs.

\begin{defn}\label{restriction_graphs}
Let $G_1\subset G$ be a subgroup, and let $(\Gamma,G)$ be an admissible $G$-graph associated to the space $\barH_{g,G,\xi}$. Fix relabeling data associated to $G_1\subset G$ and $\xi$. The \textbf{restriction} $((\Gamma,G),G_1)$ of $(\Gamma,G)$ to $G_1$ is defined by restricting the action of $G$ on $\Gamma$ to $G_1$, and relabeling the distinguished legs exactly as in Definition \S\ref{restriction_hurwitz_defn}, and recording the new monodromy data accordingly.
\end{defn}

We obtain an induced natural map of boundary strata $\res_{\Gamma}:\barH_{(\Gamma,G)}\to\barH_{((\Gamma,G),G_1))}$ defined by the product of the maps
\begin{equation*}
\xymatrix{
(\res_{\Gamma})_{v}:\barH_{g_v,G_v,\xi_v} \ar[r]^{\res^{G_v}_{G_v\cap G_1}} & \barH_{g_v,G_v\cap G_1,\xi'_v} \ar[r]^(0.36){\Delta} & \displaystyle\prod_{i=1}^{\frac{\#G/\#G_v}{\#G_1/(\#(G_1\cap G_v))}} \barH_{g_v,G_v\cap G_1,\xi'_v}
}
\end{equation*}
where the maps $\res^{G_v}_{G_v\cap G_1}$ are defined relative to relabeling data with that of $\res^{G}_{G_1}$ and $\Delta$ is the diagonal map. In particular, we get a commutative diagram
\begin{equation*}
\xymatrix@C+2pc{
\barH_{(\Gamma,G)} \ar[r]^{\xi_{(\Gamma,G)}} \ar[d]_{\res_{\Gamma}} & \barH_{g,G,\xi}  \ar[d]^{\res^{G}_{G_1}} \\
\barH_{((\Gamma,G),G_1)} \ar[r]^{\xi_{((\Gamma,G),G_1)}} & \barH_{g,G_1,\xi'}
}
\end{equation*}

In the case of corestriction, the situation is more straightforward. Suppose that $G_1\subset G$ is a normal subgroup. The fibers of $f:X\to X/G$ collapse into $G/G_1$-orbits upon passing to $f':X/G_1\to X/G$, and distinguished points of the fibers may be chosen to be the images of those in $X$. The monodromy data is given simply by taking the images of the $h_i\in G$ comprising $\xi$ in $G/G_1$. We summarize this in the following definition.

\begin{defn}
Let $G_1\subset G$ be a normal subgroup. The \textbf{corestriction} map
\begin{equation*}
\cores^{G}_{G/G_1}:\barH_{g,G,\xi}\to\barH_{g_1,G/G_1,\xi'}
\end{equation*}
is defined by 
\begin{equation*}
\cores^{G}_{G_1}([f:X\to X/G])=[f':X/G_1\to X/G].
\end{equation*}

Here, $g_1$ is the genus of $X/G_1$ (which may be computed by the Riemann-Hurwitz formula), the distinguished marked points of the fibers of $f'$ are given by the images of those in $f$, and $\xi'=(h'_1,\ldots,h'_b)$, where $h'_i$ is the image of $h_i$ under the quotient map $G\to G/G_1$.
\end{defn}

When $G_1=G$, we recover the target map $\delta$. The composition of two corestriction maps (induced by a composition of group surjections) is again a corestriction map. In particular, the corestriction maps $\cores^{G}_{G_1}$ are compatible with the target maps from both spaces to $\barM_{g',b}$, and the general corestriction maps are also flat, proper, and quasi-finite. Their degrees may be computed in terms of those of the two target maps $\delta$, which are in turn given by Hurwitz numbers.

As with restriction, we also define the notion of the corestriction of an admissible $G$-graph $(\Gamma,G)$.

\begin{defn}\label{corestriction_graphs}
Let $G_1\subset G$ be a normal subgroup and let $(\Gamma,G)$ be an admissible $G$-graph. We define the \textbf{corestriction} $(\Gamma/G_1,X/G_1)$ of $(\Gamma,G)$ to $G_1$ by the quotient graph $\Gamma/G_1$ with the induced $X/G_1$-action. The monodromy elements on the corestricted graph are the images of those from $\Gamma$ under the quotient quotient map.
\end{defn}

We obtain the natural map of boundary strata $\cores_{\Gamma}:\barH_{(\Gamma,G)}\to\barH_{(\Gamma/G_1,G/G_1)}$ defined by the product of
\begin{equation*}
\cores_{G_v\cap G_1}^{G_v}\barH_{g_v,G_v,\xi_v}\to\barH_{g_{1,v},G_v/(G_v\cap G_1),\xi'_v}
\end{equation*}
over $G$-orbit representatives $v\in V'(\Gamma)$, so that we have the commutative diagram
\begin{equation*}
\xymatrix@C+2pc{
\barH_{(\Gamma,G)} \ar[r]^{\xi_{(\Gamma,G)}} \ar[d]_{\cores_{\Gamma}}& \barH_{g,G,\xi} \ar[d]^{\cores^{G}_{G/G_1}}  \\
\barH_{(\Gamma/G_1,G/G_1)} \ar[r]^{\xi_{(\Gamma/G_1,G/G_1)}} & \barH_{g_1,G/G_1,\xi'}
}
\end{equation*}

We may combine restriction and corestriction morphisms in the following general setting. 

\begin{defn}\label{res_cores_defn}
Suppose that we have a nested pair of subgroups $G_1\subset G_2\subset G$, with $G_1$ normal in $G_2$. We define the \textbf{restriction-corestriction morphism}
\begin{equation*}
q_{G_1,G_2,G}:\barH_{g,G,\xi}\to\barH_{g_1,G_2/G_1,\xi'},
\end{equation*} 
by
\begin{equation*}
q_{G_1,G_2,G}=\cores^{G_2}_{G_2/G_1}\circ\res^{G}_{G_2},
\end{equation*}
where the restriction map depends on relabeling data.
\end{defn}
On the level of maps of curves, we have
\begin{equation*}
q_{G_1,G_2,G}([X\to X/G])=[X/G_1\to X/G_2].
\end{equation*}
The composition of two restriction-corestriction morphisms is a single such.

An important example for us will be the case $G=S_d$, $G_1=G_2=S_{d-1}$, which recovers the case of loci of Harris-Mumford admissible covers on $\barM_{g,n}$, see \S\ref{harris_mumford_section}.

\subsection{Diagonals and non-tautological cycles from $\barH_{g,G,\xi}$}\label{non-taut_section}

Our goal is to construct a framework including classical tautological classes that also incorporates the cycles obtained from source maps $\phi$, and from restriction-corestriction morphisms more generally. As we explain here, \textit{diagonal} morphisms also play an essential role in the theory. In fact, we have already seen diagonals appear naturally in Remark \ref{boundary_relate_to_classical} and Definition \ref{restriction_graphs}.

We now recall the example of Graber-Pandharipande \cite{gp} of a non-tautological cycle arising from admissible $G$-covers, further illustrating the role of diagonal maps.

Let $h\ge1$ be an integer and let $\barH=\barH_{2h,\bZ/2\bZ,(1)^2}$ denote the moduli space of admissible $\bZ/2\bZ$-covers $f:X\to Y$, where $X,Y$ have genus $2h,h$, respectively, and $f$ is ramified over two marked points. We have that $\barH$ is a smooth stack of dimension $3h-1$. We have, as usual, a source morphism 
\begin{equation*}
\phi:\barH\to\barM_{2h,2}.
\end{equation*}
We will abusively refer to the map 
\begin{equation*}
\phi:\barH\to\barM_{2h}
\end{equation*}
obtained by forgetting the two ramification points by the same name.

Let $\Gamma$ be the stable graph of genus $2h$ consisting of two genus $h$ vertices connected by a single edge, with associated boundary map
\begin{equation*}
\xi_{\Gamma}:\barM_{h,1}\times\barM_{h,1}\to\barM_{2h}.
\end{equation*}

\begin{lem}\cite[Lemma 1]{gp}
The class $\xi_{\Gamma}^{*}(\phi_{*}(1))$ is a non-zero multiple of the class of the diagonal
\begin{equation*}
\Delta:\barM_{h.1}\to\barM_{h,1}\times\barM_{h,1}.
\end{equation*}
\end{lem}

More explicitly, the only admissible $\bZ/2$-cover $f:X\to Y$ whose source lies in the boundary stratum $\barM_{\Gamma}$ has $Y=C\cup\bP^1$, where $C$ is a genus $h$ curve, and $X=C\cup\bP^1\cup C$, where the two copies of $C$ in $X$ map isomorphically to that in $Y$, and the rational component of $X$ has degree 2 over that of $Y$.

The existence of odd singular cohomology in $\barM_{h,1}$ for sufficiently large $h$, as shown by Pikaart \cite{pikaart}, implies that $\phi_{*}(1)\in H^{*}(\barM_{2h})$ cannot have a tautological K\"{u}nneth composition upon pullback by $\xi_{\Gamma}$. One may then conclude:

\begin{prop}\cite[\S 2]{gp}
For all sufficiently large $h$, we have $\phi_{*}(1)\notin RH^{*}(\barM_{2h})$.
\end{prop}

Similar constructions of non-tautological cycles, including those which already fail to be tautological in Chow, use the existence of odd middle cohomology in $\barM_{1,11}$, see \cite[Theorem 2]{gp} and \cite{vanzelm}.

In our framework, the computation of Graber-Pandharipande shows that, in order for $\phi_{*}(1)$ to be included in a $\bQ$-subalgebra of $\barM_{2h}$ containing all tautological classes, and more generally, in a system of such $\bQ$-subalgebras for all spaces of admissible $G$-covers, this system of $\bQ$-subalgebras should also include classes coming from general diagonal morphisms
\begin{equation*}
\Delta:\barH_{g,G,\xi}\to\barH_{g,G,\xi}\times\barH_{g,G,\xi}.
\end{equation*}
The target of this morphism is a product of isomorphic admissible $G$-cover spaces, and it will therefore be convenient for us to work with \textit{arbitrary} products of such spaces.

\subsection{The $\cH$-tautological ring}

We are now ready to make the main definition of this paper.

\begin{defn}\label{main_defn}
The \textbf{$\cH$-tautological ring} is the smallest system of $\bQ$-subalgebras  
\begin{equation*}
R^{*}_{\cH}\left(\prod_{j=1}^{m}\barH_{g_j,G_j,\xi_j}\right)\subset A^{*}\left(\prod_{j=1}^{m}\barH_{g_j,G_j,\xi_j}\right),
\end{equation*}
ranging over all finite products of moduli spaces of admissible $G$-covers, that contains all $\psi$ and $\kappa$ classes and is closed under pushfoward by all maps of the form
\begin{equation*}
\gamma_1\times\cdots\times\gamma _m:\mathcal{H}_1\times\cdots\times \mathcal{H}_m\to\mathcal{H}'_1\times\cdots\times \mathcal{H}'_m
\end{equation*}
where each $\mathcal{H}_i,\mathcal{H}'_i$ is itself a product of moduli spaces of admissible $G$-covers and each map $\gamma_i:\mathcal{H}_i\to \mathcal{H}'_i$ is either a boundary, forgetful, restriction-corestriction, or diagonal map, as defined earlier.
\end{defn}

\section{Bivariant pullbacks of $\cH$-tautological classes}\label{int_theory_algos}

We now build up the theory of intersections in the $\cH$-tautological ring by explaining the algorithmic computation of bivariant pullbacks of the basic $\cH$-tautological classes.

In all of the following results (except those of \S\ref{pullback_psi_kappa_section}), we will consider the intersection of two tautological maps $\gamma_j:\cH_{j}\to\cH$, and give a commutative (not always Cartesian) diagram
\begin{equation*}
\xymatrix{
\cH_0 \ar[r]^{\gamma'_1} \ar[d]^{\gamma'_2}& \cH_{1} \ar[d]^{\gamma_1} \\
\cH_2 \ar[r]^{\gamma_2} & \cH
}
\end{equation*}
such that the maps $\gamma'_j$ are compositions of tautological morphisms and the classes $(\gamma_2)^{*}(\gamma_1)_{*}(1)$ and $(\gamma_1)^{*}(\gamma_2)_{*}(1)$ may both be expressed as the pushforward of an $\cH$-tautological class on $\cH_0$.

\subsection{Pullbacks by forgetful maps}
Here, we record diagrams expressing the pullbacks of certain $\cH$-tautological (bivariant) classes by a forgetful morphism.

\subsubsection{Forgetful maps}
The following is a straightforward analog of Lemma \ref{double_forgetful_mgn}.
\begin{prop}
The commutative diagram
\begin{equation*}
\xymatrix{
\barH_{g,G,\xi+\{1\}+\{1'\}} \ar[r]^(0.56){\pi_{1'}} \ar[d]^{\pi_1} & \barH_{g,G,\xi+\{1\}}\ar[d]^{\pi}\\
\barH_{g,G,\xi+\{1'\}} \ar[r]^(0.53){\pi} & \barH_{g,G,\xi}
}
\end{equation*}
is Cartesian up to birational equivalence. In particular, the bivariant pullback of a forgetful map by itself is again a forgetful map.
\end{prop}

\subsubsection{Boundary maps} 

Suppose $(A,G)$ is an admissible $G$-graph indexing a boundary stratum of $\barH_{g,G,\xi}$ and let $v\in V(A)$ be a vertex. Let $A+\ell_{v}$ be the admissible $G$-graph obtained from $A$ by adding a distinguished marked leg with trivial monodromy to $v$, as well as a corresponding $G$-orbit of legs to the vertices in the $G$-orbit of $v$. This operation corresponds to adding a distinguished marked orbit to a stable $G$-curve with underlying admissible $G$-graph $A$.

\begin{prop}\label{boundary_forgetful}
With notation as above, the commutative diagram
\begin{equation*}
\xymatrix{
\displaystyle\coprod_{v\in V(A)}\barH_{(A+\ell_{v},G)} \ar[r]^{\xi_{(A+\ell_{v},G)}} \ar[d]_{\pi_A} & \barH_{g,G,\xi+\{1\}} \ar[d]^{\pi}\\
\barH_{(A,G)} \ar[r]^{\xi_{(A,G)}} & \barH_{g,G,\xi}
}
\end{equation*}
is Cartesian up to birational equivalence, and expresses the bivariant pullback of a forgetful map by a boundary map, and vice versa.
\end{prop}
\begin{proof}
The data of an extra orbit of unramified points with a distinguished marked point on a $G$-curve $X$ is equivalent to the data of the choice of a single unramified point. When $X$ has an $A$-structure, this amounts to the choice of a single unramified point on one of the components of $X$, which corresponds to the addition of a leg at a unique vertex of A if the dual graph of $X$ is in fact isomorphic to $A$. Thus, we get an inverse to $(\xi_{(A+\ell_{v},G)},\pi_A)$ defined on the dense open locus $\cH_{(A,G)}\subset\barH_{(A,G)}$ of $G$-curves with dual graph isomorphic to $A$.
\end{proof}

\subsubsection{Diagonal maps} 
\begin{prop}
We have the following Cartesian diagram, expressing the bivariant pullback of a diagonal map by a forgetful map, and vice versa:
\begin{equation*}
\xymatrix{
\barH_{g,G,\xi+\{1\}} \ar[r]^(0.38){(\id,\pi)} \ar[d]^{\pi} & \barH_{g,G,\xi+\{1\}}\times\barH_{g,G,\xi} \ar[d]^{\pi\times\id}\\
\barH_{g,G,\xi} \ar[r]^(0.38){\Delta} & \barH_{g,G,\xi}\times\barH_{g,G,\xi}
}
\end{equation*}
\end{prop}
\begin{proof}
Clear. Note in addition that
\begin{equation*}
(\id,\pi)=(\id\times\pi)\circ\Delta,
\end{equation*}
so in particular the top arrow is the composition of tautological morphisms.
\end{proof}

\subsubsection{Restriction-corestriction maps} 
\begin{prop}
We have the following Cartesian diagram, expressing the bivariant pullback of a restriction map by a forgetful map, and vice versa:
\begin{equation*}
\xymatrix{
\barH_{g,G,\xi+\{1\}} \ar[d]^{\pi_G}  \ar[r]^{\wt{\res}^{G}_{G_1}} & \barH_{g,G_1,\xi'+\{1\}} \ar[d]^{\pi_{G_1}}\\
\barH_{g,G,\xi} \ar[r]^{\res^{G}_{G_1}} & \barH_{g,G_1,\xi'}
}
\end{equation*}
The map $\wt{\res}^{G}_{G_1}$ is defined by restricting to $G_1$ (with relabeling data compatible with that of the bottom arrow), then forgetting all marked points in the new $G$-orbit except those in the $G_1$-orbit of the distinguished one.
\end{prop}

\begin{proof}
The inverse functor from the fiber product to $\barH_{g,G,\xi+\{1\}}$ may be described as follows. Given a $G$-cover of $\barH_{g,G,\xi}$ and an additional distinguished marked point as dictated by $\barH_{g,G_1,\xi'+\{1\}}$, one takes the same distinguished marked point on the $G$-cover to get a point of $\barH_{g,G,\xi+\{1\}}$.
\end{proof}

We now consider the fiber product of a corestriction map with a forgetful map. Here, we find a simple example of a phenomenon that will occur in several other intersections later: we do not express the fiber product directly in terms of our spaces of admissible $G$-covers, but only up to a finite cover.

\begin{lem}\label{cores_forgetful_dominan_diagram}
We have a commutative diagram
\begin{equation*}
\xymatrix{
\barH_{g,G,\xi+\{1\}} \ar[d]^{\pi_G}  \ar[r]^(0.38){\wt{\cores}_{G/G_1}^{G}} & \barH_{g_1,G/G_1,\xi'+\{1\}} \ar[d]^{\pi_{G/G_1}}\\
\barH_{g,G,\xi} \ar[r]^(0.45){\cores_{G/G_1}^{G}} & \barH_{g',G/G_1,\xi'}
}
\end{equation*}
for which the induced map
\begin{equation*}
(\pi_G,\cores^{G}_{G/G_1}):\barH_{g,G,\xi+\{1\}}\to\barH_{g,G,\xi}\times_{\barH_{g',G/G_1,\xi'}}\barH_{g_1,G/G_1,\xi'+\{1\}}
\end{equation*}
is proper and quasi-finite of degree $\#G_1$.
\end{lem}
\begin{proof}
The commutativity is clear. All spaces in question are proper, hence so is the induced map to the fiber product. 

A point of $\barH_{g,G,\xi}\times_{\barH_{g',G/G_1,\xi'}}\barH_{g_1,G/G_1,\xi'+\{1\}}$ consists of a pointed $G$-cover $f:X\to X/G$ along with a distinguished unramified point of the intermediate quotient $x\in X/G_1$. Then, the $\#G_1$ pre-images of $x$ in $X$ correspond to the $\#G_1$ points of $\barH_{g,G,\xi+\{1\}}$ lying over our target point.
\end{proof}

It follows immediately that:

\begin{prop}
We have:
\begin{equation*}
(\cores_{G/G_1}^{G})^{*}([\pi_{G/G_1}])=\frac{1}{\#G_1}\cdot[\pi_G]
\end{equation*}
and
\begin{equation*}
\pi_{G/G_1}^{*}((\cores_{G/G_1}^{G})_{*}(1))=\frac{1}{\#G_1}\cdot(\cores_{G/G_1}^{G})_{*}(1)
\end{equation*}
where we use brackets to denote bivariant classes.
\end{prop}

\begin{rem}
Because the corestriction maps are flat, proper, and quasi-finite, the classes 
\begin{equation*}
(\cores_{G/G_1}^{G})_{*}(1),(\cores_{G/G_1}^{G})_{*}(1)
\end{equation*}
are equal to multiples of the fundamental class, so the second formula can be viewed simply as expressing an equality of numbers. However, in order later to be able to carry out intersections of compositions of basic tautological morphisms, we need to express the intersection of the morphisms $\cH_1\to\cH$ and $\cH_2\to\cH$ in terms of a class coming from a space $\cH_0$ admitting tautological maps to \textit{both} $\cH_j$, as we have here; hence, it is not enough simply to remember the correct multiple of the fundamental class. This issue will arise in all of our computations of this section involving corestriction maps.
\end{rem}

\subsection{Intersections of two boundary strata}\label{boundary_strata_intersection_section}

The intersection of boundary classes on spaces of admissible $G$-covers is entirely analogous to the classical situation, see \S\ref{boundary_intersection_classical_section}.

\begin{lem}\label{boundary_intersection_hurwitz_cartesian}
Let $(A,G),(B,G)$ be admissible $G$-graphs indexing boundary strata on $\barH_{g,G,\xi}$. We have a Cartesian diagram
\begin{equation*}
\xymatrix{
\displaystyle\coprod\barH_{(\Gamma,G)} \ar[d]^{\xi_{\Gamma\to A}} \ar[r]^{\xi_{\Gamma\to B}} & \barH_{(B,G)} \ar[d]^{\xi_{(B,G)}}\\
\barH_{(A,G)} \ar[r]^{\xi_{(A,G)}} & \barH_{g,G,\xi}
}
\end{equation*}
where the disjoint union is over admissible $G$-graphs $(\Gamma,G)$ with maps $(\Gamma,G)\to (A,G),(B,G)$ with the genericity condition that the induced map on edges
\begin{equation*}
E(A)\sqcup E(B)\to E(\Gamma)
\end{equation*}
is surjective. 
\end{lem}

\begin{proof}
The commutativity is immediate. We describe an inverse functor
\begin{equation*}
\barH_{(A,G)}\times_{\barH_{g,G,\xi}}\barH_{(B,G)}\to\coprod\barH_{(\Gamma,G)}.
\end{equation*}
The data of a point on the source is equivalent to that of a pointed $G$-cover $f:X\to Y$ with the required monodromy data along with maps $f_A:(\Gamma_f,G)\to(A,G)$ and $f_B:(\Gamma_f,G)\to(B,G)$, where $(\Gamma_f,G)$ is the canonical admissible $G$-graph structure on the dual graph of $X$.

We define $(\Gamma,G)$ to be the admissible $G$-graph obtained from $\Gamma_f$ contracting all edges not coming from $E(A)$ or $E(B)$, and modifying the genus function accordingly. Then, the maps $f_A,f_B$ factor respectively as
\begin{align*}
(\Gamma_f,G)&\to(\Gamma,G)\to(A,G),\\
(\Gamma_f,G)&\to(\Gamma,G)\to(B,G)
\end{align*}
where the natural map $(\Gamma_f,G)\to(\Gamma,G)$ defines, along with the data of $f:X\to Y$, a point of $\barH_{(\Gamma,G)}$. This furnishes the desired inverse.
\end{proof}

\begin{lem}\label{boundary_excess_intersection_hurwitz}
We have the following formula for the excess bundle on $\barH_{(H,G)}$:
\begin{equation*}
\xi^{*}_{\Gamma\to A}N_{\barH_{(A,G)}/\barH_{g,G,\xi}}/N_{\barH_{(\Gamma,G)}/\barH_{(B,G)}}\cong\bigoplus_{(\ell,\ell')\in (E(A)\cap E(B))'}T_{\ell}\otimes T_{\ell'},
\end{equation*}
where $(\ell,\ell')$ denotes the pair of half-edges comprising an edge and $(E(A)\cap E(B))'$ is any sets of $G$-orbit representatives in $E(A)\cap E(B)$
\end{lem}

\begin{proof}
This follows from Lemma \ref{normal_bundle_hurwitz}.
\end{proof}

Applying the excess intersection formula as in the classical case yields:

\begin{prop}\label{boundary_intersection_formula_hurwitz}
Let $(A,G),(B,G)$ be stable graphs associated to $\barH_{g,G,\xi}$. We have:
\begin{equation*}
\xi^{*}_{(A,G)}(\xi_{(B,G)*}(1))=\sum_{(\Gamma,G)}(\xi_{(\Gamma,G)\to (A,G)})_{*}\left(\prod_{(\ell,\ell')\in (E(A)\cap E(B))'}(-\psi_{h}-\psi_{h'})\right),
\end{equation*}
where the sum is over $(\Gamma,G)$ as above.
\end{prop}

\subsection{Intersections of restriction-corestriction maps with boundary strata}\label{res-cores_boundary}

\subsubsection{Restriction maps}\label{restriction_boundary_intersection_section}

We consider the intersection of a restriction map $\res^{G}_{G_1}:\barH_{g,G,\xi}\to\barH_{g,G_1,\xi'}$ with a boundary class on $\barH_{g,G_1,\xi'}$. The case of $\res^{G}_{\{1\}}=\phi$ is the main technical result of Schmitt-van Zelm, \cite[Theorem 4.9]{schmittvanzelm}; the situation here is similar.

\begin{lem}\label{restriction_boundary_cartesian}
Suppose that $(A,G_1)$ is an admissible $G_1$-graph. Then, we have a Cartesian diagram
\begin{equation*}
\xymatrix{
\displaystyle\coprod\barH_{(\Gamma,G)} \ar[d]_{\res_\alpha} \ar[r]^(0.52){\xi_{(\Gamma,G)}} & \barH_{g,G,\xi} \ar[d]^{\res_{G_1}^{G}}\\
\barH_{(A,G_1)} \ar[r]^{\xi_{(A,G_1)}} & \barH_{g,G_1,\xi'}
}
\end{equation*}
Here, the coproduct is over admissible $G$-graphs $(\Gamma,G)$ equipped with a map of admissible $G_1$-graphs $\alpha:((\Gamma,G),G_1)\to (A,G_1)$, satisfying the genericity condition that the map
\begin{equation*}
\alpha_E:E(A)\to E(\Gamma)
\end{equation*}
is surjective.

The map $\res_\alpha$ is defined by the composition 
\begin{equation*}
\xymatrix{
\barH_{(\Gamma,G)} \ar[r]^(0.45){\res_{\Gamma}} & \barH_{(\Gamma,G),G_1} \ar[r]^{\xi_{\Gamma\to A}} & \barH_{(A,G_1)}.
}
\end{equation*}
where $((\Gamma,G),G_1)$ is the admissible $G_1$-graph obtained by restricting the $G$-action to $G_1$, as in Definition \ref{restriction_graphs}, and the relabeling data defining the restriction map is the same as that of the given map $\res^G_{G_1}$.
\end{lem}

\begin{proof}
We describe the inverse functor
\begin{equation*}
\barH_{(A,G_1)}\times_{\barH_{g,G_1,\xi'}}\barH_{g,G,\xi}\to \coprod\barH_{(\Gamma,G)}.
\end{equation*}
The construction is exactly analogous to that in the proof of \cite[Proposition 4.3]{schmittvanzelm}. For brevity's sake, we will describe the functor on the level of points without making explicit reference to families of objects, but the entire construction works in this setting.

On the source, we are given a $G$-cover $f:X\to X/G$ along with an $(A,G_1)$-structure on $X$, with appropriate marked points distinguished. 

We construct an admissible $G$-graph $(\Gamma,G)$ as follows. We define $E(\Gamma)$ to be the set of nodes of $X$ in the $G$-orbit of some node corresponding to an edge of $A$. Similarly, $H(\Gamma)$ is defined to be the set of pre-images of such nodes in the normalization $\wt{X}$ of $X$. The vertex set $V(\Gamma)$ then corresponds to the connected components of $X-E(\Gamma)$. The set of legs $L(A)$ naturally has a $G$-action coming from $X$. 

It is now straightforward to check that $V(\Gamma),H(\Gamma),L(\Gamma)$ taken together with the data $g,\iota,a,\zeta$ induced in the obvious way defines the desired graph $(\Gamma,G)$ with a map $\alpha:((\Gamma,G),G_1)\to (A,G_1)$ satisfying the needed genericity condition. Moreover, $f:X\to X/G$ with the $(\Gamma,G)$-structure on its dual graph defines a point of $\barH_{(\Gamma,G)}$, furnishing the needed inverse functor.
\end{proof}

The excess bundle on each $\barH_{(\Gamma,G)}$ also admits a similar description as in \cite[\S 4.2]{schmittvanzelm}. 

\begin{lem}
The excess bundle on $\barH_{(\Gamma,G)}$ is
\begin{equation*}
\res_{\alpha}^{*}N_{\barH_{(A,G_1)}/\barH_{g,G_1,\xi'}}/N_{\barH_{(\Gamma,G)}/\barH_{g,G,\xi}}\cong\bigoplus_{(\ell,\ell')}T_{\ell}\otimes T_{\ell'}
\end{equation*}
where the sum is over edges $(\ell,\ell')$ ranging over a set of $G_1$-orbit representatives in the image of $\alpha_E:E(A)\to E(\Gamma)$, excluding a choice of $G$-orbit representatives among the chosen edges.
\end{lem}

\begin{proof}
The normal bundle of $\barH_{(A,G_1)}$ in $\barH_{g,G_1,\xi'}$ is the direct sum of line bundle contributions from the $G_1$-orbits of edges of $A$, and the normal bundle of $\barH_{(\Gamma,G)}$ in $\barH_{g,G,\xi}$ is the direct sum of line bundle contributions of $G$-orbits of edges of $\Gamma$. Removing a choice of the latter set from the former yields the desired formula.
\end{proof}

Because the maps $\xi_{(A,G_1)}$ and $\res_{G_1}^{G}$ are both unramified, we get the same excess bundle upon pushing forward and pulling back in either order. We conclude the following.
\begin{prop}
We have:
\begin{equation*} 
\xi^{*}_{(A,G_1)}((\res_{G_1}^{G})_{*}(1))=\sum_{(\Gamma,G),\alpha}(\res_{\alpha})_{*}\left(\prod_{(\ell,\ell')}(-\psi_{\ell}-\psi_{\ell'})\right),
\end{equation*}
where the product is over a set of edges $(\ell,\ell')$ of $\Gamma$ chosen as described above, as well as
\begin{equation*} 
(\res_{G_1}^{G})^{*}((\xi_{(A,G_2)})_{*}(1))=\sum_{(\Gamma,G),\alpha}(\xi_{(\Gamma,G)})_{*}\left(\prod_{(\ell,\ell')}(-\psi_{\ell}-\psi_{\ell'})\right).
\end{equation*}
\end{prop}
Taking $G_1=1$ recovers \cite[Theorem 4.9]{schmittvanzelm}. The $\psi$ classes are invariant under the $G$-action on the marked points of $\barH_{g,G,\xi}$, so the formulas above are independent of the choices of orbit representatives. In fact, we may write
\begin{equation*}
\prod_{(\ell,\ell')}(-\psi_{\ell}-\psi_{\ell'})=\prod_{(\ell,\ell')\in E'(\Gamma)}(-\psi_{\ell}-\psi_{\ell'})^{k_{(\ell,\ell')}-1},
\end{equation*}
where the product on right hand side is over any choice of $G$-orbit representatives of $E(\Gamma)$, and
\begin{equation*}
k_{(\ell,\ell')}=\frac{\#[\{G\cdot {(\ell,\ell')}\}\cap\im(\alpha_{E})]}{\# G_1}.
\end{equation*}
The surjectivity of $\alpha_E$ onto $G$-orbits guarantees that $k_{(\ell,\ell')}\ge1$.

\subsubsection{Corestriction maps}\label{corestriction_boundary_intersection_section}

We now consider the case of corestriction maps $\cores_{G/G_1}^{G}$, where $G_1\subset G$ is a normal subgroup. Suppose that $(A,G/G_1)$ is an admissible $G/G_1$-graph. Consider the set of admissible $G$-graphs $(\Gamma,G)$ together with an \textbf{isomorphism} of admissible $G/G_1$-graphs $\alpha:(\Gamma/G_1,G/G_1)\to(A,G/G_1)$. For such an $\alpha$, let $\Aut(\alpha)$ be the group of automorphisms of $(\Gamma,G)$ inducing the identity automorphism on $\Gamma/G_1\cong A$, that is, the kernel of map $\Aut((\Gamma,G))\to\Aut((A,G/G_1))$.

\begin{lem}\label{cores_boundary_diagram}
We have a commutative diagram 
\begin{equation*}
\xymatrix{
\displaystyle\coprod\barH_{(\Gamma,G)} \ar[d]_{\cores_{\alpha}} \ar[r]^{\xi_{(\Gamma,G)}} & \barH_{g,G,\xi} \ar[d]^{\cores_{G/G_1}^{G}}\\
\barH_{(A,G/G_1)} \ar[r]^{\xi_{A}} & \barH_{g_1,G/G_1,\xi'}
}
\end{equation*}
where the disjoint union is over admissible $G$-graphs $(\Gamma,G)$ together with isomorphisms $\alpha:(\Gamma/G_1,G/G_1)\to(A,G/G_1)$ as above.

The induced map of groupoids (\textbf{on the level of closed points})
\begin{equation*}
\displaystyle\coprod\barH_{(\Gamma,G)}(\Spec(\bC))\to\barH_{g,G,\xi} \times_{\barH_{g_1,G/G_1,\xi'}} \barH_{(A,G/G_1)}(\Spec(\bC))
\end{equation*}
is surjective, and the fiber over any point of the target consists of $\#\Aut(\alpha)$ points of a unique component $\barH_{(\Gamma,G)}$ coming from a pair $((\Gamma,G),\alpha)$.
\end{lem}

Recall from Definition \ref{corestriction_graphs} that $(\Gamma/G_1,G/G_1)$ denotes the corestriction of $\Gamma$ to $G/G_1$. The map $\cores_{\alpha}$ is induced as the composition
\begin{equation*}
\xymatrix{
\barH_{(\Gamma,G)} \ar[r]^(0.38){\cores_{\Gamma}} & \barH_{(\Gamma/G_1,G/G_1)} \ar[r]^(0.54){\xi_{\Gamma\to A}} & \barH_{(A,G/G_1)}
}.
\end{equation*}

Above, we refer to $\barH_{(\Gamma,G)}$ as a ``component'' without claiming its connectedness, but rather to specify a single boundary stratum equipped with an isomorphism of $G/G_1$-graphs $\alpha$.

We remark that the map $\Aut((\Gamma,G))\to\Aut((A,G/G_1))$ induced by $\alpha$ is not always surjective. For example, suppose $G_1=G=\bZ/6\bZ$, the graph $A$ has a single edge connecting two vertices of genus 1, and $(\Gamma,G)$ is the bipartite graph $K_{2,3}$. The $G$-action on $\Gamma$ is such that $E(\Gamma)$ forms a $G$-torsor, and a vertex of valence $r$ is stabilized by $r\bZ/6\bZ$. Then, $\cH_{(\Gamma,G)}$ is a boundary divisor of $\barH_{7,\bZ/6\bZ,\emptyset}$ dominating the boundary divisor $\barM_A$ of $\barM_2$ via the target map. Now, we have $\Aut((A,G/G_1))\cong\bZ/2\bZ$, but the map $\Aut((\Gamma,G))\to\Aut((A,G/G_1))$ induced by $\alpha$ is trivial.

\begin{proof}
The commutativity of the diagram is clear. Consider a fiber of the map
\begin{equation*}
\displaystyle\coprod\barH_{(\Gamma,G)}(\Spec(\bC)) \to \barH_{g,G,\xi} \times_{\barH_{g_1,G/G_1,\xi'}} \barH_{(A,G/G_1)}(\Spec(\bC)).
\end{equation*}

A point of the target consists of a stable $G$-curve $X$ along with an $A$-marking on its partial quotient $X/G_1$. One then constructs the graph $\Gamma$ by taking $E(\Gamma)$ to be the set of pre-images of the nodes corresponding to $E(A)$ under the map $X\to X/G_1$, with half-edges given by the pre-images of the individual branches at such nodes. The vertex set $V(\Gamma)$ is then the set of connected components of $X-E(\Gamma)$, and the legs of $\Gamma$ corresponding to the marked points on $X$ are attached accordingly. One easily checks that $\Gamma$ has the natural structure of an admissible $G$-graph along with an isomorphism $\alpha:(\Gamma/G_1,G/G_1)\to(A,G/G_1)$.

On the other hand, we have multiple choices of a point of $\barH_{(\Gamma,G)}(\Spec(\bC))$ lying over our chosen point of $\barH_{g,G,\xi} \times_{\barH_{g_1,G/G_1,\xi'}} \barH_{(A,G/G_1)}(\Spec(\bC))$, owing to different choices of distinguished marked points in the half-edge orbits. Such choices naturally form a torsor under the action of $\Aut(\alpha)$, yielding the last claim.
\end{proof}

However, we note that the above local inverse construction cannot work in families as it does, for instance, in Lemmas \ref{boundary_intersection_hurwitz_cartesian} and \ref{restriction_boundary_cartesian}, where the morphisms being intersected are unramified.

Indeed, suppose for example that $A$ contains exactly one orbit of edges, and that the monodromy is non-trivial, so that $\barH_{(A,G/G_1)}$ parametrizes $G/G_1$-covers whose target generically has one node, and with non-trivial ramification over that node. Then, the map $\cores^{G}_{G/G_1}$ is ramified over the image of $\barH_{(A,G/G_1)}$, so the fiber product must be non-reduced. 

In particular, let $S=\Spec(\bC[\epsilon]/\epsilon^2)$. There exists a family of admissible $G$-covers $f:X\to Y$ over $S$ smoothing to first order, such that the induced deformation of $X/G_1$ has an $A$-structure over all of $S$ (that is, the nodes of $X/G_1$ corresponding to the edge of $A$ are not smoothed over $S$). Correspondingly, we have a non-constant map $S\to\barH_{g,G,\xi}$ along with a constant map $S\to\barH_{(A,G/G_1)}$ agreeing on the constant map $S\to\barH_{g_1,G/G_1,\xi'}$. However, because $A$ has only one edge, we have that $\cores_{\alpha}$ is unramified, so we see that there will be no simultaneous lift $S\to\barH_{(\Gamma,G)}$.

On the other hand, we have the following:

\begin{lem}\label{cores_boundary_multiplicities}
The multiplicity of a component in the image of the map
\begin{equation*}
(\xi_{(\Gamma,G)},\cores_{\alpha}):\barH_{(\Gamma,G)}\to\barH_{g,G,\xi} \times_{\barH_{g_1,G/G_1,\xi'}} \barH_{(A,G/G_1)}
\end{equation*}
is
\begin{equation*}
\prod_{(\ell,\ell')\in E'(A)}\frac{\ord_{G}(h_{\wt{\ell}})}{\ord_{G/G_1}(h_{\ell})}.
\end{equation*}
where $E'(A)\subset E(A)$ is a set of $G/G_1$-orbit representatives, and $\wt{\ell}$ is any half-edge of $\Gamma$ lying over $\alpha_L(\ell)\in L(\Gamma/G_1)$.
\end{lem}

We remark also that the terms in the denominators depend only on $(A,\Gamma/G_1)$, but the terms in the numerator depend on $(\Gamma,G)$.

In fact, we will obtain more precise information, that the multiplicity in the tangent direction on $\barH_{(\Gamma,G)}$ in the direction corresponding to the edge $(\ell,\ell')$ is 
\begin{equation*}
\frac{\ord_{G}(h_{\wt{\ell}})}{\ord_{G/G_1}(h_{\ell})}=\frac{\ord_{G}(h_{\wt{\ell'}})}{\ord_{G/G_1}(h_{\ell'})}.
\end{equation*}

\begin{proof}
We compute the local picture of the intersection of $\xi_{(\Gamma,G)}$ and $\cores_{\alpha}$. 

The complete local ring of $\barH_{g,G,\xi}$ at $[X\to X/G]$ is isomorphic to
\begin{equation*}
\bC[[t_1,\ldots,t_m,t_{m+1},\ldots,t_{3g'-3+b}]],
\end{equation*}
where we may take $t_1,\ldots,t_m$ to be smoothing parameters for the $G$-orbits of nodes of $X$ coming from edges of $A$. The map on deformation spaces induced by $\xi_{(\Gamma,G)}$ is then the quotient by the ideal $(t_1,\ldots,t_m)$, by the surjectivity of $\alpha_E$.

Similarly, the complete local ring of $\barH_{g_1,G/G_1,\xi'}$ at $[X/G_1\to X/G]$ is isomorphic to 
\begin{equation*}
\bC[[t'_1,\ldots,t'_{m},t'_{m+1},\ldots,t'_{3g'-3+b}]],
\end{equation*}
where $t'_1,\ldots,t'_m$ are smoothing parameters for the $G/G_1$-orbits of nodes of $X/G_1$, and the map on deformation spaces induced by $\xi_{(A,G)}$ is given by the quotient by the ideal $(t'_1,\ldots,t'_m)$.

For $j=1,2,\ldots,m$, we have that $t_j^{e_j}$ is a smoothing parameter for the corresponding node $y_j\in X/G$, where $e_j$ is the common ramification index of $X\to X/G$ above $y_j$. On the other hand, ${t'_{j}}^{e'_{j}}$ is also a smoothing parameter at $y_j$, where $e'_j$ is the common ramification index of $X/G_1\to X/G$ above $y_j$. Therefore, up to units, we have
\begin{equation*}
{t'_j}^{e'_j}={t_j}^{e_j}.
\end{equation*}

Thus, the map 
\begin{equation*}
\bC[[t'_1,\ldots,t'_{m},t_{m+1},\ldots,t_{3g'-3+b}]]\to\bC[[t_1,\ldots,t_m,t'_{m+1},\ldots,t'_{3g'-3+b}]]
\end{equation*}
induced by $\cores_{G/G_1}^{G}$ decomposes into the map
\begin{equation*}
\bC[[t'_1,\ldots,t'_{m}]]\to\bC[[t_1,\ldots,t_m]]
\end{equation*}
taking (up to units) ${t'_j}\mapsto {t_j}^{e_j/e'_j}$, and the map
\begin{equation*}
\bC[[t_{m+1},\ldots,t_{3g'-3+b}]]\to\bC[[t'_{m+1},\ldots,t'_{3g'-3+b}]]
\end{equation*}
induced on complete local rings by $\cores_{\alpha}:\barH_{(\Gamma,G)}\to\barH_{(A,G/G_1)}$.

It follows that, on the level of complete local rings, the map $(\xi_{(\Gamma,G)},\cores_{\alpha})$ is of the form
\begin{equation*}
\bC[[t_1,\ldots,t_m]]/(t_1^{e_1/e'_1},\ldots,t_m^{e_m/e'_m})\otimes_{\bC}\bC[[t_{m+1},\ldots,t_{3g'-3+b}]]\to\bC[[t_{m+1},\ldots,t_{3g'-3+b}]],
\end{equation*}
where the variables $t_1,\ldots,t_m$ map to 0 and $t_{m+1},\ldots,t_{3g'-3+b}$ to themselves. The desired conclusion follows.
\end{proof}

We have therefore proven the following:
\begin{prop}
With notation as above, we have
\begin{equation*} 
\xi_{A}^{*}((\cores_{G/G_1}^{G})_{*}(1))=\sum_{(\Gamma,G),\alpha}\#\Aut(\alpha)^{-1}\cdot\left(\prod_{(\ell,\ell')\in E'(A)}\frac{\ord_{G}(h_{\wt{\ell}})}{\ord_{G/G_1}(h_{\ell})}\cdot(\cores_{\alpha})_{*}(1)\right)
\end{equation*}
and
\begin{equation*} 
(\cores_{G/G_1}^{G})^{*}((\xi_{A})_{*}(1))=\sum_{(\Gamma,G),\alpha}\#\Aut(\alpha)^{-1}\cdot\left(\prod_{(\ell,\ell')\in E'(A)}\frac{\ord_{G}(h_{\wt{\ell}})}{\ord_{G/G_1}(h_{\ell})}.\cdot(\xi_{(\Gamma,G)})_{*}(1)\right)
\end{equation*}
\end{prop}

\subsection{Pullbacks of $\psi$ and $\kappa$ classes}\label{pullback_psi_kappa_section}

We consider the pullbacks of $\psi$ and $\kappa$ classes by the four types of tautological morphisms.

\subsubsection{Forgetful maps}

\begin{prop}
Let $\pi_G:\barH_{g,G,\xi+\{1\}}\to\barH_{g,G,\xi}$ be a forgetful map. We have:
\begin{enumerate}
\item[(a)]
\begin{equation*}
\pi_G^{*}(\psi_{p_{ia}})=\psi_{p_{ia}}-\sum_{g\in G/\langle h_{ia}\rangle}(r_g\circ\sigma_{p_{ia}})_{*}(1),
\end{equation*}
for any marked point $p_{ia}$ parametrized by $\barH_{g,G,\xi}$, where $\sigma_{p_{ia}}$ denotes the universal section described in \S\ref{universal_family_hurwitz}, the sum is over coset representatives $g\in G/\langle h_{ia}\rangle$, where $\langle h_{ia}\rangle\subset G$ is the stabilizer of $p_{ia}$, and $r_g:\barH_{g,G,\xi+\{1\}}\to \barH_{g,G,\xi+\{1\}}$ denotes the relabeling isomorphism sending the distinguished additional ramification point to its image under multiplication by $g$. 

\item[(b)]
\begin{equation*}
\pi^{*}(\kappa_i)=\kappa_i-\#G\cdot\psi_{b+1}^i,
\end{equation*}
where $\psi_{b+1}$ denotes the $\psi$ class on $\barH_{g,G,\xi+\{1\}}$ associated to any point in the additional unramified $G$-orbit.
\end{enumerate}

In particular, the pullbacks of $\psi$ and $\kappa$ classes by $\pi_G$ are $\cH$-tautological.
\end{prop}

\begin{proof}
We have a commutative diagram
\begin{equation*}
\xymatrix{
\barH_{g,G,\xi+\{1\}} \ar[r]^{\phi'} \ar[d]^{\pi_G} & \barM_{g,r+\#G} \ar[d]^{\pi}\\
\barH_{g,G,\xi} \ar[r]^{\phi} & \barM_{g,r}
}
\end{equation*}
where we recall that the source map $\phi'$ remembers \textit{all} of the marked points in the additional unramified orbit, and $\pi$ denotes the forgetful map forgetting all $|G|$ of these points.

By definition, the $\psi$ and $\kappa$ classes on $\barH_{g,G,\xi}$ are pulled back from $\barM_{g,r}$, so we may compute the desired pullbacks by pulling back by $\pi$, then $\phi'$. Let $B$ be the subset of marked points forgotten under the map $\pi$. Propositions \ref{psi_kappa_forgetful_pullback_classical} and \ref{boundary_forgetful} imply

\begin{equation*}
\pi^{*}(\psi_{p_{ia}})=\psi_{p_{ia}}-\sum_{B'\subset B}\Delta_{0,p_{ia}\cup B'},
\end{equation*}
where the sum is over non-empty subsets $B'\subset B$, and $\Delta_{0,p_{ia}\cup B'}$ denotes the boundary divisor of curves with a rational tail containing the marked points of $p_{ia}\cup B'$. 

By Lemma \ref{restriction_boundary_cartesian}, the pullback of $\Delta_{0,p_{ia}\cup B'}$ by $\phi'$ is equal to the union of boundary divisors corresponding to admissible $G$-graphs $(\Gamma,G)$ for which the legs associated to $p_{ia}\cup B'$ are attached to a common genus 0 vertex, each appearing with multiplicity 1. This is possible exactly when $B'$ is an orbit of $B$ under $\langle h_{ia}\rangle\subset G$, in which case the boundary divisor $\barH_{(\Gamma,G)}$ is of the form $(r_g\circ\sigma_{p_{ia}})_{*}(1)$. The union of such divisors is obtained by varying over a set of coset representatives $g\in G/\langle h_{ia}\rangle$, yielding (a).

For (b), Propositions \ref{psi_kappa_forgetful_pullback_classical} and \ref{boundary_forgetful} show that 
\begin{equation*}
\pi^{*}(\kappa_i)=\kappa_i-\left(\sum_{1}^{\#G}\psi_{r+j}^i\right)+\gamma,
\end{equation*}
where $\psi_{r+j}$ denote the $\psi$ classes at the marked points forgotten under $\pi$, and $\gamma\in A^i(\barM_{g,r+\#G})$ is supported on the union of boundary divisors where at least two of these marked points are equal to each other. Note that $\phi'^{*}(\gamma)=0$, because these points are constrained to be pairwise distinct on $\barH_{g,G,\xi+\{1\}}$. Thus, pulling back by $\phi'$ yields (b), as the $\psi_{r+j}^i$ all become equal to $\psi_{b+i}^i$.

\end{proof}

\subsubsection{Diagonal maps}

\begin{prop}\label{diagonal_pullback}
Let $\Delta:\barH_{g,G,\xi}\to\barH_{g,G,\xi}\times\barH_{g,G,\xi}$ be a diagonal map. Then, 
\begin{equation*}
\Delta^{*}(\theta,1)=\theta
\end{equation*} 
for any class $\theta$; in particular, $\theta$ can be taken to be a $\psi$ or $\kappa$ class.
\end{prop}

\begin{proof}
Trivial.
\end{proof}

\subsubsection{Boundary maps}

\begin{prop}
Let $\xi_{(\Gamma,G)}:\barH_{(\Gamma,G)}\to\barH_{g,G,\xi}$ be a boundary class. We have:
\begin{equation*}
\xi_{(\Gamma,G)}^{*}\psi_{p_{ia}}=\pr_{v}^{*}\psi_{p_{ia'}},
\end{equation*}
where $\pr_v$ is the projection to the component $\barH_{g_v,G_v,\xi_v}$ of $\barH_{(\Gamma,G)}$ corresponding to the component containing the $i$-th marked orbit and $p_{ia'}$ is any point belonging to this orbit. Moreover, 
\begin{equation*}
\xi_{(\Gamma,G)}^{*}\kappa_i=\sum_{v\in V'(\Gamma)}\frac{\#G}{\#G_v}\pr_{v}^{*}(\kappa_i)
\end{equation*}
\end{prop}

\begin{proof}
Recall from Remark \S\ref{boundary_relate_to_classical} that we have the commutative diagram
\begin{equation*}
\xymatrix{
\barH_{(\Gamma,G)} \ar[r]^{\xi_{(\Gamma,G)}} \ar[d]_{\phi_{(\Gamma,G)}} & \barH_{g,G,\xi} \ar[d]^{\phi} \\
\barM_{\Gamma} \ar[r]^{\xi_{\Gamma}} & \barM_{g,r}
}
\end{equation*}
where $\phi_{(\Gamma,G)}$ is a composition of the map $\phi$ on individual components followed by a product of diagonal maps. The claims now follow from Propositions \ref{psi_kappa_boundary_pullback_classical} and \ref{diagonal_pullback}.
\end{proof}

\subsubsection{Restriction-corestriction maps}

\begin{prop}
Let $\res_{G_1}^{G}:\barH_{g,G,\xi}\to\barH_{g,G_1,\xi'}$ be a restriction map. Then, 
\begin{equation*}
(\res_{G_1}^{G})^{*}(\psi_{p})=\psi_{p}
\end{equation*}
for any marked point $p$ on the source, and 
\begin{equation*}
(\res_{G_1}^{G})^{*}(\kappa_i)=\kappa_i
\end{equation*}
for all $i$.
\end{prop}

\begin{proof}
Immediate from the compatibility of restriction maps with the source maps $\phi$.
\end{proof}

\begin{prop}
Let $\cores_{G/G_1}^{G}:\barH_{g,G,\xi}\to\barH_{g_1,G/G_1,\xi'}$ be a corestriction map. Then,
\begin{equation*}
(\cores_{G/G_1}^{G})^{*}(\psi_{p})=\frac{\ord_{G/G_1}(h)}{\ord_{G}(h)}\cdot\psi_{\wt{p}},
\end{equation*} 
where $p$ is a marked point among the data of $\barH_{g_1,G/G_1,\xi'}$, $\wt{p}$ is any marked point among the data of $\barH_{g,G,\xi}$ living over $p$, and $h$ abusively denotes the corresponding monodromy element on both spaces. Moreover, 
\begin{equation*}
(\res_{G_1}^{G})^{*}(\kappa_i)=\frac{1}{\#G_1}\cdot \kappa_i
\end{equation*}
for all $i$.
\end{prop}

\begin{proof}
Immediate from the compatibility of corestriction maps with the target maps $\delta$, along with Proposition \ref{psi_kappa_comparison_forgetful}.
\end{proof}

\subsection{Interlude: Chern classes of the tangent bundle}\label{tangent_bundle_hurwitz}

In order to be able (in principle) to carry out further excess intersection calculations in the $\cH$-tautological ring, we will need to be able to compute the Chern classes of the tangent bundles of the spaces $\barH_{g,G,\xi}$. We do this by pullback from the proper and quasi-finite morphism $\delta:\barH_{g,G,\xi}\to\barM_{g',b}$, applying Bini's formula for the Chern classes of $\barM_{g',b}$.

Let $R$ be the ramification divisor of $\delta$. We have
\begin{equation*}
\cT_{\barH_{g,G,\xi}}\cong \delta^{*}(\cT_{\barM_{g',b}})\otimes\cO_{\barH_{g,G,\xi}}(-R).
\end{equation*}
Recall from \S\ref{galois_covers_def_section} that $\delta$ is ramified at the loci parametrizing admissible $G$-covers with ramification at nodes. More precisely, $R$ is a sum of divisors generically parametrizing admissible $G$-covers $f:X\to Y$ where $X$ has exactly one $G$-orbit of nodes with ramification index $e_i$ over the unique node of $Y$; such a divisor appears in $R$ with multiplicity $e_i-1$.

In the language of admissible $G$-graphs, we consider $(\Gamma,G)$ for which $G$ acts transitively on the non-empty set $E(\Gamma)$. Let $e_{\Gamma}=\ord_{G}(h_{\ell})$ for any half-edge $\ell\in H(\Gamma)$. Let $\barH^0_{(\Gamma,G)}$ be the image substack of $\barH_{(\Gamma,G)}$ in $\barH_{g,G,\xi}$, which has class
\begin{equation*}
[\barH^0_{(\Gamma,G)}]=\frac{1}{\#\Aut(\Gamma,G)}\cdot(\xi_{\Gamma})_{*}(1)\in A^{1}(\barH_{g,G,\xi}). 
\end{equation*}

Putting everything together, we conclude
\begin{prop}\label{chern_class_hurwitz}
We have
\begin{equation*}
\cT_{\barH_{g,G,\xi}}\cong \delta^{*}(\cT_{\barM_{g',b}})\left(-\sum_{(\Gamma,G)}(e_{\Gamma}-1)\cdot\barH^0_{(\Gamma,G)}\right),
\end{equation*}
where the sum is over admissible $G$-graphs $(\Gamma,G)$ with exactly one $G$-orbit of edges.

In particular, the Chern classes of $\cT_{\barH_{g,G,\xi}}$ may be computed explicitly by combining the formula of Bini \cite[Theorem 2]{bini} with the earlier results of this section.
\end{prop}

Bini's formula is expressed in terms of classical tautological classes, so it follows from \S\ref{boundary_strata_intersection_section}, \S\ref{corestriction_boundary_intersection_section}, and \S\ref{pullback_psi_kappa_section} that $\cT_{\barH_{g,G,\xi}}$ may be expressed in terms of decorated boundary strata, that is, pushforwards of polynomials in $\psi$ and $\kappa$ classes by boundary morphisms.

\subsection{Intersections of restriction-corestriction maps}\label{res-cores_int}

In this section, we show that the intersections of restriction-corestriction maps $\gamma_i:\cH_i\to\cH$ may be expressed in terms of $\cH$-tautological classes on a space mapping simultaneously to both $\cH_i$. These intersections may be computed in principle, in terms of a finite amount of data, but, as will be clear, the algorithmic complexity seems prohibitive in practice.

As usual, it suffices to consider restriction and corestriction maps separately; there are three cases.

\subsubsection{Intersection of two restriction maps}

Consider two inclusions of groups $G\subset H_1$ and $G\subset H_2$, with restriction maps $\res_{G}^{H_1}:\barH_{g,H_1,\xi_1}\to\barH_{g,G,\xi}$ and $\res_{G}^{H_2}:\barH_{g,H_2,\xi_2}\to\barH_{g,G,\xi}$. We first consider the fiber product of the $\res_{G}^{H_j}$.

Let us first explain the main idea of the construction. Suppose that $X$ is a stable curve admitting simultaneous actions by the two groups $H_1,H_2$, both of which are free away from the marked points of $X$, and which agree when restricted to the common subgroup $G$. Let $H$ be the subgroup of $\Aut(X)$ generated by $H_1$ and $H_2$. Then, $H$ acts on $X$, the actions of $H_1,H_2$ are both induced by that of $H$, and $G\subset H_1\cap H_2$.

However, the action of $H$ on $X$ may not be generically free; rather, for individual components $X_0\subset X$, the action of $H_0=\Stab_H(X_0)$ on $X_0$ may have a normal subgroup $N_0\subset H_0$ acting trivially.

\begin{ex}\cite[Proposition 5.15]{schmittvanzelm}
Consider the locus of stable genus 4 curves $X$ admitting both a hyperelliptic and bielliptic involution, that is, the intersection of the morphisms below:
\begin{equation*}
\xymatrix{
 & \barH_{4,\bZ/2\bZ,(1)^6} \ar[d]^{\phi_0}
 \\
\barH_{4,\bZ/2\bZ,(1)^2} \ar[r]^(0.59){\phi_1} & \barM_4
}
\end{equation*}
For simplicity, we have forgotten the marked points on $\barM_{4}$. We have also labelled the maps $\phi_i$ with a subscript keeping track of the target genus. Hyperelliptic and bielliptic involutions necessarily commute, so the two actions combine to give an action of $H=\bZ/2\bZ\times\bZ/2\bZ$ on $X$.

The intersection of these two morphisms consists of two strata:
\begin{itemize}
\item The locus $\cA$ of curves $X=X_1\cup X_2$, where $X_k$ is a curve of genus $k$ and the two components are attached at two nodes, such that $X_2$ admits a bielliptic involution interchanging the nodes, and
\item The locus $\cB$ of curves $X=X_1\cup X_2$, where $X_k$ is a curve of genus $k$ and the two components are attached at two nodes, such that $X_1$ admits a bielliptic involution interchanging the nodes.
\end{itemize}
In both cases, the hyperelliptic involution $(1,0)\in H$ acts by a hyperelliptic involution on both components, interchanging the nodes. The bielliptic involution $(0,1)\in H$ acts by the bielliptic involution on one of the two components, but the hyperelliptic one on the other. Thus, the diagonal element $(1,1)$ acts trivially on the non-bielliptic component.

In particular, a point in either locus cannot smooth to a non-singular genus 4 curve admitting both a hyperelliptic and bielliptic involution; in fact, no such non-singular curves exist for $g\ge4$.
\end{ex}

We thus make the following definition, generalizing the notion of an admissible $H$-graph to keep track of the possibility that the action of $H$ on a stable curve $X$ may have kernels when restricted to individual components, but that upon restriction to the subgroups $H_j\subset H$, gives an admissible $H_j$-action in the usual sense.

\begin{defn}\label{augmented_graph_defn}
Let $H$ be a finite group and let $H_1,H_2$ be subgroups with $\langle H_1,H_2\rangle=H$. An \textbf{augmented admissible $H$-graph relative to $H_1,H_2$}
\begin{equation*}
(\Gamma,H,\{N_v\}_{v\in V(\Gamma)})
\end{equation*}
consists of the following data:
\begin{itemize}
\item a stable graph $\Gamma$ with an $H$-action, 
\item an ordering on the set of $H$-orbits of $L(\Gamma)$, and a distinguished element $\ell_{i1}$ of each such orbit, so that all elements of $L(\Gamma)$ may be labelled in the usual way by left cosets of the stabilizer $H_{\ell_{i1}}$ of $\ell_{i1}$,
\item for each $v\in V(\Gamma)$, a normal subgroup $N_v$ of the stabilizer $H_v$ of $v$, that additionally acts trivially on the set $L_v\subset H(\Gamma)\cup L(\Gamma)$ of half-edges and legs incident at $v$, and
\item for all $v$ and $\ell\in L_v$, a monodromy datum $h_\ell\in H_v/N_v$.
\end{itemize}

We require the following conditions:
\begin{itemize}
\item if $(\ell,\ell')$ is a pair of half-edges comprising an edge, then $t\cdot\ell\neq\ell'$ for all $t\in H$,
\item $N_{t\cdot v}=tN_{v}t^{-1}$ for any $v\in V(\Gamma)$ and $t\in H$,
\item the stabilizer $H_{\ell}\subset H_v/N_v$ of $\ell$ is a cyclic group generated by $h_{\ell}$ for all $v\in V(\Gamma)$ and $\ell\in L_v$,
\item $h_{t\cdot\ell}=th_{\ell}t^{-1}$ in $H_{t\cdot v}/N_{t\cdot v}=tH_vt^{-1}/tN_{v}t^{-1}$ for all $\ell\in L_v$ and $t\in H$,
\item $\bigcap_{v\in V}N_v=\{1\}$,
\item $N_v\cap H_j=\{1\}$ for any $v\in V(\Gamma)$ and $j=1,2$, and
\item the $h_{\ell}$ satisfy the balancing condition when restricted, in the sense of Definition \ref{restriction_graphs}, to $H_1$ and $H_2$.
\end{itemize}
\end{defn}

Let us spell out the last condition in more detail. Suppose that $(\ell,\ell')$ is a pair of half-edges comprising an edge and let $v,v'$ be the vertices to which $\ell,\ell'$ are attached, respectively. Let $\Gamma_v$ be the (unpointed) admissible $H_v/N_v$-graph with the single vertex $v$ and leg set $L_v$ with monodromy data $\ell_v$.

Then, under the restriction of the $H_v/N_v$-action on $\Gamma_v$ to $H_j\cap H_v$ (which is a subgroup of $H_v/N_v$ because $N_v\cap H_j=\{1\}$), the monodromy datum $h_{\ell}\in H_{\ell}\subset H_v/N_v$ becomes $h_{\ell,j}=h_{\ell}^{r_{\ell,j}}$, where
\begin{equation*}
r_{\ell,j}=\frac{\#H_{\ell}}{\#H_{\ell}\cap H_j}.
\end{equation*}
Similarly, we define $h_{\ell',j}$. The balancing condition then requires that $h_{\ell,j}=h_{\ell',j}^{-1}\in H_j$ for $j=1,2$.

We may define a \textbf{morphism} (resp. automorphism) of augmented admissible $H$-graphs to be an $H$-equivariant morphism (resp. automorphism) of graphs compatible with the data of the leg orbit orderings, normal subgroups $N_v$, and monodromy data $h_{\ell}$.

The following explains how augmented admissible $H$-graphs arise when intersecting restriction morphisms.

\begin{lem}\label{get_aug_adm_graph}
Suppose that $X$ is a closed point of the functor $\barH_{g,H_1,\xi_1}\times_{\barH_{g,G,\xi}}\barH_{g,H_2,\xi_2}$. Let $H=\langle H_1,H_2\rangle\subset\Aut(X)$, where by $\Aut(X)$ we mean the finite group of automorphisms of $X$ that permute the marked points. Then, we have an augmented admissible $H$-graph
\begin{equation*}
(\Gamma_X,H,\{N_v\}),
\end{equation*}
where $\Gamma_X$ is the dual graph of $X$, $N_v$ is the kernel of the $H$-action on the component of $X$ corresponding to $v$, and we make arbitrary choices of distinguished legs $\ell_{i1}$ in the $H$-orbits of $L(\Gamma)$ and of an ordering of these $H$-orbits.
\end{lem}

\begin{proof}
Each component $X_v\in X$ has the structure of an $H_v/N_v$-curve, and we obtain monodromy data $h_{\ell}\in H_v/N_v$ satisfying the needed compatibilities. Because $H\subset\Aut(X)$, the action $H$ on $X$ has no global kernel, so the $N_v$ have no common intersection. Finally, restricting to the action on $X$ recovers the underlying admissible $H_j$-structures, so we get the last two conditions.
\end{proof}

Note that the construction of $(\Gamma_X,H,\{N_v\})$ also makes sense for a family of curves, as we may replace the nodes of a single curve with the nodal sections of a family. Moreover, the restriction maps are unramified, so nodes are smoothed to any order on $\barH_{g,H_j,\xi_j}$ if and only if they are on $\barH_{g,G,\xi}$.

We have the following space of curves with a given underlying augmented admissible $H$-graph:

\begin{defn}
Given any augmented admissible $H$-graph $(\Gamma,H,\{N_v\})$, we define
\begin{equation*}
\barH_{(\Gamma,H,\{N_v\})}=\prod_{v\in V(\Gamma)/H}\barH_{g_v,H_v/N_v,\xi_v}.
\end{equation*}
where in the factors on the right hand side, we make arbitrary choices of distinguished marked points, as in the definition of boundary strata, Definition \ref{boundary_stratum_hurwitz_defn}.
\end{defn}

Using an analogous gluing procedure to Definition \ref{boundary_morphism_hurwitz_defn}, a point of $\barH_{(\Gamma,H,\{N_v\})}$ may be interpreted as a stable curve on $X$ with $H$-action. The condition $N_v\cap H_j=\{1\}$ and the balancing condition upon restriction to $H_j$ ensure that this $H$-curve may be restricted, relative to a choice of relabeling data for the subgroup $H_j\subset H$, to an admissible $H_j$-curve. Then, the data of a point of $\barH_{(\Gamma,H,\{N_v\})}$ is equivalent to that of a stable curve on $X$ with $H$-action and a marking by $(\Gamma,H,\{N_v\})$, that is, a morphism of augmented admissible $H$-graphs from the dual graph of $X$ to $\barH_{(\Gamma,H,\{N_v\})}$.

In the augmented setting, the monodromy data are elements of $H_v/N_v$, rather than $H$, but the new distinguished marked points are still chosen among $H_j$-orbits of the marked points of the $H$-curve; the relabeling data, as before, amounts to an ordered choice of the corresponding distinguished points.

\begin{defn}
Following the constructions of Definitions \ref{boundary_morphism_hurwitz_defn} and \ref{restriction_graphs}, we define the proper, quasi-finite, and unramified restriction maps
\begin{equation*}
\res_{\Gamma,j}:\barH_{(\Gamma,H,\{N_v\})}\to\barH_{g,H_j,\xi_j}.
\end{equation*}
given by sending a point in $\barH_{(\Gamma,H,\{N_v\})}$ first to the corresponding stable curve equipped with a $\Gamma$-structure and $H$-action, then restricting to $H_j$.
\end{defn}

The details of the construction are left to the reader. Note that the maps $\res_{\Gamma,j}$ are compositions of (products of) diagonal, boundary, and restriction maps.

In order to compute the fiber product of two restriction morphisms, we need one more notion that will provide the correct genericity condition.

\begin{defn}
Let $(\Gamma,H,\{N_v\})$ be an augmented admissible $H$-graph, and let $(\ell,\ell')$ be a pair of half-edges comprising an edge of $\Gamma$, attached to the (possibly equal) vertices $v,v'$, respectively. Let $\wt{H}_{\ell}=\wt{H}_{\ell'}\subset H$ denote the stabilizer of the edge $(\ell,\ell')$, or equivalently the stabilizer of either half-edge individually, so that $\wt{H}_{\ell}/N_v=H_\ell$ and $\wt{H}_{\ell'}/N_{v'}=H_{\ell'}$.

We say that the edge $(\ell,\ell')$ is \textbf{smoothable} if $N_v=N_{v'}$ and we have the balancing condition $h_{\ell}=h_{\ell'}^{-1}$ in $\wt{H}_{\ell}/N_v=\wt{H}_{\ell'}/N_{v'}$.

We call the augmented admissible $H$-graph \textbf{non-smoothable} if none of its edges are smoothable.
\end{defn}

The smoothability condition on the edges of $\Gamma$ corresponds to the smoothability of the nodes as a point of $\barH_{(\Gamma,H,\{N_v\})}$. Namely, $X$ is a specialization of another closed point $X_\eta$ in the fiber product, where a given node of $X$ is smoothed into $X_\eta$, if and only if the corresponding edge of $\Gamma$ is smoothable. Correspondingly, if $\Gamma$ fails to be smoothable, we obtain a \textit{non-smoothable} $(\Gamma',H,\{N_v\})$ by applying the $H$-equivariant operation of contracting all smoothable edges and modifying the genus function accordingly.

\begin{lem}\label{res_res_diagram_lemma}
We have a commutative diagram
\begin{equation*}
\xymatrix{
\displaystyle\coprod\barH_{(\Gamma,H,\{N_v\})} \ar[r]^(0.55){\res_{\Gamma,1}} \ar[d]^{\res_{\Gamma,2}} & \barH_{g,H_1,\xi_1} \ar[d]^{\res^{H_1}_{G}}\\
\barH_{g,H_2,\xi_2} \ar[r]^(0.55){\res^{H_2}_{G}} & \barH_{g,G,\xi}
}
\end{equation*}
where the coproduct is over all \textbf{non-smoothable} augmented admissible $H$-graphs 
\begin{equation*}
(\Gamma,H,\{N_v\})
\end{equation*}
relative to $H_1,H_2$, along with all possible relabeling data defining the maps $\res_{\Gamma_j}$ that are compatible with the monodromy data $\xi_j$ and the maps $\res^{H_j}_{G}$.

Moreover, the induced map
\begin{equation*}
\coprod(\res_{\Gamma,1},\res_{\Gamma,2}):\displaystyle\coprod\barH_{(\Gamma,H,\{N_v\})} \to \barH_{g,H_1,\xi_1}\times_{\barH_{g,G,\xi}}\barH_{g,H_2,\xi_2}
\end{equation*}
is proper, quasi-finite, \'{e}tale, and surjective. In particular, the fiber product is smooth. 
\end{lem}

More precisely, the compatibility of the relabeling data entails the following two conditions:
\begin{itemize}
\item The resulting restriction maps $\res_{\Gamma,j}$ produce the correct monodromy data $\xi_j$ from that of $(\Gamma,H,\{N_v\})$, and
\item The relabeling data defining the composite morphisms $\res^{H_j}_{G}\circ\res^{\Gamma,j}$ agree, that is, produce the same ordered collection of distinguished marked points.
\end{itemize}

\begin{proof}
The commutativity is immediate from the requirements on compatibility; the common composition is simply the restriction to the subgroup $G\subset H$, with the induced relabeling data. All of the restriction maps are proper and quasi-finite so $(\res_{\Gamma,1},\res_{\Gamma,2})$ is as well.

We now show that $(\res_{\Gamma,1},\res_{\Gamma,2})$ is an analytic local isomorphism. Consider a point of $\barH_{g,H_1,\xi_1}\times_{\barH_{g,G,\xi}}\barH_{g,H_2,\xi_2}$ over a local base scheme $S$. This consists of the data of a family $\mathcal{C}\to S$ of $G$-curves equipped with fiberwise $H_j$-actions, along with discrete relabeling data. The data of the group actions are equivalent to embeddings of \textit{constant} group schemes $H_j\times S\hookrightarrow\Aut(\mathcal{X})$ over $S$. We then define $H$ to be the group for which the $H_j$ generate the constant group scheme $H\times S$ inside $\Aut(\mathcal{X})$.

Then, we obtain a non-smoothable augmented admissible $H$-graph $(\Gamma,H,\{N_v\})$ as in Lemma \ref{get_aug_adm_graph}. Any ordered choice of distinguished marked points yields an $S$-point of the smooth stack $\barH_{(\Gamma,H,\{N_v\})}$, which maps isomorphically to the original point of $\barH_{g,H_1,\xi_1}\times_{\barH_{g,G,\xi}}\barH_{g,H_2,\xi_2}$ relative to the unique compatible choice of relabeling data. Moreover, the genericity conditions of non-smoothability and $H=\langle H_1,H_2\rangle$ guarantee that these are the only points of $\displaystyle\coprod\barH_{(\Gamma,H,\{N_v\})}$ living over $\mathcal{C}$. 
\end{proof}

We now compute the degree of $\displaystyle\coprod\barH_{(\Gamma,H,\{N_v\})}$ over the fiber product in Lemma \ref{res_res_diagram_lemma}.

\begin{lem}
Given a point $[\cC]\in \barH_{g,H_1,\xi_1}\times_{\barH_{g,G,\xi}}\barH_{g,H_2,\xi_2}(\Spec(\bC))$, the number of isomorphism classes of augmented admissible $H$-graphs $(\Gamma,H,\{N_v\})$ for which $\barH_{(\Gamma,H,\{N_v\})}$ contains points mapping to of $[\cC]$ under $\coprod(\res_{\Gamma,1},\res_{\Gamma,2})$ is
\begin{equation*}
(\#L'(\Gamma))!\cdot\prod_{\ell\in\#L'(\Gamma)}\frac{\# H}{\#\wt{H}_{\ell}}=(\#L'(\Gamma))!\cdot\prod_{\ell\in\#L'(\Gamma)}\frac{\# H}{\ord_{H_\ell}(h_\ell)\cdot\# N_{v_{\ell}}}.
\end{equation*}
where $L'(\Gamma)\subset L(\Gamma)$ is any choice of $H$-orbit representatives. Furthermore, the degree of such a $\barH_{(\Gamma,H,\{N_v\})}$ over $\barH_{g,H_1,\xi_1}\times_{\barH_{g,G,\xi}}\barH_{g,H_2,\xi_2}$ is 
\begin{equation*}
\#\Aut((\Gamma,H,\{N_v\})).
\end{equation*}
\end{lem}

\begin{proof}
As we saw in the proof of Lemma \ref{res_res_diagram_lemma}, the genericity conditions require all pre-images of $[\cC]$ to come from the same $H$-graph with the same subgroups $N_v$, while the only choice available is for the ordered set of distinguished marked points on $(\Gamma,H,\{N_v\})$; the relabeling data for $\res_{\Gamma,j}$ are then determined. In the $H$-orbit of $\ell\in L(\Gamma)$, there are 
\begin{equation*}
\frac{\#H}{\#\wt{H}_{\ell}}=\frac{\# H}{\ord_{H_\ell}(h_\ell)\cdot\# N_{v_{\ell}}}
\end{equation*}
possible choices of a distinguished marked point, where $v_{\ell}$ is the vertex to which $\ell$ is attached, and $(\#L'(\Gamma))!$ ways to order these choices. 

Finally, similarly as in the proof of Lemma \ref{cores_boundary_diagram}, each $\barH_{(\Gamma,H,\{N_v\})}$ forms a $\Aut((\Gamma,H,\{N_v\}))$-torsor over $\barH_{g,H_1,\xi_1}\times_{\barH_{g,G,\xi}}\barH_{g,H_2,\xi_2}$, corresponding to possible choices of distinguished half-edge orbit representatives. This gives the last claim.
\end{proof}

\begin{rem}
The index set for the coproduct can in principle be enumerated by a finite computation. Indeed, $\#\Aut(X)$ is uniformly bounded in terms of $g$ and the number of marked points, so there are only finitely many possible $H\subset\Aut(X)$, and finitely many choices of the subsequent data. 
\end{rem}

It remains to compute the top Chern class of the excess bundle on the fiber product. We may do so on each component of the finite unramified cover
\begin{equation*}
\displaystyle\coprod\barH_{(\Gamma,H,\{N_v\})}
\end{equation*}
and then divide by the degree over the corresponding component of the fiber product.

We have no combinatorial description of the normal bundles of restriction maps as in the case of boundary maps, but we may compute explicit formulas for the Chern classes of the spaces involved, see \S\ref{tangent_bundle_hurwitz}.

In the end, we conclude the following.

\begin{prop}
We have
\begin{align*}
(\res_{G}^{H_2})^{*}((\res_{G}^{H_1})_{*}(1))=\sum_{(\Gamma,H,\{N_v\})}\left(\#\Aut((\Gamma,H,\{N_v\}))\cdot(\#L'(\Gamma))!\cdot\prod_{\ell\in\#L'(\Gamma)}\frac{\# H}{\ord_{H_\ell}(h_\ell)\cdot\# N_{v_{\ell}}}\right)^{-1}\\
\cdot \left(\res_{\Gamma,2}\right)_{*}\left(c_{\topchern}\left(\frac{\cT_{\barH_{(\Gamma,H,\{N_v\})}}\cdot\cT_{\barH_{g,G,\xi}}}{\cT_{\barH_{g,H_1,\xi_1}}\cdot\cT_{\barH_{g,H_2,\xi_2}}}\right)\right)
\end{align*}
where the sum is over augmented admissible $H$ graphs along with relabeling data as above, and all tangent bundles are pulled back to $\barH_{(\Gamma,H,\{N_v\})}$. 
\end{prop}

Recall from \S\ref{tangent_bundle_hurwitz} that the Chern classes of the tangent bundles in question may be expressed in terms of decorated boundary classes, which, owing to \S\ref{restriction_boundary_intersection_section}, themselves pull back to decorated boundary classes by tautological restriction and boundary maps. However, in pulling back by the maps $\res_{\Gamma,i}$, we also have to pull such classes back by diagonals. To do this, we may apply Proposition \ref{diagonal_with_product} with $\gamma_i$ taken to be boundary maps, which is proven later but independently of this section. We find, by Proposition \ref{boundary_intersection_formula_hurwitz}, that we again get decorated boundary classes in this pullback. In particular, the bivariant pullback of one restriction morphism by another is $\cH$-tautological.

\subsubsection{Intersection of a restriction and corestriction map}

Let $p:H_1\to G$ be a surjection and let $i:G\to H_2$ be an inclusion of groups. We consider the intersection of the corestriction $\cores^{H_1}_{G}:\barH_{g_1,H_1,\xi_1}\to\barH_{g,G,\xi}$ and the restriction $\res^{H_2}_{G}:\barH_{g,H_2,\xi_2}\to\barH_{g,G,\xi}$. 

Consider a point in the fiber product, given by the $H_1$-cover $W\to Y$ corestricting to $X\to Y$ and the $H_2$-cover $X\to Z$ restricting also to $X\to Y$. Then, the data of all three of these maps is contained in the composition
\begin{equation*}
\xymatrix{
W \ar[r] \ar@/^2.0pc/[rr]^{H_1} & X \ar[r]^{G} \ar@/_2.0pc/[rr]_{H_2} & Y \ar[r] & Z
}
\end{equation*}
However, this composition may not be Galois. If we assume that the curves in question are smooth, we may pass to the Galois closure $\wt{W}\to Z$ over $W\to Z$, which is a $\wt{H}$-cover for some finite group $\wt{H}$. Let $\wt{p}:\wt{H}\to H_2$ be the natural surjection. Then, the Galois group of $\wt{W}$ over $Y$ is 
\begin{equation*}
\wt{H}_1:=\wt{p}^{-1}(G)\subset\wt{H};
\end{equation*}
let $\wt{i}:\wt{H}_1\to\wt{H}$ be the inclusion and $\wt{q}:\wt{H}\to H_1$ be the natural surjection. We then have a commutative diagram
\begin{equation*}
\xymatrix{
\wt{H}_1 \ar@{^{(}->}[r]^{\wt{i}} \ar@{->>}[d]^{\wt{q}} & \wt{H} \ar@{->>}[dd]^{\wt{p}} \\
H_1 \ar@{->>}[d]^p & \\
G \ar@{^{(}->}[r]^i & H_2
}
\end{equation*}
\begin{defn}
Given a surjection $p_1:H_1\to G$ and injection $i:G\to H_2$ of finite groups, we define a \textbf{GC-diagram} (\textit{Galois closure} diagram) to consist of the data of a surjection $\wt{p}:\wt{H}\to H_2$ of finite groups along with a surjection $\wt{q}$ from $\wt{H}_1:=\wt{p}^{-1}(G)\subset\wt{H}$ to $H_1$ making the diagram as above commute.

A GC-diagram is \textbf{minimal} if it does not factor through another GC-diagram via a non-trivial surjection $\wt{H}\to\wt{H'}$, as below:
\begin{equation*}
\xymatrix{
\wt{H}_1 \ar@{^{(}->}[r] \ar@{->>}[d] & \wt{H} \ar@{->>}[d] \\
\wt{H}'_1 \ar@{^{(}->}[r]^{\wt{i}} \ar@{->>}[d]^{\wt{q}} & \wt{H}' \ar@{->>}[dd]^{\wt{p}} \\
H_1 \ar@{->>}[d]^p & \\
G \ar@{^{(}->}[r]^i & H_2
}
\end{equation*}
\end{defn}

Fix a minimal GC diagram, and monodromy data $\wt{\xi}=(\wt{h}_1,\ldots,\wt{h}_b)\in\wt{H}^b$ which map to $\xi_2=(h^2_1,\ldots,h^2_b)\in H_2^b$ under the surjection $\wt{p}$. Thus, we have a corestriction map
\begin{equation*}
\cores^{\wt{H}}_{\wt{H_2}}:\barH_{\wt{g},\wt{H},\wt{\xi}}\to\barH_{g,H_2,\xi_2},
\end{equation*}
where the genus $\wt{g}$ is determined by the Riemann-Hurwitz formula.

On the other hand, suppose that we have a choice of relabeling data for the subgroup $\wt{H}_1\subset\wt{H}$ that the resulting restriction map
\begin{equation*}
\res^{\wt{H}}_{\wt{H}_1}:\barH_{\wt{g},\wt{H},\wt{\xi}}\to \barH_{g_1,\wt{H}_1,\wt{\xi}_1}
\end{equation*}
yields monodromy data $\wt{\xi}_1$ mapping to $\xi_1$ under the quotient map $\wt{q}:\wt{H}_1\to H_1$. Then, we have a restriction-corestriction map
\begin{equation*}
q_{\wt{K}_1,\wt{H}_1,\wt{H}}=\cores^{\wt{H}_1}_{H_1}\circ\res^{\wt{H}}_{\wt{H}_1}:\barH_{\wt{g},\wt{H},\wt{\xi}}\to \barH_{g_1,H_1,\xi'_1}
\end{equation*}
where $\wt{K}_1=\ker(\wt{q})$

Suppose, in addition, that the relabeling data defining $\res^{\wt{H}}_{\wt{H}_1}$ is compatible with that defining $\res^{H_2}_{G}$ in the sense that the ordered set of orbit representatives of $\wt{H}_1$ acting on the $\wt{H}/\langle \wt{h}_i\rangle$ maps to that of the orbit representatives of $G$ acting on $\wt{H}_2/\langle h^2_i\rangle$ under $\wt{p}$ (with the same indices $i$ everywhere). Then, by construction, we have a commutative diagram

\begin{equation*}
\xymatrix@C+2pc{
\barH_{\wt{g},\wt{H},\wt{\xi}} \ar[r]^{q_{\wt{K_1},\wt{H}_1,\wt{H}}} \ar[d]_{\cores^{\wt{H}}_{H_2}} & \barH_{g_1,H_1,\xi_1} \ar[d]^{\cores^{H_1}_{G}} \\
\barH_{g_2,H_2,\xi_2} \ar[r]^{\res^{H_2}_{G}} & \barH_{g,G,\xi}
}
\end{equation*}

The following is thus immediate:

\begin{lem}
Consider the commutative diagram
\begin{equation*}
\xymatrix@C+2pc{
\coprod \barH_{\wt{g},\wt{H},\wt{\xi}} \ar[r]^{q_{\wt{K_1},\wt{H}_1,\wt{H}}} \ar[d]_{\cores^{\wt{H}}_{H_2}} & \barH_{g_1,H_1,\xi_1} \ar[d]^{\cores^{H_1}_{G}} \\
\barH_{g_2,H_2,\xi_2} \ar[r]^{\res^{H_2}_{G}} & \barH_{g,G,\xi}
}
\end{equation*}
where the coproduct is taken over all possible minimal $GC$ diagrams with compatible choices of monodromy data $\wt{\xi}$ and relabeling data defining $q_{\wt{K_1},\wt{H}_1,\wt{H}}$ as detailed above. 

Then, the induced map
\begin{equation*}
\coprod(q_{\wt{K_1},\wt{H}_1,\wt{H}},\cores^{\wt{H}}_{H_2}):\coprod \barH_{\wt{g},\wt{H},\wt{\xi}} \to \barH_{g_1,H_1,\xi_1}\times_{\barH_{g,G,\xi}} \barH_{g_2,H_2,\xi_2}
\end{equation*}
is proper and quasi-finite.
\end{lem}

Indeed, the properness and quasi-finiteness follows from that of all of the morphisms involved. We also remark that the index set of the coproduct can in principle be computed, as the order of $\wt{H}$ has is bounded below by $((\# H_1\cdot \# H_2)/\#G)!$.

It remains to compute the degree of $\coprod(q_{\wt{K_1},\wt{H}_1,\wt{H}},\cores^{\wt{H}}_{H_2})$.

\begin{lem}
For any point 
\begin{equation*}
([W\to Y],[X\to Z])\in \barH_{g_1,H_1,\xi_1}\times_{\barH_{g,G,\xi}} \barH_{g_2,H_2,\xi_2}
\end{equation*}
where $W,X,Y,Z$ are smooth curves, let $\wt{W}\to Z$ be the Galois closure of the composition $W\to Z$, and let 
\begin{equation*}
\wt{K}_1\subset\wt{K}_2\subset\wt{H}_1\subset\wt{H}
\end{equation*}
be the Galois groups of $\wt{W}$ over $W,X,Y,Z$, respectively. Let $\wt{h}_i$, for $i=1,\ldots,b$, be any monodromy elements in the $\wt{H}$-orbits of $\wt{W}\to Z$. Then, the degree of $\coprod(q_{\wt{K_1},\wt{H}_1,\wt{H}},\cores^{\wt{H}}_{H_2})$ over $([W\to Y],[X\to Z])$ is
\begin{equation*}
\prod_{i=1}^{b}\#(\wt{K}_2/(\langle\wt{h}_i\rangle\cap\wt{K}_2))\cdot(\#\wt{K}_1/(\langle \wt{h}_i\rangle\cap \wt{K}_1))^{\# H_2/\# G}
\end{equation*}
\end{lem}

\begin{proof}
By the minimality requirement, the GC diagram giving a point in the pre-image of $([W\to Y],[X\to Z])$ must be that arising from $\wt{W}\to Z$. The points in the fiber above $([W\to Y],[X\to Z])$ then correspond to possible choices of distinguished marked points of $\wt{W}\to Z$ (which determine $\wt{\xi}$), along with possible choices of distinguished marked points of the restriction $\wt{W}\to Y$ (which determine the relabeling data).

For the cover $\wt{W}\to Z$, we must choose distinguished marked points lying over those of $X\to Z$, and their order is also dictated by that on $X\to Z$. The number of possible distinguished marked points in the $i$-th $\wt{H}$-orbit of $\wt{W}$ is then $\#(\wt{K}_2/(\langle\wt{h}_i\rangle\cap\wt{K}_2))$.

Similarly, for the cover $\wt{W}\to Y$, we must choose distinguished marked points lying the given ones of $W\to Y$, in the same order. The $i$-th $\wt{H}$-orbit of $\wt{W}$ breaks into $\#H/\#H_1$ $\wt{H}_1$-orbits, each of which has $\#\wt{K}_1/(\langle \wt{h}_i\rangle\cap \wt{K}_1)$ choices for a distinguished marked point over $Y$.

Combining yields the conclusion.
\end{proof}

Because a general point of the fiber product parametrizes a pair of covers of smooth curves, we have proven the following.
\begin{prop}
With notation as above, we have:
\begin{align*}
(\res^{H_2}_{G})^{*}&((\cores^{H_1}_{G})_{*}(1))\\
=&\sum\left(\prod_{i=1}^{b}\#(\wt{K}_2/(\langle\wt{h}_i\rangle\cap\wt{K}_2))\cdot(\#\wt{K}_1/(\langle \wt{h}_i\rangle\cap \wt{K}_1))^{\# H_2/\# G}\right)^{-1}(\cores^{\wt{H}}_{H_2})_{*}(1)
\end{align*}
and
\begin{align*}
(\cores^{H_1}_{G})^{*}&((\res^{H_2}_{G})_{*}(1))\\
=&\sum\left(\prod_{i=1}^{b}\#(\wt{K}_2/(\langle\wt{h}_i\rangle\cap\wt{K}_2))\cdot(\#\wt{K}_1/(\langle \wt{h}_i\rangle\cap \wt{K}_1))^{\# H_2/\# G}\right)^{-1}(q_{\wt{K_1},\wt{H}_1,\wt{H}})_{*}(1),
\end{align*}
where the sums are over choices of minimal GC-diagrams, monodromy data $\wt{\xi}$, and relabeling data as above. In particular, the intersections in question are $\cH$-tautological.
\end{prop}

\subsubsection{Intersection of two corestriction maps}

Let $H_1\to G,H_2\to G$ be surjections of finite groups with kernels $K_1\subset H_1,K_2\subset H_2$. We consider the intersection of the corestriction maps $\cores^{H_1}_{G}:\barH_{g_1,H_1,\xi_1}\to\barH_{g,G,\xi}$ and $\cores^{H_2}_{G}:\barH_{g_2,H_2,\xi_2}\to\barH_{g,G,\xi}$.

We first explain the idea of the construction. Consider admissible $H_j$-covers $f_j:X_j\to Y$ equipped with an isomorphism $X_1/K_1\cong X_2/K_2$ over $Y$; we identify both curves with the $G$-curve $X$. Suppose further that the $X_j$ and $Y$ are smooth. Then, away from the marked points of $Y$, the fiber product $f'_\bullet:X'_\bullet=X_1\times_X X_2\to Y$ is a principal $H_\bullet=H_1\times_{G}H_2$-bundle over $Y$. 

After normalizing $X'_\bullet$, we then get a $H_\bullet$-cover of smooth curves $f_\bullet:X_\bullet\to Y$. However, $X_\bullet$ may be disconnected. On the other hand, let $H_0$ be the stabilizer of a connected component $X_0\subset X_\bullet$; alternatively, $H_0$ is the image of the monodromy representation (obtained after choosing base-points appropriately) $\pi_1(Y,\{q_k\})\to H_1\times_G H_2$ induced by $f_1$ and $f_2$. Then, we can essentially recover the data of the original $H_j$-covers by corestricting from $H_0$ to $H_j$.

The issue remains that distinguished marked points must be chosen on $X_0$, and that these may not map to those on the $X_j$, so we will need to post-compose with relabeling isomorphisms on $\barH_{g_j,H_j,\xi_j}$.

%
%

\begin{defn}\label{cores_prime_defn}
As above, define $H_\bullet=H_1\times_{G}H_2$, and let $H_0\subset H_\bullet$ be a subgroup such that the compositions
\begin{equation*}
\xymatrix{
\pi_j:H_0 \ar[r] & H_\bullet=H_1\times_{G}H_2 \ar[r]^(0.71){\pr_j} & H_j
}
\end{equation*}
are both surjective. Moreover, fix $t_1,\ldots,t_b\in H_\bullet$. 

We define, for $j=1,2$, the map
\begin{equation*}
\cores'^{H_0}_{H_j}:\barH_{h_0,H_0,\xi_0}\to\barH_{h_j,H_j,\xi_j}
\end{equation*}
by the composition of the usual corestriction map
\begin{equation*}
\cores^{H_0}_{H_j}:\barH_{h_0,H_0,\xi_0}\to\barH_{h_j,H_j,\xi'_j}
\end{equation*}
and the relabeling isomorphism
\begin{equation*}
\res^{H_j}_{H_j}:\barH_{h_0,H_j,\xi'_j}\to\barH_{h_j,H_j,\xi_j}
\end{equation*}
replacing the distinguished marked point $p_{i1}$ with $\pr_j(t_i)\cdot p_{i1}$.
\end{defn}

Write $\xi_j=(h^j_1,\ldots,h^j_b)$ for $j=0,1,2$. Then, for the maps $\cores'^{H_0}_{H_j}$ to make sense, we need
\begin{equation*}
h^j_i=\pr_j(t_i)\cdot\pi_j(h^0_i)\cdot\pr(t_i^{-1})
\end{equation*}
for $j=1,2$.

\begin{lem}
For any choice of $H_0\subset H_\bullet$ and $t_1,\ldots,t_b\in H_\bullet$ as in Definition \ref{cores_prime_defn}, set, for $i=1,2,\ldots,b$,
\begin{equation*}
h^0_i=t_i^{-1}h^\bullet_i t_i\in H_\bullet,
\end{equation*}
where $h^\bullet_i=(h^1_i,h^2_i)$.
Then, we have a commutative diagram
\begin{equation*}
\xymatrix{
\coprod \barH_{h_0,H_0,\xi_0} \ar[r]^{\cores'^{H_0}_{H_1}} \ar[d]_{\cores'^{H_0}_{H_2}} & \barH_{g_1,H_1,\xi_1} \ar[d]^{\cores^{H_1}_{G}}\\
\barH_{g_2,H_2,\xi_2} \ar[r]^{\cores^{H_2}_{G}} & \barH_{g,G,\xi}
}
\end{equation*}
where the coproduct is taken over all possible choices above.

Moreover, the induced map
\begin{equation*}
\coprod(\cores'^{H_0}_{H_1},\cores'^{H_0}_{H_2}): \coprod\barH_{h_0,H_0,\xi_0} \to \barH_{g_1,H_1,\xi_1}\times_{\barH_{g,G,\xi}} \barH_{g_2,H_2,\xi_2}
\end{equation*}
is proper and quasi-finite.
\end{lem}

\begin{proof}
The choice of $h^0_i$ ensures that the maps $\cores'^{H_0}_{H_j}$ are well-defined. Then, the composition in both directions is the corestriction from $H_0$ to $G$, but where the distinguished marked points $p'_{i1}$ on the resulting $G$-cover are replaced by $t'_{i}\cdot p'_{i1}$, where $t'_{i}$ is the image of $t_i$ in $G$.

The properness and quasi-finiteness are then immediate from the same properties for the corestriction maps.
\end{proof}

%
%

\begin{lem}
For $j=1,2$, let $f_j:X_j\to Y$ be an $H_j$-cover, agreeing upon corestriction on the $G$-cover $f:X\to Y$. Assume further that the $X_j$ and $Y$ are smooth. Let $f_{\bullet}:X_\bullet\to Y$ be the normalized fiber product of the $f_j$ over $f$, and let $H'_0\subset H_\bullet$ be the stabilizer of any component of $X_\bullet$. Then, over $(f_1,f_2)$, the map
\begin{equation*}
\coprod(\cores'^{H_0}_{H_1},\cores'^{H_0}_{H_2}): \coprod\barH_{h_0,H_0,\xi_0} \to \barH_{g_1,H_1,\xi_1}\times_{\barH_{g,G,\xi}} \barH_{g_2,H_2,\xi_2}
\end{equation*}
has degree
\begin{equation*}
\#H_\bullet\cdot (\#H'_0)^{b-1}
\end{equation*}
\end{lem}

\begin{proof}
Suppose that we have a point of $\coprod\barH_{h_0,H_0,\xi_0}$ lying over $(f_1,f_2)$, given by a pointed $H_0$-cover $f_0:X_0\to Y$ along with $t_1,\ldots,t_b\in H_\bullet$ defining relabeling isomorphisms, such that the corestrictions of $f_0$ to $H_j$ yield the covers $f_j$. Then, by the universal property of fiber product, $f_0$ factors through $X_\bullet$; more specifically, it surjects onto a unique component of $X'_0\subset X_\bullet$. 

We claim, however, that the induced map $X_0\to X'_0$ is an isomorphism. If not, then we have a non-trivial automorphism of $X_0$ over both $X_1$ and $X_2$, but $H_0$ is a subgroup $H_1\times_G H_2$, so we have a contradiction. Therefore, we may canonically identify $H_0,X_0$ with $H'_0,X'_0$, respectively.

Now, in identifying points in the fiber over $(f_1,f_2)$, we have the following choices: a component $X'_0\subset X_\bullet$, which determines $H_0$, a collection of distinguished marked points of $X'_0$, which determines $\xi_0$, and the $t_i\in H_\bullet$. The number of choices of a component of $X_\bullet$ is $\#H/\#H'_0$, the number of choices of distinguished marked points of $X'_0$ in the $i$-th orbit is $\#H'_0/\ord_{H_\bullet}(h^\bullet_i)$, and the number of possible choices of $t_i$ is $\ord_{H_\bullet}(h^\bullet_i)$ (as the set of possible choices is a left coset of the stabilizer of the distinguished marked point in $H_\bullet$). Combining gives the desired degree.
\end{proof}

Because a general point on any component of $\barH_{g_1,H_1,\xi_1}\times_{\barH_{g,G,\xi}} \barH_{g_2,H_2,\xi_2}$ parametrizes a pair of covers of smooth curves, we have proven the following.

\begin{prop}
We have
\begin{equation*}
(\cores^{H_2}_G)^{*}((\cores^{H_1}_G)_{*}(1))=\sum_{H_0,\{t_i\}} \frac{1}{\#H_\bullet\cdot (\#H_0)^{b-1}}\cdot(\cores'^{H_0}_{H_2})_{*}(1),
\end{equation*}
where the sum, as before, is over subgroups $H_0\subset H_\bullet=H_1\times_{G}H_2$ surjecting onto the $H_i$ and choices of $t_i\in H_\bullet$.
\end{prop}

\subsection{Diagonals}\label{diagonal_intersection_section}

\begin{prop}\label{diagonal_with_product}
Let $\barH_{g,G,\xi}$ be any stack of $G$-covers, and let $\gamma_i:\cH_{i}\to\barH_{g,G,\xi}$ for $i=1,2$ be forgetful, boundary, corestriction, or restriction maps. Then, the intersections of $\gamma_1\times\gamma_2$ and the diagonal class $\Delta$ (in either order) are given by $\cH$-tautological classes computed earlier in this section.
\end{prop}

\begin{proof}
The intersection of $\gamma_1\times\gamma_2$ and $\Delta$ is given by that of $\gamma_1$ and $\gamma_2$, essentially by definition (see \cite[Chapter 8]{fulton}). We have already computed all of these intersections, in terms of an $\cH$-tautological class on a space $\cH_0$ mapping (via a composition of tautological morphisms) to both $\cH_j$, so in particular to $\barH_{g,G,\xi}$ and to $\cH_{1}\times\cH_2$.
\end{proof}

The only remaining case to consider is that of the self-intersection of the diagonal of $\barH_{g,G,\xi}$; this is given simply by the top Chern class of the tangent bundle, which we have already explained how to compute.

\begin{prop}\label{diagonal_squared}
We have $\Delta^{*}(\Delta_{*}(1))=c_{\topchern}(\cT_{\barH_{g,G,\xi}})$, which may be expressed in terms of decorated boundary classes, and is in particular $\cH$-tautological.
\end{prop}

\begin{proof}
See Proposition \ref{chern_class_hurwitz}.
\end{proof}

\section{Additive structure}

\subsection{Generators for the $\cH$-tautological ring}

Our main theorem is now immediate from the results of the previous section:

\begin{thm}\label{main_thm}
Let
\begin{equation*}
\cH=\prod_{i=1}^{m}\barH_{g_i,G_i,\xi_i}
\end{equation*}
be an arbitrary product of spaces of admissible $G$-covers. Consider all classes of the form $\gamma_{*}(\theta)\in A^{*}(\cH)$, where $\gamma:\cH'\to\cH$ is a composition of forgetful, boundary, diagonal, and restriction-corestriction morphisms from another product of spaces of admissible $G$ covers, and $\theta$ is a monomial in the $\psi$ and $\kappa$ classes on each of the components of $\cH'$. Then, the classes $\gamma_{*}(\theta)$ form an additive set of generators for $R^{*}_{\cH}(\cH)$. Moreover, the $\cH$-tautological ring is closed under bivariant pullback by all of the aforementionoed tautological morphisms.
\end{thm}

\begin{proof}
All such classes are $\cH$-tautological by definition. The results of the previous section allow one to compute the bivariant pullback of such a composition by another one, one tautological morphism at a time. At each step, we intersect basic $\cH$-tautological morphisms $\gamma_i:\cH_i\to\cH$, giving the answer in terms of an $\cH$-tautological class in terms of a space $\cH_0$ mapping simultaneously to the $\cH_j$, so that we can proceed to the next step. We may also pull $\psi$ and $\kappa$ classes back by any of these morphisms.

In particular, the intersection of any two classes of the form $\gamma_{*}(\theta)$ is a linear combination of such, so the theorem is proven.
\end{proof}

In fact, one can give a smaller set of additive generators for the $\cH$-tautological ring:

\begin{prop}\label{additive_gens_efficient}
The $\cH$-tautological ring is generated additively by classes $\gamma_{*}(\theta)$ as above, where
\begin{equation*}
\gamma=\xi\circ\pi\circ(\cores\circ\res)\circ\Delta
\end{equation*}
is taken to be the composition of diagonal, restriction-corestriction, forgetful, and boundary morphisms, \textbf{in that order}. 
\end{prop}

\begin{proof}
This follows from the following claims, whose verifications we leave to the reader.
\begin{itemize}
\item the composition of a boundary, forgetful, or restriction-corestriction map followed by a diagonal map is equal to a composition of a diagonal map followed by maps of the original type,
\item the composition of a boundary map followed by a forgetful map is equal to the composition of forgetful maps followed by a boundary map,
\item the composition of a boundary map followed by a restriction-corestriction map is equal to the composition of restriction-corestriction maps, diagonal maps, and a boundary map, in that order, and
\item the composition of a forgetful map followed by a restriction-corestriction map is equal to a restriction-corestriction map followed by forgetful maps.
\end{itemize}
\end{proof}

\subsection{Integration of zero-cycles}

We remark that, given our results, one can also compute intersection \textit{numbers} between $\cH$-tautological classes of complementary dimension. Indeed, by the results of this section, one only needs to integrate top-degree monomials in $\psi$ and $\kappa$ classes on $\barH_{g,G,\xi}$. However, such classes, up to constant factors, are pulled back from target maps $\delta:\barH_{g,G,\xi}\to\barM_{g',b}$. Therefore, we need only the information of the degrees of target maps $\delta$, given by Hurwitz numbers (see \cite[Theorem 3.19]{schmittvanzelm}), and integrals of monomials in $\psi$ and $\kappa$ classes on $\barM_{g',b}$, given by KdV hierarchy, see \cite{kontsevich}.

%
%
%
%
%

\section{Classical Hurwitz loci}\label{harris_mumford_section}

\subsection{The Harris-Mumford space and its normalization}

Recall from \S\ref{intro} that our motivation is to study cycles on moduli spaces of curves coming from classical Hurwitz spaces, that is, cycles $(\phi_{g/h,d})_{*}(1)$ where $\phi:\barH_{g/h,d}\to\barM_{g,r}$ is the source map. 

The first compactification $\barH_{g/h,d}$ was given by Harris-Mumford \cite{hm} in terms of admissible covers, whose definition we now recall.

\begin{defn}
Let $X,Y$ be curves. Let $b=(2p_a(X)-2)-d(2p_a(Y)-2)$, and let $y_1,\ldots,y_b\in Y$ be such that $(Y,y_1,\ldots,y_b)$ is stable. Then, an \textbf{admissible cover} consists of the data of the stable marked curve $(Y,y_1,\ldots,y_b)$ and a finite morphism $f:X\to Y$ such that:
\begin{itemize}
\item $f(x)$ is a smooth point of $Y$ if and only if $x$ is a smooth point of $X$,
\item $f$ is simply branched over the $y_i$ and \'{e}tale over the rest of the smooth locus of $Y$, and
\item at each node of $X$, the ramification indices of $f$ restricted to the two branches are equal.
\end{itemize}
\end{defn}

The definition can also be generalized in an obvious way to compactify spaces of covers whose ramification is not assumed to be simple. As in the case of Galois covers, spaces of admissible covers admit source and target maps to moduli spaces of curves.

In general, the spaces $\barH_{g/h,d}$ are Cohen-Macaulay, but fail to be smooth. We recall from \cite{hm} the explicit description of their complete local rings. Let $[f:X\to Y]$ be a point of $\barH_{g/h,d}$. Let $y'_1,\ldots,y'_n$ be the nodes of $Y$, and let $y_{1},\ldots,y_{b}\in Y$ be the branch points. Let $\mathbb{C}[[t_1,\ldots,t_{3h-3+b}]]$ be the deformation space of $(Y,y_1,\ldots,y_b)$, so that $t_1,\ldots,t_n$ are smoothing parameters for the nodes $y'_1,\ldots,y'_n$. Let $x_{i,1},\ldots,x_{i,r_i}$ be the nodes of $X$ mapping to $y'_i$, and denote the ramification index of $f$ at $x_{i,j}$ by $a_{i,j}$. 

\begin{prop}[\cite{hm}]\label{adm_defos}
The complete local ring of $\Adm_{g/h,d}$ at $[f]$ is
\begin{equation*}
\mathbb{C}\left[\left[t_1,\ldots,t_{3h-3+b},\{t_{i,j}\}^{1\le i\le n}_{1\le j\le r_i}\right]\right]/\left(t_1=t_{1,1}^{a_{1,1}}=\cdots=t_{1,r_1}^{a_{1,r_1}},\ldots,t_n=t_{n,1}^{a_{n,1}}=\cdots=t_{n,r_n}^{a_{n,r_n}}\right).
\end{equation*} 
\end{prop}

It follows that $[f:X\to Y]\in \barH_{g/h,d}$ is a singular point if and only if there exists a node of $Y$ over which $f$ is ramified at more than one point.

Abramovich-Corti-Vistoli \cite{acv} give a normalization of $\barH_{g/h,d}$ by way of \textit{twisted $S_d$-covers.} The precise definitions will not be important for us, but we recall the idea. 

The space $\wt{\cH}_{g/h,d}^{ACV}$ is defined to be the space of \textit{balanced twisted $S_d$-covers} $f:\wt{X}\to\cY$, where $\cY$ is a marked nodal curve with certain stacky structure at its marked and singular points, and $f$ is a principal $S_d$-bundle. One can then recover an admissible cover of degree $d$ by replacing $\wt{X}$ by $X=\wt{X}/S_{d-1}$ and $\cY$ by its coarse space. 

We then get a map (after restricting to the appropriate components of $\wt{H}_{g/h,d}^{ACV}$)
\begin{equation*}
\nu:\wt{\cH}_{g/h,d}^{ACV}\to \barH_{g/h,d}
\end{equation*}
which is a normalization, see \cite[Proposition 4.2.2]{acv}. If $f:X\to Y$ is a degree $d$ cover of smooth curves, then the unique balanced twisted $S_d$-cover lying over $[f]$ in the normalization is given over the \'{e}tale locus $Y_0\subset Y$ of $f$ by
\begin{equation*}
\wt{X}_0=(X_0\times_{Y_0}\cdots\times_{Y_0} X_0)-\Delta.
\end{equation*}
where the product is $d$-fold and $\Delta$ is the locus on which the points in any two factors are equal. The induced map $\wt{X}_0\to Y_0$ then extends uniquely to a principal $S_d$-bundle $f:\wt{X}\to\cY$.

However, one can just as well understand the normalization $\nu$ in terms of the \textit{scheme-theoretic} $S_d$-covers $\wt{f}':\wt{X}\to \wt{X}/S_{d-1}=Y$, where the target is now identified with the coarse space of $\cY$. Now, $\wt{f}'$ is no longer a principal $S_d$-bundle over the nodes and marked points of $Y$, but rather is an admissible $S_d$-cover in the sense of this paper. We will therefore work with the normalization
\begin{equation*}
\nu:\wt{\cH}_{g/h,d}=\coprod\barH'_{g',S_d,\xi}\to \barH_{g/h,d}
\end{equation*}
which is isomorphic to the normalization by twisted $S_d$-covers, see \cite[Remark 3.6]{schmittvanzelm} for a detailed discussion.

Here, $\barH'_{g',S_d,\xi}$ a moduli space of \textit{possibly disconnected} admissible $S_d$-covers $\wt{X}\to Y$, which, upon applying the restriction-corestriction map remembering $\wt{X}/S_{d-1}$, yield an admissible cover of the form parametrized by $\barH_{g/h,d}$. We will see that the intersection-theoretic tools of \S\ref{int_theory_algos}, where we have dealt with spaces of connected covers, will carry over to this setting.

\subsection{Intersecting Harris-Mumford loci with boundary classes}\label{admissible_plan}

In this section, we will outline an approach to computing intersections of normalized Harris-Mumford loci 
\begin{equation*}
\xymatrix{
\wt{\cH}_{g/h,d} \ar[r]^{\nu} & \barH_{g/h,d} \ar[r]^{\phi} &  \barM_{g,r}
}
\end{equation*}
with boundary classes $\xi_{\Gamma}:\barM_{\Gamma}\to\barM_{g,r}$.

For simplicity, we will assume that $\xi_{\Gamma}$ and $\phi\circ\nu$ define classes of complementary dimension in $\barM_{g,r}$ and ask only for the intersection numbers; a more general discussion is developed in \cite[\S 3.2]{lian_nontaut}.

The computation can be carried out applying the results of \S\ref{res-cores_boundary} directly (after mild changes to allow for disconnected source curves), as we have
\begin{equation*}
\phi\circ\nu=\delta\circ\res^{S_d}_{S_{d-1}}.
\end{equation*}
However, we will opt for a more streamlined approach working with the Harris-Mumford admissible covers directly and invoking the results \S\ref{res-cores_boundary} implicitly; the advantage will be that we will need to compute only a minimal amount of data for the $S_d$-covers parametrized by $\wt{\cH}_{g/h,d}$, which are in general much more complicated than the Harris-Mumford covers.

Let
\begin{equation*}
\xymatrix{
\wt{\cB} \ar[r] \ar[d]^{\nu_{\Gamma}} & \wt{\cH}_{g/h,d} \ar[d]^{\nu} \\
\cB \ar[r] \ar[d] & \barH_{g/h,d} \ar[d]^{\phi} \\
\barM_{\Gamma} \ar[r]^{\xi_\Gamma} & \barM_{g,r} \\
}
\end{equation*}
be the fiber diagram resulting from the intersection of the desired classes.

Note, by the results of \S\ref{res-cores_boundary}, that $\wt{B}$ has underlying reduced space given by the disjoint union of boundary classes indexed by $S_d$-admissible graphs, which is in particular smooth.

We will compute the intersection number of $\xi_\Gamma$ and $\phi\circ\nu$ as follows.

\begin{enumerate}
\item Compute $\cB$ set-theoretically as a disjoint union of boundary classes on $\barH_{g/h,d}$, expressed in terms of dual graphs and covers of curves of lower genus.
\item Compute the \textit{reduced} degree (degree of the induced map of reduced spaces) of $\nu_{\Gamma}$ over any point of $\cB$; this is given by the number of branches in the normalization of its complete local ring, as given by Proposition \ref{adm_defos}.
\item Express the normal bundles of the underlying reduced components of the boundary strata comprising $\wt{\cB}^{\red}$ in terms of $\psi$ classes on the dual graphs of the \textit{target curves} parametrized by $\wt{\cB}$, and thus in terms of the target curves parametrized by $\cB$.
\item Compute the multiplicities of the components appearing in $\wt{\cB}$ in the tangent directions parametrized by the edges of the dual graphs of target curves parametrized by $\cB$.
\item Combine the previous two steps to compute the Segre class of $\wt{\cB}$ in $\wt{\cH}_{g/h,d}$, and apply the excess intersection formula to express the desired intersection number in terms of $\psi$ classes on $\wt{\cB}^{\red}$.
\item Integrate the resulting classes on $\wt{\cB}^{\red}$ by using the fact that they are pulled back from appropriate target maps $\delta:\wt{\cB}\to\barM_{h',b'}$, which factor through $\cB$. Because we have already computed the reduced degree of $\nu_{\Gamma}$, it remains to compute the reduced degrees of target maps out of $\cB$, which are given by Hurwitz numbers.
\end{enumerate}

We emphasize that all of these steps are carried out without computing $\wt{\cB}$ explicitly. Rather, we will express all of the needed data in terms of the combinatorics of the covers appearing in $\cB$. We will see that the only data of the $S_d$-covers that we will need to access is \textit{local}; more specifically, we will only need to know the ramification indices that appear.

For zero-dimensional strata of $\cB$ and $\wt{B}$, the last four steps simplify, as the intersection appears in the correct dimension. Once we have the reduced degree of $\nu_{\Gamma}$ over each zero-dimensional component of $\cB$ and the multiplicity of the corresponding points of $\wt{\cB}$, we immediately have their contributions to the intersection.

\subsection{$d$-elliptic loci in genus 2}

We now give an explicit calculation of the classes of Hurwitz loci in the case $g=2,h=1$ using the method described in the previous section.

Let $\barH_{2/1,d}$ be the Harris-Mumford space parametrizing simply ramified admissible covers $f:X\to Y$, where $X,Y$ are stable curves of genus 2, 1, respectively. By definition, we will mark all $2d-2$ pre-images of the two branch points of $f$.

Let $\phi_d:\barH_{2/1,d}\to\barM_{2,2d-2}$ denote the source map. Let $\phi'_d:\barH_{2/1,d}\to\barM_{2}$ be the post-composition of $\phi$ with the map forgetting all marked points (and contracting non-stable components).

The class of $\phi'_d\in A^1(\barM_2)$ is determined by its intersections with the boundary classes $\Delta_{00},\Delta_{01}\in A^2(\barM_2)$ parametrizing curves of the form shown in Figure \ref{Fig:classes_M2_dim1}, see \cite[Part III]{mumford} or \cite[\S 2.2.3]{lian_qmod}.

\begin{figure}[!htb]
     \includegraphics[width=.75\linewidth]{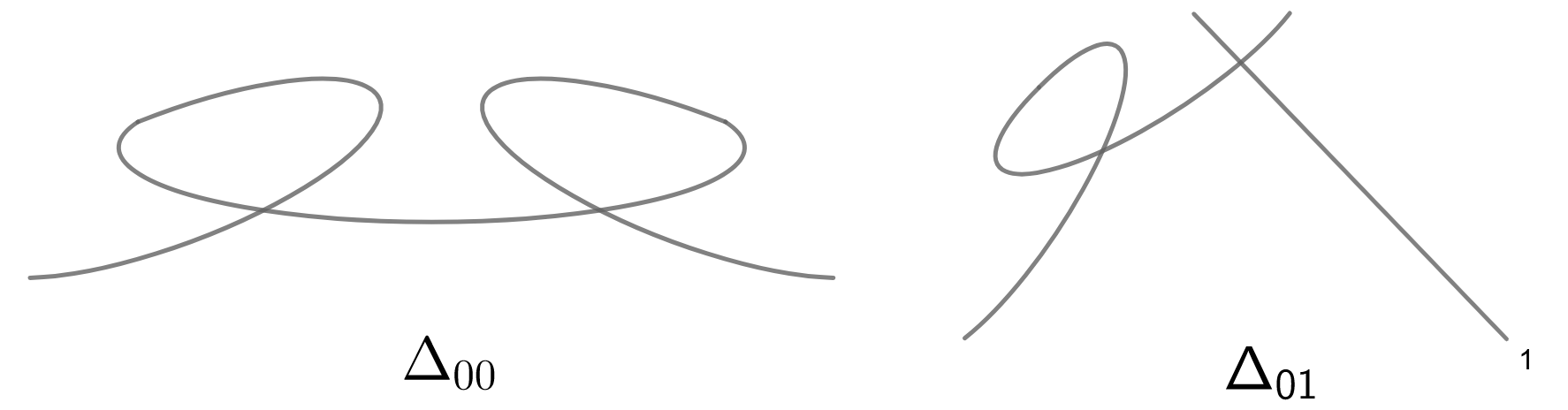}
     \caption{Boundary classes in $A^2(\barM_2)$}\label{Fig:classes_M2_dim1}
\end{figure}

Let $\Gamma_{00},\Gamma_{01}$ denote the corresponding dual graphs, so that $\Delta_{0i}=(\xi_{\Gamma_{0i}})_{*}(1)$.

\begin{thm}\cite[Theorem 1.3]{lian_qmod}\label{d-elliptic_recover}
We have:
\begin{align*}
\int_{\barM_2}(\phi'_{d})_{*}(1)\cdot\Delta_{00}&=4(d-2)!^2(d-1) \sigma_1(d)\\
\int_{\barM_2}(\phi'_{d})_{*}(1)\cdot\Delta_{01}&=2(d-2)!^2\sum_{d_1+d_2=d}\sigma_1(d_1)\sigma_1(d_2),
\end{align*}
where $\sigma_1(m)$ denotes the sum of the positive integer divisors of $m$.

In particular,
\begin{equation*}
\sum_{d\ge1}\frac{1}{(d-2)!^2}\cdot[(\phi'_{d})_{*}(1)]q^d\in A^1(\barM_2)\otimes\bQ[[q]]
\end{equation*}
is a cycle-valued quasimodular form.
\end{thm}

The factors of $(d-2)!^2$ appearing here but not in \cite{lian_qmod} come from our marking of the additional $d-2$ points in each of the fibers over the branch points of a $d$-elliptic cover.

In fact, while we do not give the details, we have been able to check, using the same computation, the following mild generalization:

\begin{thm}\label{d-elliptic_recover_general}
Let $\phi_{d,N}:\barH_{2/1,d}\to\barM_{2,N}$ be a system of maps obtained by forgetting all but a fixed subset of $N$ marked points for all $d$. Then, 
\begin{equation*}
\sum_{d\ge1}\frac{1}{\gamma_d}\left(\int_{\barM_{2,N}}(\phi_{d,N})_{*}(1)\cdot\alpha\right) q^d\in\bQ[[q]]
\end{equation*}
is a quasimodular form for any \textit{tautological} class $\alpha\in R^2(\barM_{2,N})$, where $\gamma_d$ is a combinatorial factor counting relabelings of the forgotten points.
\end{thm}

To prove Theorem \ref{d-elliptic_recover}, we follow the method outlined in \S\ref{admissible_plan}, after pulling the boundary strata of $\barM_2$ back to $\barM_{2,2d-2}$. To carry out step (i), we will need a complete description of all boundary strata of $\barH_{2/1,d}$. This can be computed algorithmically, by considering all possible topological types of the target curve, then all possible ways in which the components can be covered, then glued.

We do not give full classification here, but rather show in Figures \ref{Fig:Delta00_Delta01}-\ref{Fig:Delta000_Delta0} the strata of covers that will appear in our calculations. To label these strata, we follow the notational convention for admissible covers of \cite{lian_qmod}, recording the boundary strata in which the source and target, respectively, lie.

\begin{figure}[!htb]
     \includegraphics[width=.85\linewidth]{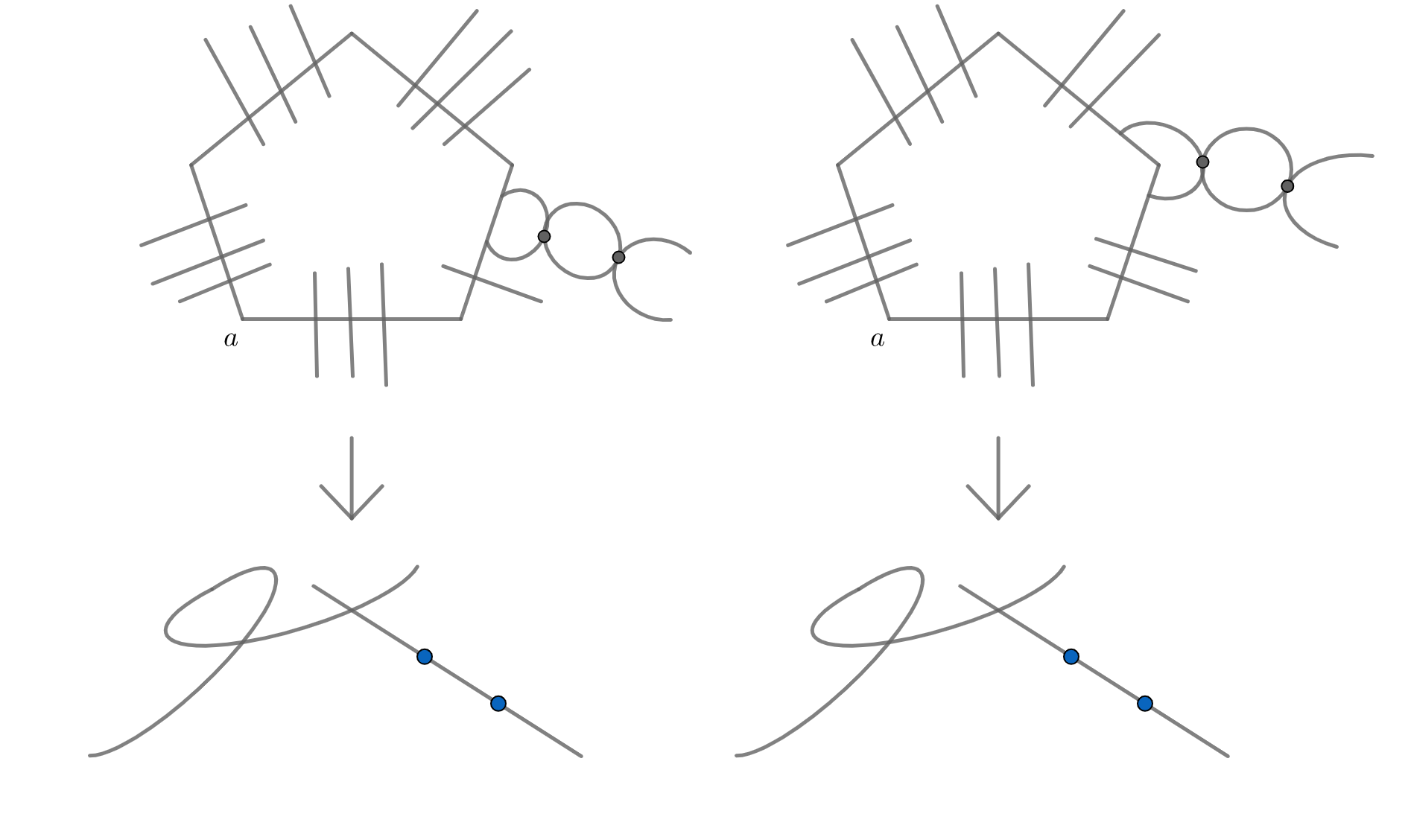}
     \caption{Covers of type $(\Delta_{00},\Delta_{01})$}\label{Fig:Delta00_Delta01}
\end{figure}

\begin{figure}[!htb]
   \begin{minipage}{0.48\textwidth}
     \centering
     \includegraphics[width=.95\linewidth]{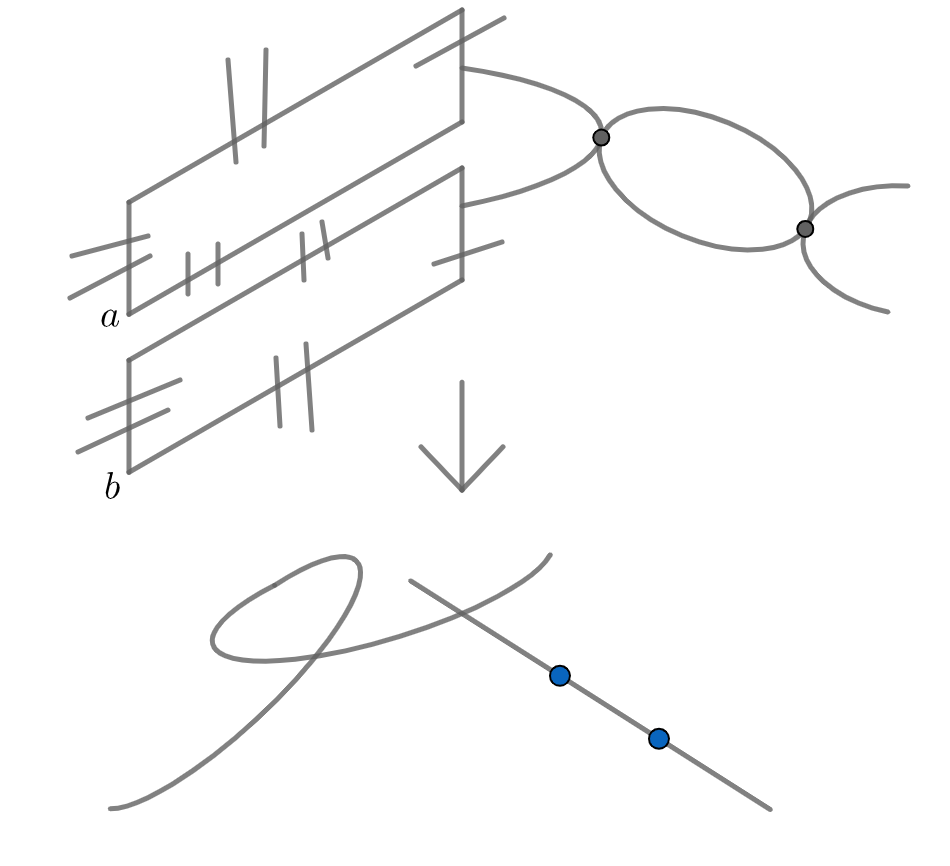}
     \caption{Cover of type $(\Delta_{001},\Delta_{01})$}\label{Fig:Delta001_Delta01}
   \end{minipage}\hfill
   \begin{minipage}{0.48\textwidth}
     \centering
     \includegraphics[width=.8\linewidth]{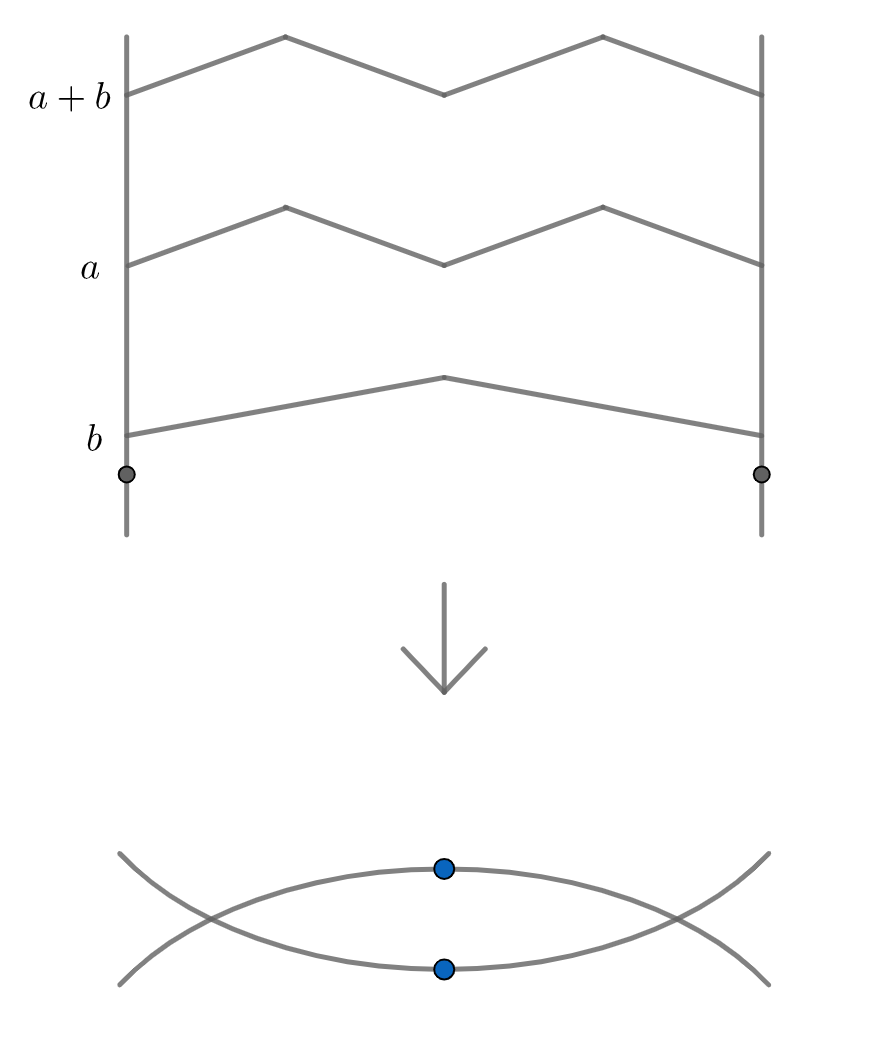}
     \caption{Cover of type $(\Delta_{000},\Delta_{00})$}\label{Fig:Delta000_Delta00}
   \end{minipage}
\end{figure}

\begin{figure}[!htb]
   \begin{minipage}{0.48\textwidth}
     \centering
     \includegraphics[width=.9\linewidth]{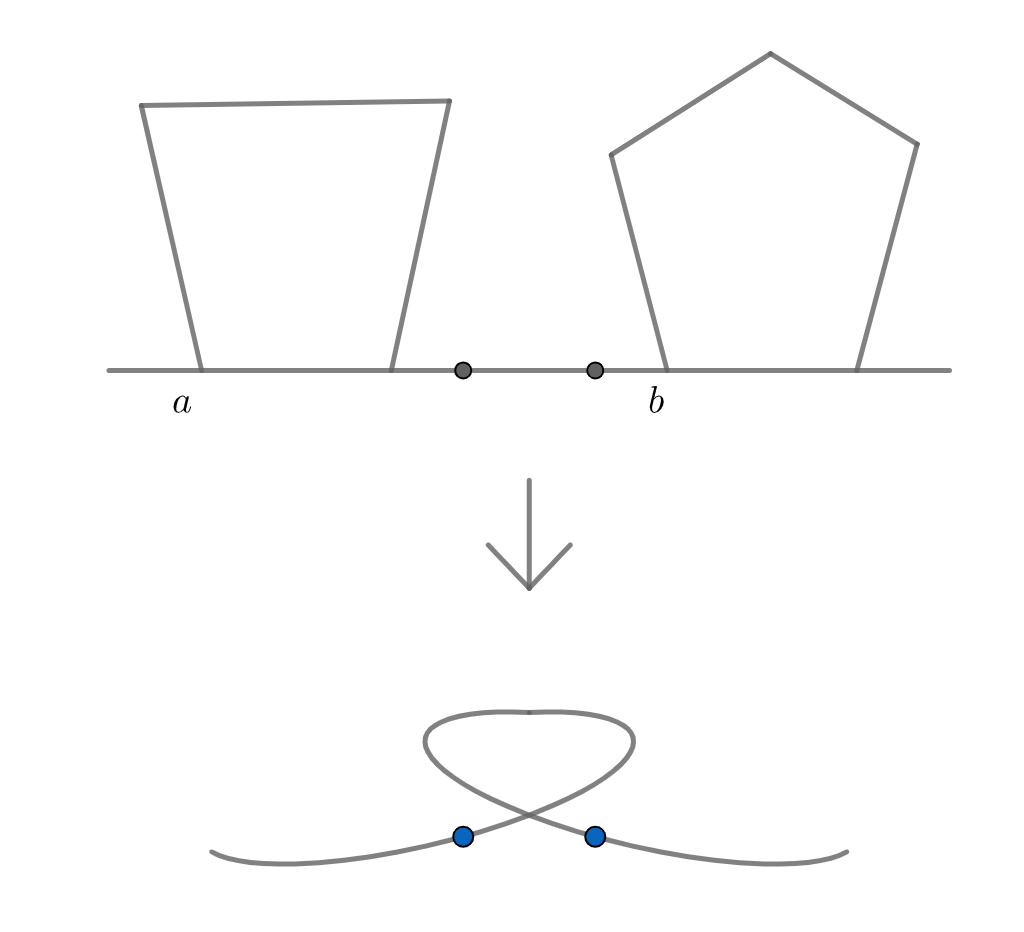}
     \caption{Cover of type $(\Delta_{00},\Delta_0)$}\label{Fig:Delta00_Delta0}
   \end{minipage}\hfill
   \begin{minipage}{0.48\textwidth}
     \centering
     \includegraphics[width=.9\linewidth]{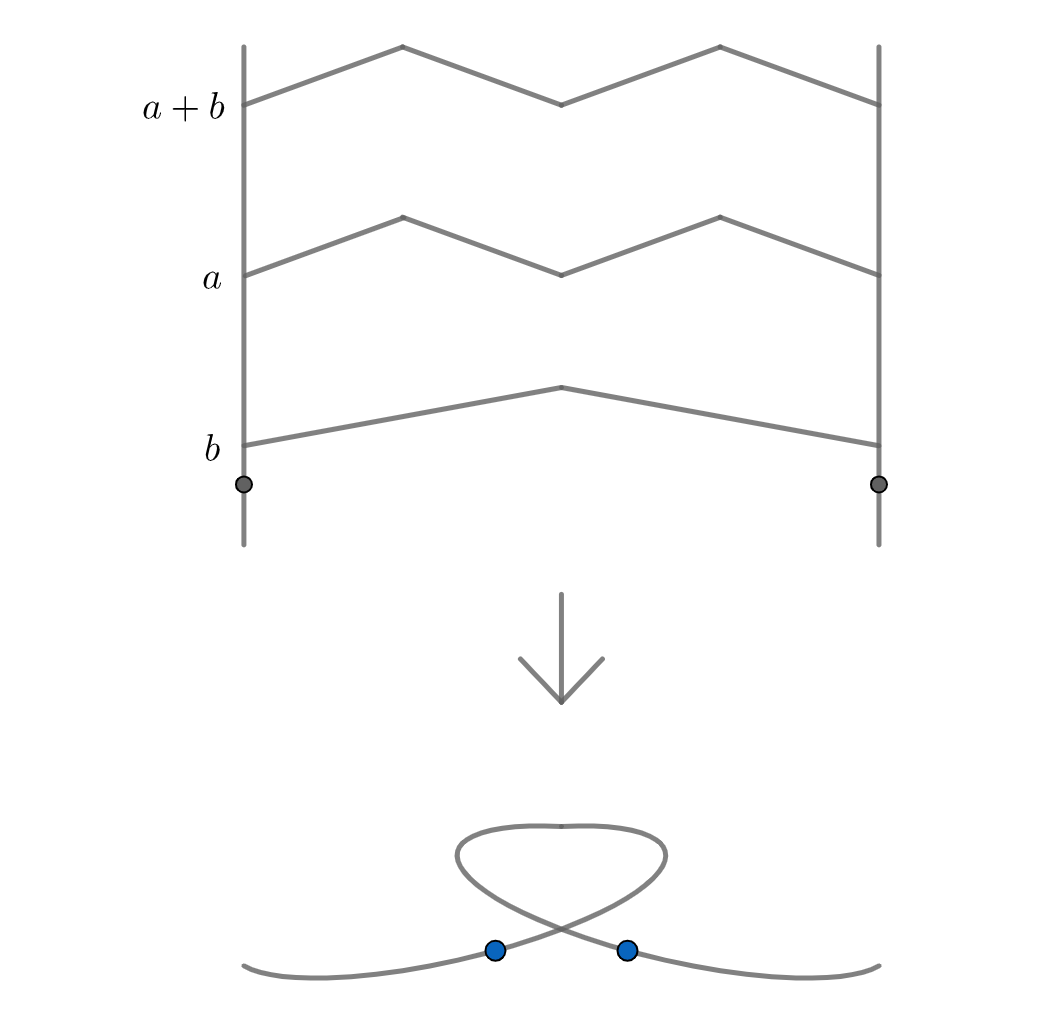}
     \caption{Cover of type $(\Delta_{000},\Delta_{0})$}\label{Fig:Delta000_Delta0}
   \end{minipage}
\end{figure}

\subsubsection{Intersection with $\Delta_{01}$}\label{delta01_intersection} See Figure \ref{Fig:Delta001_Delta01}. We consider the fiber diagram
\begin{equation*}
\xymatrix{
\wt{\cB} \ar[r] \ar[d]^{\nu_{01}}& \wt{\cH}_{2/1,d} \ar[d] \\
\cB \ar[r] \ar[d] & \barH_{2/1,d} \ar[d]^{\phi_d} \\
\coprod \overline{M}_{0,3+r} \times \barM_{1,2d-1-r} \ar[d] \ar[r] & \barM_{2,2d-2} \ar[d] \\
M_{0,3}\times\barM_{1,1} \ar[r]^(0.58){\xi_{01}} & \barM_2
}
\end{equation*}
The coproduct is over all possible ways of distributing the $2d-2$ marked points between the two components.

Following the method of  \S\ref{admissible_plan}, we begin by enumerating the points of $\cB$, which are all given by covers of type $(\Delta_{001},\Delta_{01})$.

Over the nodal component of $Y$, we have an $m$-gon of rational curves whose components each have degree $a$ over $Y$ and an $n$-gon of rational curves whose components each have degree $b$ over $Y$; both components are ramified only at the nodes. (Note that we allow one or both of $m,n$ to be equal to 1, corresponding to a rational nodal component $Y$.) Over the smooth component of $Y$, we have a rational bridge of degree 2, connecting the two polygons, along with $d-2$ components mapping isomorphically.

We distinguish the two cycles of rational curves mapping to $Y$ by declaring that the non-separating node corresponding to the self-edge of $\Gamma_{01}$ is to come from the $m$-gon. We have $m$ distinct choices for the node of $x\in X$ corresponding to the self-edge $(h_0,h'_0)$ of $\Gamma_{01}$ -- these are distinguished by the distance from edge along the dual graph of $X$ to that connecting to the bridge of $X$, as measured in the direction of $h$.

We also have 2 choices to identify the other edge $(h_1,h'_1)$ of $\Gamma_{01}$ with a node of $X$. We must then label the marked points of $X$; this can be done in $(d-2)!^2$ ways, but we also have actions by $a$-th and $b$-th roots of unity, respectively, on the non-distinguished components of the cycles. Therefore, we find
\begin{equation*}
\#\cB(\bC)=\sum_{am+bn=d}2m\cdot\frac{(d-2)!^2}{a^{m-1}b^{n-1}}
\end{equation*}
where we have weighted the count by automorphisms.

Next, we consider the degree of the map $\nu_{01}$ over a point of $\cB$ as above. The complete local ring of $[f:X\to Y]$ is isomorphic to
\begin{equation*}
\bC[[s_1,\ldots,s_m,t_1,\ldots,t_n,u]]/(s_i^a=t_j^b),
\end{equation*}
where the $s_i,t_j$ are smoothing parameters of the vertices of the $m$-gon and $n$-gon, respectively, and $u$ is a smoothing parameter for the separating nodes. The normalization will thus have $a^{m-1}b^{n-1}\gcd(a,b)$ branches. It follows that
\begin{equation*}
\#\wt{\cB}(\bC)=\sum_{am+bn=d}2 m(d-2)!^2\gcd(a,b).
\end{equation*}

Finally, we consider the multiplicities of the points of $\wt{\cB}$. Let $\wt{f}:\wt{X}\to Y$ be any point of $\wt{H}_{2/1,d}$ lying over $[f:X\to Y]\in\barH_{2/1,d}$. Then, the common ramification index of the (possibly disconnected) $S_d$-cover $\wt{f}$ over the non-separating node of $Y$ is $\lcm(a,b)$. On the other hand, the ramification index of $f$ at the chosen node $x\in X$ corresponding to the self-edge of $\Gamma_{01}$ is $a$. Therefore, the common ramification index over $x$ in the $S_{d-1}$-cover $\wt{X}\to X$ is equal to $\lcm(a,b)/a$.

The local computation of Lemma \ref{cores_boundary_multiplicities} shows, then, that $[\wt{f}]\in\wt{\cB}$ appears with this multiplicity. Putting everything together, we obtain 
\begin{align*}
\int_{\barM_2}(\phi'_{d})_{*}(1)\cdot\Delta_{01}&=\sum_{am+bn=d}2 m(d-2)!^2\gcd(a,b)\cdot\frac{\lcm(a,b)}{a}\\
&=2(d-2)!^2\sum_{am+bn=d}mb\\
&=2(d-2)!^2\sum_{d_1+d_2=d}\sigma_1(d_1)\sigma_1(d_2).
\end{align*}

\subsubsection{Intersection with $\Delta_{00}$}

We now consider the fiber diagram
\begin{equation*}
\xymatrix{
\wt{\cB} \ar[r] \ar[d]^{\nu_{00}} & \wt{\cH}_{2/1,d} \ar[d] \\
\cB \ar[r] \ar[d] & \barH_{2/1,d} \ar[d]^{\phi_d} \\
\overline{M}_{0,2d+2} \ar[d] \ar[r]^{\xi'_{00}} & \barM_{2,2d-2} \ar[d] \\
\overline{M}_{0,4}\ar[r]^{\xi_{00}} & \barM_2
}
\end{equation*}

The space $\cB$, and thus also $\wt{\cB}$, consists of four families of components, corresponding to the topological types $(\Delta_{00},\Delta_{01})$, $(\Delta_{000},\Delta_{00})$, $(\Delta_{00},\Delta_0)$, and $(\Delta_{000},\Delta_0)$. Note that, while the remaining topological type, $(\Delta_{001},\Delta_{01})$, studied in the previous section, has source curve in the locus $\Delta_{00}$, the chosen non-separating nodes of $X$ map to the same node of $Y$, so upon pullback to $\wt{\cB}$, the genericity condition of Lemma \ref{restriction_boundary_cartesian} will be violated. More concretely, such covers are specializations of those of type $(\Delta_{00},\Delta_0)$, so already appear on another of the enumerated coomponents.

We will label the families of components of $\cB$ and $\wt{\cB}$ with subscripts according to their topological type (for example, $\cB_{(\Delta_{00},\Delta_{01})}$ is the union of components of $\cB$ corresponding to covers of type $(\Delta_{00},\Delta_{01})$). We first consider zero-dimensional components of the intersection.\\

\underline{Type $(\Delta_{00},\Delta_{01})$.} See Figure \ref{Fig:Delta00_Delta01}. We have a cycle of $m$ rational curves of degree $a$ over the nodal component of $Y$, a rational bridge of degree 2 over the smooth component, and all other components mapping isomorphically to the smooth component. 

We first compute the number of points in $\cB_{(\Delta_{00},\Delta_{01})}$. Consider first the identification of nodes of $X$ with edges of $\Gamma_{00}$. Note that the genericity condition of Lemma \ref{restriction_boundary_cartesian} requires us to choose nodes over both of those of $Y$, so one must come from the cycle of rational curves, and the other from one of two attachment points of the rational bridge. We have two choices for which of the edges corresponds to which type of node.

We begin with an $m$-gon of rational curves with a distinguished node corresponding to the prescribed edge of $\Gamma_{00}$, and a fixed orientation, determining an identification of half-edges with branches.

Next, consider the node of $X$ identified with the other edge of $\Gamma_{00}$; we have $m$ distinct choices for the component of the $m$-gon on which it lies, and two choices for the identification of half-edges with branches. If the rational bridge connects two different components of the $m$-gon, then there are $m-1$ choices for the other component, while if it connects the chosen component to itself, there are $a-1$ choices for the second attachment point.

We next have $(d-2)!^2$ labellings of the unramified marked points, but we must correct for the cyclic actions on the rational components with only two nodes, of which there are $m-2$ if the rational bridge connects two distinct components and $m-1$ otherwise. We find:
\begin{equation*}
\#\cB_{(\Delta_{00},\Delta_{01})}(\bC)=\sum_{am=d}4(d-2)!^2\left[\frac{m(m-1)}{a^{m-2}}+\frac{(a-1)m}{a^{m-1}}\right]
\end{equation*}
points, where as usual we weight the count by automorphisms.

A cover $f:X\to Y$ of type $(\Delta_{00},\Delta_{01})$ has complete local ring
\begin{equation*}
\bC[[s_1,\ldots,s_m,t]]/(s_1^a=\cdots s_m^a),
\end{equation*}
so
\begin{equation*}
\#\wt{\cB}_{(\Delta_{00},\Delta_{01})}(\bC)=\sum_{am=d}4(d-2)!^2\left[\frac{m(m-1)}{a^{m-2}}+\frac{(a-1)m}{a^{m-1}}\right]a^{m-1}
\end{equation*}

Finally, note that if $[\wt{f}:\wt{X}\to Y]$ is a point of $\wt{\cB}$ over $[f]$, then $f$ and $\wt{f}$ will have the same ramification indices of $a$ and 1 over the nodes of $Y$. Thus, $\wt{X}\to X$ is unramified over the nodes, and in particular, $\wt{\cB}_{(\Delta_{00},\Delta_{01})}$ is reduced. We therefore get a contribution from $\#\wt{\cB}_{(\Delta_{00},\Delta_{01})}$ to the desired intersection number of 
\begin{align*}
&\quad\sum_{am=d} 4(d-2)!^2\left[\frac{m(m-1)}{a^{m-2}}+\frac{(a-1)m}{a^{m-1}}\right] a^{m-1}\\
&=\sum_{am=d}4(d-2)!^2[am(m-1)+(a-1)m]\\
&=4(d-2)!^2\sum_{am=d}m(am-1)\\
&=4(d-2)!^2(d-1)\sigma_1(d).
\end{align*}

\underline{Type $(\Delta_{000},\Delta_{00})$.} See Figure \ref{Fig:Delta000_Delta00}. 

\begin{lem}\label{hurwitz_(a+b)_(a,b)_2}
Given $a,b\ge1$, there is a unique degree $a+b$ cover $g:\bP^1\to\bP^1$ that is totally ramified over 0, simply ramified over $\infty$, and has ramification profile $(a,b)$ over 1. Here, we consider the two points in $f^{-1}(1)$ to be distinguished if $a=b$.
\end{lem}

\begin{proof}
The number of such covers is given by the number of pairs $(\sigma_0,\sigma_\infty)\in (S_{a+b})^2$ such that $\sigma_0$ is an $(a+b)$-cycle, $\sigma_\infty$ is a transposition, and $\sigma_\infty\sigma_0$ has cycle type $(a,b)$, up to simultaneous conjugation. It is straightforward to check that the unique such choice is
\begin{equation*}
(\sigma_0,\sigma_\infty)=((12\cdots(a+b)),(1(a+1)),
\end{equation*}
\end{proof}

We consider pointed admissible covers $f:X\to Y$, where $Y$ is a union of rational curves $Y_1,Y_2$ intersecting at two points $y_{+},y_{-}$, and $X$ contains two rational curves $X_1,X_2$ mapping to $Y_1,Y_2$ in a unique way as in Lemma \ref{hurwitz_(a+b)_(a,b)_2}. We take $y_{+}$ to be the node of $Y$ over which the $X_j$ are totally ramified, and $Y_j$ to be the component containing the $j$-th branch point of $f$. The components $X_1$ and $X_2$ are then joined by three chains of rational curves totally ramified over the nodes of $Y$. Note that intersecting components of $X$ must map to opposite components of $Y$, and different nodes on components other than $X_1,X_2$ must map to different nodes of $Y$.

Suppose that the chains of rational curves with degrees $a+b,a,b$ over one of the $Y_j$ contains $k,m,n$ pre-images, respectively, of $y_{+}$. Then, these chains contain $k-1,m+1,n+1$ pre-images, respectively, of $y_{-}$. Here, we must have $k\ge 1$ and $m,n\ge0$. In order for the degree of $f$ to be equal to $d$, we require $(a+b)k+am+bn=d$.

We first choose nodes of $X$ which will correspond to the edges of $\Gamma_{00}$. By the usual genericity condition, one node each must be chosen above the nodes $y_{\pm}$. First, suppose that the two nodes chosen have ramification indices $a$ and $b$ over $Y$; we will distinguish the chains of these ramification indices by identifying them with the two edges of $\Gamma_{00}$. 

Suppose first that the node of ramification index $a$ is chosen over $y_{+}$, and that of ramification index $b$ is chosen over $y_{-}$. Then, there are $m(n+1)$ ways to choose the two nodes and $2\cdot2=4$ ways to identify their branches with half-edges. There are $(d-2)!^2$ ways to label the unramified marked points, and $f$ has an automorphism group of order $(a+b)^{2k-2}a^{2m}b^{2n}$.

We compute the reduced degree of $\nu_{00}$ over $[f]$ to be
\begin{equation*}
\gcd(a,b)^4(a+b)^{2k-3}a^{2m-1}b^{2n-1}.
\end{equation*}
and the intersection multiplicity at any point $[\wt{f}]$ over $[f]$ to be
\begin{equation*}
\frac{\lcm(a+b,a,b)}{a}\cdot\frac{\lcm(a+b,a,b)}{b}=\frac{(a+b)^2a^2b^2}{ab\gcd(a,b)^4}.
\end{equation*}

Combining, we get a total contribution of
\begin{align*}
&(d-2)!^2\sum_{k(a+b)+ma+nb=d}4 m(n+1)\frac{\gcd(a,b)^4(a+b)^{2k-3}a^{2m-1}b^{2n-1}}{(a+b)^{2k-2} a^{2m} b^{2n}}\frac{(a+b)^2a^2b^2}{ab\gcd(a,b)^4}\\
=&(d-2)!^2\sum_{(a+b)k+am+bn=d}4m(n+1)(a+b).
\end{align*}
to the desired intersection number.

Similarly, if the nodes with ramification indices $a,b$ live over $y_{-},y_{+}$, respectively, we get a contribution of
\begin{equation*}
\sum_{(a+b)k+am+bn=d}4 (m+1)n(a+b).
\end{equation*}

On the other hand, we may also choose one of our nodes to have ramification index $a+b$. In this case, to distinguish $a$ and $b$, we declare that the other chosen node has ramification index $a$. We then have two ways to identify these nodes with the edges of $\Gamma_{00}$. Analogous computations to the above show that we get a total contribution from these pointed covers of
\begin{equation*}
(d-2)!^2\sum_{(a+b)k+am+bn=d}8 [k(m+1)+(k-1)m] b.
\end{equation*}

Combining everything, covers of type $(\Delta_{000},\Delta_{00})$ contribute
\begin{equation*}
(d-2)!^2\sum_{\substack{(a+b)k+am+bn=d \\ k\ge1, m,n\ge0}}[(8mn+4m+4n)(a+b)+(16km+8k-8m)b]
\end{equation*}
to the intersection of $\phi_d$ and $\xi_{00}$.\\

Now, we consider one-dimensional components of the intersection; here will need to implement the full program of \S\ref{admissible_plan}.\\

\underline{Type $(\Delta_{00},\Delta_{0})$.} See Figure \ref{Fig:Delta00_Delta0}. 

\begin{lem}\label{(a,b),(a,b),2,2}
Given $a,b\ge1$, the number of degree $a+b$ covers $g:\bP^1\to\bP^1$ simply ramified over two fixed points and with ramification profile $(a,b)$ over another two fixed points is equal to $2\max(a,b)$. Here, we consider the points in the latter two fibers distinguished if $a=b$.

Equivalently, if $\barH_{(a,b)}$ denotes the Harris-Mumford space of admissible covers of this type, the degree of the target map $\delta_{(a,b)}:\barH_{(a,b)}\to\overline{M}_{0,4}$ is $2\max(a,b)$.
\end{lem}

\begin{proof}
We degenerate the target to a union of two copies of $\bP^1$ attached at a node, so that each component contains one branch point of each type. Without loss of generality, assume that $a\ge b$. There are two admissible covers of the desired type over this target:
\begin{itemize}
\item two covers of the type in Lemma \ref{hurwitz_(a+b)_(a,b)_2}, glued together at the ramification and branch points of order $a+b$,
\item two covers of the type in Lemma \ref{hurwitz_(a+b)_(a,b)_2}, where the pair $(a,b)$ is replaced by $(a-b,b)$, glued together at the ramification and branch points of order $a-b$, along with one additional component covering each component of the target via the map $x\mapsto x^b$.
\end{itemize}
On the Harris-Mumford stack $\barH_{(a,b)}$, the first cover appears with multiplicity $a+b$ over the desired target. The second appears with multiplicity $b^2(a-b)$, but the two components of the target with the action by a cyclic group of order $b$ introduce a correction factor of $1/b^2$. Summing these contributions yields the claim.
\end{proof}

\begin{enumerate}
\item Suppose we have integers $a,m,b,n\ge1$ with $am+bn=d$. Then, $\barH_{(a,b)}$ gives rise to components of $\cB_{(\Delta_{00},\Delta_{0})}$ in the following way. Let $g:X_0\to Y_0$ be a cover as in Lemma \ref{(a,b),(a,b),2,2}, where $X_0$ and $Y_0$ are allowed to be singular. We attach chains of rational curves of lengths $m-1$ and $n-1$, where $m,n\ge1$, totally ramified over $Y_0$ of degrees $a$ and $b$, to the two points each of $X_0$. Then, we glue the images of these nodes on $Y_0$ to get the curve $Y$.

Suppose that we have an identification of the half-edges of $\Gamma_{00}$ with branches of nodes of $X_0$; the pairs $(a,m)$ and $(b,n)$ may be distinguished by this identification, and we get orientations on both chains. The $\Gamma_{00}$-structure on $X$ gives us a map $\barH_{(a,b)}\to \overline{M}_{0,2d+2}$ and identifies $\barH_{(a,b)}$ with a component of $\cB^{\red}$. 

For a given choice of $a,m,b,n$, the number of components of $\cB_{(\Delta_{00},\Delta_{0})}$ associated to $\barH_{(a,b)}$ in the way we have described is
\begin{equation*}
2mn\cdot\frac{(d-2)!^2}{a^{m-1}b^{m-1}}.
\end{equation*}

Here, the factor of 2 comes from a choice of identification of half-edges and branches at one of the two edges of $\Gamma_{00}$; the two choices give distinct components, because both nodes map to the same node of $Y$, so their branches may be distinguished in pairs. The factor of $mn$ comes from choices of nodes along each chain of rational curves, the factor of $(d-2)!^2$ comes from choices of labels of unramified points, and the denominator comes from automorphisms of $X$.

\item The reduced degree of $\nu_{00}$ is the number of branches in the normalization over a general cover $[f:X\to Y]$ of type $(\Delta_{00},\Delta_0)$. The complete local ring of such a cover is of the form
\begin{equation*}
\bC[[s_1,\ldots,s_m,t_1,\ldots,t_n,v]]/(s_i^a=t_j^b),
\end{equation*}
where the $s_i$ and $t_j$ are deformation parameters for the nodes of ramification indices $a,b$, respectively, and the deformation parameter $v$ comes from that on the target $\overline{M}_{0,4}$. Therefore, the number of branches upon normalization is $a^{m-1}b^{n-1}\gcd(a,b)$.

\item Fix $a,m,b,n$ as before, and let $\wt{\cB}_{(a,b)}\subset \wt{\cB}_{(\Delta_{00},\Delta_{0})}$ be the union of the components lying over $\wt{B}$ identified with $\barH_{(a,b)}$ as above. Then, a point of $\wt{\cB}_{(a,b)}$ is an $S_d$-cover constructed from Galois covers $f_0:\wt{X}_0\to Y_0\cong\bP^1$ ramified over four points with common ramification indices $2,2,\lcm(a,b),\lcm(a,b)$, by gluing together copies of $\wt{X}_0$ at ramification points of order $\lcm(a,b)$.

The Chern class of the normal bundle of $\wt{\cB}_{(a,b)}^{\red}$ in $\wt{\cH}_{2/1,d}$ is then given by $1-\psi$, where $\psi=\psi_{h}+\psi_{h'}$ is the sum of the $\psi$ classes on $\wt{\cB}_{(a,b)}^{\red}$ corresponding to the $S_d$-orbits of branches with ramification order $\lcm(a,b)$.

Let $\wt{\delta}_{(a,b)}:\wt{\cB}_{(a,b)}^{\red}\to\overline{M}_{0,4}$ be the target map. Then, by Lemma \ref{psi_kappa_comparison_forgetful}, we have
\begin{equation*}
\psi=\frac{2}{\lcm(a,b)}\cdot\wt{\delta}_{(a,b)}^{*}\psi_t,
\end{equation*}
where $\psi_t$ is any of the rationally equivalent $\psi$ classes on $\overline{M}_{0,4}$.

\item We claim that the multiplicity of $\wt{\cB}_{(a,b)}$ is given by $\lcm(a,b)/\max(a,b)$. Indeed, the map on complete local rings at $[f:X\to Y]$ induced by $\wt{\phi}':\wt{H}_{2/1,d}\to\barM_{2,2d-2}$ takes the deformation parameters corresponding to the two edges of $\Gamma_{00}$ to $t^{\lcm(a,b)/a}$ and $t^{\lcm(a,b)/b}$, where $t$ is the common deformation parameter of the nodes on $\wt{\cH}_{2/1,d}$. Thus, we get the claimed multiplicity in the direction corresponding to the smoothing of the nodes.

\item Analytically locally, $\wt{\cB}_{(a,b)}$ is a Cartier divisor in $\wt{\cH}_{2/1,d}$ cut out by the $\lcm(a,b)/\max(a,b)$-th power of an equation in the normal direction corresponding to the smoothing of the nodes. It follows that

\begin{equation*}
s(\wt{\cB}_{(a,b)},\wt{\cH}_{2/1,d})|_{\wt{\cB}_{(a,b)}^{\red}}=c(N_{\wt{\cB}_{(a,b)}/\wt{\cH}_{2/1,d}})^{-1}|_{\wt{\cB}_{(a,b)}^{\red}}=\frac{\lcm(a,b)}{\max(a,b)}+\left(\frac{\lcm(a,b)}{\max(a,b)}\right)^2\psi.
\end{equation*}

\begin{rem}
In the general setting, a component of $\wt{\cB}$ mapping to a normalized Harris-Mumford space $\wt{\cH}$ will always have the local picture of the closed embedding 
\begin{equation*}
\Spec\bC[[t_1,\ldots,t_r]]/(t_1^{c_1},\ldots,t_r^{c_r})\to \bC[[t_1,\ldots,t_n]],
\end{equation*}
where $t_1,\ldots,t_r$ correspond to deformation parameters at nodes of the target curve $\wt{X}$. The Segre class of such an embedding in this general situation may then be computed using the formula of \cite[\S 3.2]{aluffi}.
\end{rem}

By the excess intersection formula, the contribution of $\wt{\cB}_{(a,b)}$ to the desired intersection number is
\begin{equation*}
\int_{\cB_{(a,b)}^{\red}}c(N_{\overline{M}_{0,2d+2}/\barM_{2,2d-2}})|_{\cB_{\red}}\cdot s(\cB,\cH_{2/1,d})|_{\cB_{(a,b)}^{\red}}.
\end{equation*}
We have that $c(N_{\overline{M}_{0,2d+2}/\barM_{2,2d-2}})$ is given by $\psi$ class contributions corresponding to the two chosen nodes of $X$, with ramification indices $a$ and $b$. More precisely, by Lemma \ref{psi_kappa_comparison_forgetful}, we have
\begin{equation*}
c(N_{\overline{M}_{0,2d+2}/\barM_{2,2d-2}})|_{\cB_{(a,b)}^{\red}}=\left(1-\frac{\lcm(a,b)}{a}\psi\right)\left(1-\frac{\lcm(a,b)}{b}\psi\right)\\
\end{equation*}

\item We are now ready to put everything together. Continuing from above, the desired contribution is
\begin{align*}
&\int_{\cB_{(a,b)}^{\red}}\left(1-\frac{\lcm(a,b)}{a}\psi\right)\left(1-\frac{\lcm(a,b)}{b}\psi\right)\left(\frac{\lcm(a,b)}{\max(a,b)}+\left(\frac{\lcm(a,b)}{\max(a,b)}\right)^2\psi\right)\\
=&\frac{\lcm(a,b)^2}{\max(a,b)}\cdot\left(-\frac{1}{a}-\frac{1}{b}+\frac{1}{\max(a,b)}\right)\int_{\cB_{(a,b)}^{\red}}\psi\\
&=-\frac{\lcm(a,b)^2}{\min(a,b)\max(a,b)}\cdot\int_{\cB_{(a,b)}^{\red}}\psi\\
&=-\frac{\lcm(a,b)}{\gcd(a,b)}\int_{\cB_{(a,b)}^{\red}}\psi.
\end{align*}

On the other hand, we found in step (iii) that
\begin{align*}
\int_{\cB_{(a,b)}^{\red}}\psi&=\frac{2}{\lcm(a,b)}\int_{\cB_{(a,b)}^{\red}}\wt{\delta}_{(a,b)}^{*}\psi_t\\
&=\frac{2}{\lcm(a,b)}\cdot\deg(\wt{\delta}_{(a,b)})\int_{\overline{M}_{0,4}}\psi_t\\
&=\frac{2}{\lcm(a,b)}\cdot\deg(\delta_{(a,b)})\cdot\deg(\nu_{00})\int_{\overline{M}_{0,4}}\psi_t\\
&=\frac{2}{\lcm(a,b)}\cdot2\max(a,b)\cdot a^{m-1}b^{n-1}\gcd(a,b),
\end{align*}
where we by $\deg(\nu_{00})$ we mean the \textit{reduced} degree we computed in step (ii), that is, the degree of the induced map on reduced spaces, measured by the number of branches of the normalization. Therefore, the total contribution from $\cB_{(a,b)}$ is
\begin{equation*}
-4\max(a,b)\cdot a^{m-1}b^{n-1}.
\end{equation*}

Finally, we must multiply by the number of components of $\cB_{(\Delta_{00},\Delta_{0})}$ associated to $a,m,b,n$ as computed in step (i), and sum over all $a,m,b,n$. This gives the final excess contribution from covers of type $(\Delta_{00},\Delta_{0})$ of
\begin{equation*}
-(d-2)!^2\sum_{\substack{am+bn=d \\ m,n\ge0}}8\max(a,b) mn
\end{equation*}

\end{enumerate}

\underline{Type $(\Delta_{000},\Delta_{0})$.} See Figure \ref{Fig:Delta000_Delta0}. 

We follow the same method as before, so we omit the details in this case. We first assume that the $\Gamma_{00}$-structure on $X$ identifies two nodes of ramification indices $a,b$ with the edges of $\Gamma_{00}$.

Fix integers $a,m,b,n,k\ge1$. 
\begin{enumerate}
\item We form a cover $f:X\to Y$ of type $(\Delta_{000},\Delta_{0})$ by starting with two (distinguishable) components $X_1,X_2\subset X$ mapping to $\bP^1$ via a cover as in Lemma \ref{hurwitz_(a+b)_(a,b)_2}. Then, we connect these components by three chains of lengths $k-1,m-1,n-1$, respectively, totally ramified over the same $\bP^1$. Finally, we glue the two images of the nodes of $X$ on the target $\bP^1$ to get $Y$.

In this way, we get a total of
\begin{equation*}
4mn\cdot\frac{(d-2)!^2}{a^{m-1}b^{n-1}(a+b)^{k-1}}
\end{equation*}
components of $\cB_{(\Delta_{000},\Delta_{0})}$.

\item Normalizing the complete local ring of $[f]$ as above shows that the reduced degree of $\nu_{00}$ above such a cover is
\begin{equation*}
\gcd(a,b)^2a^{m-1}b^{n-1}(a+b)^{k-1}.
\end{equation*}

\item If $\wt{\cB}_{(a,b)}$ is the union of components of $\wt{\cB}_{(\Delta_{000},\Delta_{0})}$ lying over one of $\cB_{(\Delta_{000},\Delta_{0})}$ as above, then the covers parametrized by $\wt{\cB}_{(a,b)}$ are obtained by gluing Galois covers of ramification indices $2,2,\lcm(a,b,a+b),\lcm(a,b,a+b)$ in a similar way to the case of covers of type $(\Delta_{00},\Delta_0)$.

The Chern class of the normal bundle of $\wt{\cB}_{(a,b)}^{\red}$ is given by $1-\psi$, where $\psi=\psi_h+\psi_{h'}$ is the sum of $\psi$ classes on $\wt{\cB}_{(a,b)}^{\red}$ corresponding to the $S_d$-orbits of branches with ramification order $\lcm(a,b,a+b)=ab(a+b)/\gcd(a,b)$.

Let $\wt{\delta}_{(a,b)}:\wt{\cB}^{\red}_{(a,b)}\to\overline{M}_{0,4}$ be the target map. Then, we have
\begin{equation*}
\psi=\frac{2\gcd(a,b)}{ab(a+b)}\wt{\delta}_{(a,b)}^{*}\psi_t.
\end{equation*}

\item The multiplicity of $\wt{\cB}_{(a,b)}$ is given by
\begin{equation*}
\frac{\lcm(a,b)\cdot(a+b)}{\gcd(a,b)\cdot\max(a,b)},
\end{equation*}
and occurs in the direction of smoothing of the nodes of an $S_d$-cover $\wt{X}\to Y$ parametrized by $\wt{\cB}_{(a,b)}$

\item We have
\begin{equation*}
s(\wt{\cB}_{(a,b)},\wt{\cH}_{2/1,d})|_{\wt{\cB}_{(a,b)}^{\red}}=\frac{\lcm(a,b)(a+b)}{\gcd(a,b)\max(a,b)}+\left(\frac{\lcm(a,b)(a+b)}{\gcd(a,b)\max(a,b)}\right)\psi
\end{equation*}
and
\begin{equation*}
c(N_{\overline{M}_{0,2d+2}/\barM_{2,2d-2}})|_{\cB_{(a,b)}^{\red}}=\left(1-\frac{b(a+b)}{\gcd(a,b)}\psi\right)\left(1-\frac{a(a+b)}{\gcd(a,b)}\psi\right)\\
\end{equation*}

\item Applying the excess intersection formula as before and incorporating the above data yields a total contribution of
\begin{equation*}
-8(d-2)!^2\cdot\sum_{\substack{(a+b)k+am+bn \\ k,m,n\ge1}}(a+b)mn.
\end{equation*}

\end{enumerate}

Finally, we must also consider pointed curves where the $\Gamma_{00}$-structure on $X$ identifies one of the nodes of ramification index $a+b$ with an edge of $\Gamma_{00}$. Following the steps above, this yields an additional contribution of
\begin{equation*}
-16(d-2)!^2\cdot\sum_{\substack{(a+b)k+am+bn \\ k,m,n\ge1}}bkm.
\end{equation*}

\subsection{Proof of Theorem \ref{d-elliptic_recover}}

We will need the following lemma. We thank David Yang for providing the proof.

\begin{lem}\label{david_lemma}
We have
\begin{equation*}
\sum_{\substack{am+bn=d \\ a,b,m,n>0}}(mn-am)\min(a,b)=0.
\end{equation*}
and 
\begin{equation*}
\sum_{\substack{am+bn=d \\ a,b,m,n>0}}(mn-bn)\min(a,b)=0.
\end{equation*}
\end{lem}

\begin{proof}
The two statements are clearly equivalent. We prove the first. We have
\begin{align*}
\sum_{\substack{am+bn=d \\ a,b,m,n>0}}m(n-1)\min(a,b)&=\sum_{\substack{am+bn=d \\ a,b,m,n>0}} m(n-a)(\min(a,b)-b)\\
&=\sum_{\substack{am+bn=d \\ a,b,m,n>0 \\ a<b}} m(n-a)(a-b).
\end{align*}
On the index set in the last sum, we have an involution
\begin{equation*}
(a,b,m,n)\mapsto (n,n+m,b-a,a)
\end{equation*}
on which the terms being summed become $(b-a)(a-n)(-m)$, that is, they change sign. The conclusion is now immediate.
\end{proof}

\begin{proof}[Proof of Theorem \ref{d-elliptic_recover}]
We have already obtained the intersection number with $\Delta_{01}$ in \S\ref{delta01_intersection}. It remains to combine the above contributions to the intersection number with $\Delta_{00}$. We have the desired intersection number coming from covers of type $(\Delta_{00},\Delta_{01})$; it suffices to show now that the other three contributions cancel. Ignoring the common factors of $(d-2)!^2$, we have
\begin{align*}
-&8\left(\sum_{\substack{am+bn=d \\ m,n\ge1}}\max(a,b)mn\right)-8\left(\sum_{\substack{(a+b)k+am+bn=d \\ k,m,n\ge1}}(a+b)mn\right)-16\left(\sum_{\substack{(a+b)k+am+bn \\ k,m,n\ge1}}bkm\right)\\
&+\sum_{\substack{(a+b)k+am+bn=d \\ k\ge1, m,n\ge0}}[(8mn+4m+4n)(a+b)+(16km+8k-8m)b]\\
=-&8\left(\sum_{\substack{am+bn=d \\ m,n\ge1}}\max(a,b)mn\right)+4\left(\sum_{\substack{(a+b)k+am+bn=d \\ k,m,n\ge1}}[(m+n)(a+b)+2(k-m)b]\right)\\
&+16\left(\sum_{\substack{(a+b)k+am=d \\ k\ge1,m\ge0}}bkm\right)
\end{align*}
where we have cancelled the second and third sums into the fourth, but with extra terms left over remaining from allowing $n=0$ in the sum over $16bkm$.

Now, note that we can replace the sum over $2(k-m)b$ with the sum over $(k-m)b+(k-n)a$, by symmetry. This yields simply
\begin{equation*}
-8\left(\sum_{\substack{am+bn=d \\ m,n\ge1}}\max(a,b)mn\right)+4d\left(\sum_{\substack{(a+b)k+am+bn=d \\ k,m,n\ge1}}1\right)+16\left(\sum_{\substack{(a+b)k+am=d \\ k\ge1,m\ge0}}bkm\right)
\end{equation*}
We will now combine these sums by changing the index conditions on the second and third terms also to be $am+bn=d$. We rewrite the condition in the second sum as $a(k+m)+b(k+n)=d$, and change coordinates to get:
\begin{equation*}
-8\left(\sum_{\substack{am+bn=d \\ m,n\ge1}}\max(a,b)mn\right)+4d\left(\sum_{\substack{am+bn=d \\ m,n\ge1}}\min(m,n)\right)+16\left(\sum_{\substack{(a+b)k+am=d \\ k\ge1,m\ge0}}bkm\right)
\end{equation*}
We rewrite the third sum as
\begin{align*}
16\left(\sum_{\substack{(a+b)k+am=d \\ k\ge1,m\ge0}}bkm\right)&=8\left(\sum_{\substack{a(k+m)+bk=d \\ k\ge1,m\ge0}}bkm\right)+8\left(\sum_{\substack{ak+b(k+n)=d \\ k\ge1,n\ge0}}akn\right)\\
&=8\sum_{\substack{am+bn=d \\ m,n\ge1}}\theta(\min(m,n))\min(m,n)|m-n|,
\end{align*}
where $\theta(m)=a$ and $\theta(n)=b$. On the other hand, combining back with the second sum, note that
\begin{equation*}
4d\min(m,n)+8\theta(\min(m,n))\min(m,n)|m-n|=8mn(a+b)-4d\min(m,n).
\end{equation*}
Therefore, putting everything together, we are left with
\begin{align*}
&\quad\sum_{\substack{am+bn=d \\ m,n\ge1}}[-8\max(a,b)mn+8mn(a+b)-4d\min(m,n)]\\
&=\sum_{\substack{am+bn=d \\ m,n\ge1}}[8mn\min(a,b)-4(am+bn)\min(m,n)]\\
&=0,
\end{align*}
by Lemma \ref{david_lemma}. The proof is complete.
\end{proof}

%
%
%
%

\end{document}